\documentclass[british, a4paper,11pt]{article}

\usepackage[utf8]{inputenc}
\usepackage[british]{babel}

\usepackage[mono=false]{libertine}
\usepackage[T1]{fontenc}

\usepackage{makecell}
\usepackage{amsthm}
\usepackage{amssymb,bbm}
\usepackage[cmintegrals,libertine]{newtxmath}
\usepackage[cal=dutchcal, calscaled=.95, scr=boondoxo]{mathalfa}
\useosf
\let\mathcal\mathbcal

\usepackage{microtype}

\usepackage{array,booktabs}
\newcolumntype{x}[1]{%
>{\centering\hspace{0pt}}p{#1}}%

\usepackage{multirow}

\usepackage[explicit]{titlesec}

\titleformat{\section}[block]{\normalfont\large\bfseries\boldmath\centering}{\raggedright\makebox[1em][l]{\thesection.}}{.25em}{#1}
\titleformat{\subsection}[block]{\normalfont\bfseries\boldmath\centering}{\raggedright\makebox[1em][l]{\thesubsection.}}{1em}{#1}
\titleformat{\subsubsection}[runin]{\normalfont\bfseries\boldmath}{\raggedright\makebox[1em][l]{\thesubsubsection.}}{1.5em}{#1.~\hbox{---}}

\renewenvironment{abstract}{%
\begin{center}
\begin{minipage}{.9\textwidth}\linespread{1.05}\selectfont\small
\makebox[5em][l]{\bfseries\abstractname.~\hbox{---}}
\normalfont}
{\par\vspace{1em}
\end{minipage}
\end{center}
}

\usepackage[a4paper,vmargin={2cm,2cm},hmargin={2.25cm,2.25cm}]{geometry}
\selectfont

\usepackage{tikz}

\usepackage[font=sf]{caption}
\usepackage{hyperref}
\hypersetup{
	colorlinks=true,
		citecolor=blue!60!black,
		linkcolor=red!60!black,
		urlcolor=green!40!black,
		filecolor=yellow!50!black,
	breaklinks=true,
	pdfpagemode=UseNone,
	bookmarksopen=false,
}

\usepackage{enumitem}

\newtheorem{thm}{\bfseries \upshape Theorem}[section]
\newtheorem{lem}[thm]{Lemma}
\newtheorem{prop}[thm]{Proposition}

\newtheorem{ex}[thm]{Example}

\theoremstyle{definition}
\newtheorem{rem}[thm]{Remark}

\newlist{thmlist}{enumerate}{1}
\setlist[thmlist]{label=(\roman{thmlisti}), ref=\thethm(\roman{thmlisti}),noitemsep}

\renewcommand{\P}{\mathbb{P}}
\newcommand{\E}{\mathbb{E}}
\newcommand\Es[1]{\mathbb{E}\left[#1\right]}
\renewcommand\Pr[1]{\mathbb{P}\left(#1\right)}
\newcommand{\Var}{\mathrm{Var}}
\newcommand{\R}{\mathbb{R}}
\newcommand{\Z}{\mathbb{Z}}
\newcommand{\N}{\mathbb{N}}
\newcommand{\D}{\mathbb{D}}

\newcommand{\ind}[1]{\mathbbm{1}_{\{#1\}}}

\def\Vcd{\mathcal{V}^{\mathsf{d}}}

\newcommand{\map}{M}
\newcommand{\Map}{\mathbf{M}}

\newcommand{\Tree}{\mathbf{T}}

\newcommand{\dgr}{d_{\mathrm{gr}}}
\newcommand{\pgr}{p_{\mathrm{unif}}}

\newcommand{\Bmap}{\mathscr{M}}
\newcommand{\dBmap}{\mathscr{D}}
\newcommand{\pBmap}{\mathscr{m}}
\newcommand{\CRT}{\mathscr{T}}
\newcommand{\dCRT}{\mathscr{d}}
\newcommand{\pCRT}{\mathscr{p}}

\newcommand{\exc}{\mathrm{exc}}
\newcommand{\br}{\mathrm{br}}
\newcommand{\scale}{v}
\newcommand{\poids}{p}

\newcommand{\ic}{{i\mkern1mu}}
\newcommand{\e}{\operatorname{e}} 
\newcommand{\q}{\boldsymbol{q}}
\renewcommand{\d}{\mathrm{d}}

\renewcommand{\ge}{\geqslant}
\renewcommand{\le}{\leqslant}
\renewcommand{\geq}{\geqslant}
\renewcommand{\leq}{\leqslant}

\newcommand{\cv}[1][n]{\enskip\mathop{\longrightarrow}^{}_{#1 \to \infty}\enskip}
\newcommand{\cvloi}[1][n]{\enskip\mathop{\longrightarrow}^{(d)}_{#1 \to \infty}\enskip}

\newcommand{\cvproba}[1][n]{\enskip\mathop{\longrightarrow}^{\P}_{#1 \to \infty}\enskip}
\newcommand{\eqloi}{\enskip\mathop{=}^{(d)}\enskip}
\newcommand{\equi}[1][n]{\enskip\mathop{\sim}^{}_{#1 \to \infty}\enskip}

\usepackage{mathtools}

\DeclarePairedDelimiter\floor{\lfloor}{\rfloor}

\let\originalleft\left
\let\originalright\right
\renewcommand{\left}{\mathopen{}\mathclose\bgroup\originalleft}
\renewcommand{\right}{\aftergroup\egroup\originalright}

\DeclareSymbolFont{extraup}{U}{zavm}{m}{n}
\DeclareMathSymbol{\vardspade}{\mathalpha}{extraup}{81}
\DeclareMathSymbol{\varheart}{\mathalpha}{extraup}{86}
\DeclareMathSymbol{\vardiamond}{\mathalpha}{extraup}{87}
\DeclareMathSymbol{\varclub}{\mathalpha}{extraup}{84}

\makeatletter
\renewcommand*{\@fnsymbol}[1]{\ensuremath{\ifcase#1\or  \vardspade \or \varheart \or \vardiamond\or \varclub \or
   \mathsection\or \mathparagraph\or \|\or **\or \dagger\dagger   \or \ddagger\ddagger \else\@ctrerr\fi}}
\makeatother

\author{
	Igor \textsc{Kortchemski}\thanks{CNRS \& CMAP, \'{E}cole polytechnique.\hfill  \href{mailto:igor.kortchemski@math.cnrs.fr}{\texttt{igor.kortchemski@math.cnrs.fr}}} 
\qquad\&\qquad
	Cyril \textsc{Marzouk}\thanks{CMAP, \'{E}cole polytechnique.\hfill  \href{mailto:cyril.marzouk@polytechnique.edu}{\texttt{cyril.marzouk@polytechnique.edu}}}
}

\title{Large deviation Local Limit Theorems\\and\\limits of biconditioned Trees and Maps}

\begin{document}

\maketitle

\begin{abstract}
We first establish new local limit estimates for the probability that a nondecreasing integer-valued random walk lies at time $n$ at an arbitrary value, encompassing in particular large deviation regimes. This enables us to derive scaling limits of such random walks conditioned by their terminal value at time $n$ in various regimes. We believe both to be of independent interest. 
We then apply these results to obtain invariance principles for the \L ukasiewicz path of Bienaymé--Galton--Watson trees conditioned on having a fixed number of leaves and of vertices at the same time, which constitutes a first step towards understanding their large scale geometry. 
We finally deduce from this scaling limit theorems for random bipartite planar maps under a new conditioning by fixing their number of vertices, edges, and faces at the same time. In the particular case of the uniform distribution, our results confirm a prediction of Fusy \& Guitter on the growth of the typical distances and show furthermore that in all regimes, the scaling limit is the celebrated Brownian map.
\end{abstract}

\section{Introduction}
The main purpose of this paper is to  obtain new large deviation local limit theorems for random walks. As applications, we obtain scaling limits for random walks under a `bridge-type' conditioning in  large deviation regimes,  limit theorems for large random Bienaymé--Galton--Watson trees conditioned on having a fixed number of leaves and of vertices at the same time, and we also establish continuum limits of large random Boltzmann planar maps under a new conditioning by fixing their total number of vertices, edges, and faces at the same time.
Let us first present our results concerning random walks before briefly mentioning applications concerning random trees as well as random maps.

\subsection{Local limit theorems and bridge-type conditioning for random walks}
\label{sec:intro_LLT}

Consider a random walk, say $S_{n} = \xi_{1} + \dots + \xi_{n}$ for every $n \ge 1$, where $(\xi_{i})_{i\ge 1}$ is a sequence of i.i.d.~copies of a random variable $\xi$. 
The question of estimating quantities of the form $\P(S_{n} \in (x,x+T])$ for $T \in (0,\infty]$ is fundamental in the study of random walks with numerous applications (e.g.~in statistics, risk theory, statistical mechanics, queueing theory, etc.), and has received considerable interest, in particular in large deviation regimes. When the tail of the jump distribution vanishes exponentially fast (which is roughly speaking the so-called Cramér condition), seminal results go back to Cramér~\cite{Cra38} with  extensions by  Bahadur \& Ranga Rao~\cite{BR60} and Petrov~\cite{Pet65}; see also  H\"oglund~\cite{Hog79}. The class of subexponential random walks (for which the Cramér condition does not hold), which bears close connections with the so-called `one-big jump prinicple', has also received considerable interest, see the important paper by Denisov, Dieker \& Schneer~\cite{DDS08} and references therein. Quite some effort has also been devoted to the much more delicate estimation of `local' probabilities (corresponding to $T<\infty$), and whole books have been devoted to the subject; see in particular Borovkov \& Borovkov~\cite{BB08} for an exhaustive account.

In this work, given a sequence $x_{n} \to \infty$ as $n \to \infty$, we are interested in finding an asymptotic equivalent of the probability $\P(S_{n} = x_{n})$, as well as bounding  $\P(S_{n} = x_{n}+ k)$ uniformly in $k$.  This may be viewed as a `local' analogue of results of Jain \& Pruitt~\cite{JP87} who obtained estimates for $\P(S_{n} \leq x_{n})$.

As a first application, these estimates will be the key ingredient in Sec.~\ref{sec:marches} in establishing scaling limit theorems for the associated bridges, namely the trajectory $(S_{0},S_{1}, \dots,S_{n})$ under the so-called `bridge' conditioning $\P(\, \cdot \mid S_{n}=x_{n})$, extending a result of Liggett~\cite{Lig68}, who roughly speaking focused on bridge conditioned random walks in the domain of attraction of a stable law where the value of the endpoint is of the same order as the fluctuations of the random walk.
In turn, this will be useful for applications to random trees and maps in Sec.~\ref{sec:arbres} and~\ref{sec:cartes}.
Let us mention that scaling limits of random walk bridges conditioned to stay positive (see~\cite{CC13}) appear in connection with statistical physics and in particular with polymer models.

In view of our applications, we shall restrict ourselves to the case where $\xi$ is supported by $\Z_{\ge0}$, and, in order to avoid periodicity issues, that its support is not included in a sublattice of $\Z$, i.e.~the largest $h>0$ such that there exists $a \in \R$ such that $ \{k \geq 0: \P(\xi=k)>0\} \subset a+h\Z$ is $h=1$.  Before presenting the different regimes, let us say some words on the main method. It is a classical change of probability (sometimes called the Cram\'er or Esscher transform) which modifies the distribution of the increments in order to tune the drift of the random walk. Precisely, let
\[G(s)=\sum_{k \ge 0} s^{k} \P(\xi=k)\]
denote the generating function of $\xi$, and let $\rho \ge 1$ denote its radius of convergence. For every $b \in (0, \rho)$, let $\xi^{(b)}$ have the law with generating function $G(bs)/G(b)$ and observe that $\E[\xi^{(b)}] = b G'(b)/G(b)$. It is a simple matter to check that this quantity increases with $b$, and further, when $b \to 0$, it converges to $i_{\xi} \coloneqq \inf\{k \ge 0 : \P(\xi=k) > 0\}$, whereas when $b \to \rho$ the limit is finite if and only if $G'(\rho) < \infty$.
In particular, when $ i_{\xi}<x_{n}/n < \rho G'(\rho)/G(\rho)$,  one may define $b_{n}$ such that $b_{n} G'(b_{n})/G(b_{n}) = x_{n}/n$. It follows that
\begin{equation}\label{eq:tilt}
\Pr{S_{n}=x_{n}+k}= \frac{G(b_{n})^{n}}{b_{n}^{x_{n}+k}} \cdot \Pr{S^{(b_{n})}_{n}=x_{n}+k},
\end{equation}
where $S^{(b_{n})}_{n}$ is the sum of $n$ independent copies of $\xi^{(b_{n})}$. Thus estimating $\P(S_{n}=x_{n}+k)$ boils down to estimating $\P(S^{(b_{n})}_{n}=x_{n}+k)$. 
The advantage is that the event $ \{S^{(b_{n})}_{n}=x_{n}\}$ is more `typical', however the drawback to working with $S^{(b_{n})}$ is that the step distribution then varies with $n$.

In the `Brownian regime', we shall obtain estimates of the following form:
\begin{equation}\label{eq:LTT_forme_generale}
\sup_{k \ge -x_{n}} \left|\scale_{n} \cdot \frac{b_{n}^{x_{n}+k}}{G(b_{n})^{n}} \cdot \Pr{S_{n} = x_{n}+k} -  \frac{1}{\sqrt{2\pi}}\exp\left(-\frac{k^{2}}{2\scale_{n}^{2}}\right)\right| \cv 0,
\end{equation}
for some appropriate scaling factor $\scale_{n}$.

A simple case which falls in this regime is when $x_{n}/n$ converges to some limit belonging to the interval $(i_{\xi}, \rho G'(\rho)/G(\rho))$. In this case, $b_{n}$ converges to a limit $b \in (0, \rho)$, and the generating function $G$ is very regular at this point, so one can readily  adapt classical proofs of the local limit theorem in the Gaussian regime to the walk $S^{(b_{n})}_{n}$ and use~\eqref{eq:tilt}.
This case is indeed known, see e.g.~Borovkov \& Borovkov~\cite[Theorem~6.1.5]{BB08} and is included in the next theorem just for completeness. However, to the best of our knowledge the two other cases are new. 
Note that no regularity assumptions are made on $G$ in the first two cases. For the last one, we recall from e.g.~\cite[Definition~VI]{FS09}) that $G$ is said to be \emph{$\Delta$-analytic} when there exist a radius $R>\rho$ and an angle $\phi \in (0,\pi/2)$ such that $G$ is analytic on the domain $ \{z : |z|<R, z \neq \rho, |\arg(z-1)|>\phi\}$. 

\begin{thm}\label{thm:LTT_forme_generale_hors_stable}
Let $b_{n}$ be defined by $b_{n} G'(b_{n})/G(b_{n}) = x_{n}/n$.
The estimate~\eqref{eq:LTT_forme_generale} holds in each of the following three regimes:
\begin{thmlist}
\item\label{thm:LTT_bulk}
\emph{The bulk regime:} $x_{n}/n$ converges to some limit belonging to $(i_{\xi}, \rho G'(\rho)/G(\rho))$, and $\scale_{n}^{2} / n$ equals the variance of $\xi^{(b_{n})}$, explicitly given by
\[\frac{\scale_{n}^{2}}{n} = \frac{b_{n}^{2} G^{(2)}(b_{n}) + b_{n} G'(b_{n})}{G(b_{n})} - \left(\frac{b_{n} G'(b_{n})}{G(b_{n})}\right)^{2}.\]

\item\label{thm:LTT_small_endpoint}
\emph{The small endpoint regime}: $x_{n}/n \to 0$ and $\Pr{\xi=0}>0$, $\Pr{\xi=1}>0$, and finally $\scale_{n}^{2} = x_{n}$.

\item\label{thm:LTT_large_endpoint}
\emph{The large endpoint regime}: $x_{n}/n \to \infty$, $G$ is $\Delta$-analytic, and there exist $c, \alpha >0$ such that $G(\rho-z) \sim c z^{-\alpha}$ as $z \to 0$ with $\mathrm{Re}(z)>0$; finally $\scale_{n}^{2} = x_{n}^{2}/(\alpha n)$.
\end{thmlist}
\end{thm}

In the second case $x_{n}/n \to 0$, it is clear that the assumption $\P(\xi=0) > 0$ is necessary as otherwise $S_{n}$ is larger than or equal to $n$. On the other hand it seems that when $\P(\xi=1) = 0$, the model is close to be $i$-aperiodic in the limit, where $i = \min\{j \ge 1 : \P(\xi=j) \ne 0\}$, so one should restrict to values of $k \in i \N$ in order to have a local limit theorem.

The  regime (iii) (large endpoint regime) naturally appears in the context  of random planar maps: we will indeed see that the model of uniform biconditioned bipartite Boltzmann planar maps discussed below is closely related to the distribution given by
$\P(\xi = k) = 2 \left(\frac{3}{16}\right)^{k+1}\binom{2k+1}{k}$ for every $k \geq 0$,
with generating function 
$G(z) = z^{-1} ((1-3z/4)^{-1/2} - 1)$ for $z \in (0, 4/3)$,
which satisfies the assumptions of Theorem~\ref{thm:LTT_large_endpoint} with $\alpha = 1/2$.

As a first application, in Sec.~\ref{sec:marches} we will use these estimates to get a functional convergence, namely that in the regimes of Theorem~\ref{thm:LTT_forme_generale_hors_stable}, the rescaled paths
\[\left( \frac{1}{\scale_{n}} \left(S_{\lfloor n t \rfloor} - x_{n} t\right)  ; 0 \le  t \le  1\right)_{n\ge 1}\]
under $\P(\, \cdot \mid S_{n}=x_{n})$ converge in distribution towards a standard Brownian bridge.

To keep this Introduction as short as possible, we have not mentioned here the case when $x_{n}/n \to \rho G'(\rho)/G(\rho)$, with the latter being finite. In this case, when $\xi$ also belongs to the domain of attraction of a stable law with index $\alpha \in (1,2]$, several different behaviours appear depending on the deviation of $x_{n}$ from $n \rho G'(\rho)/G(\rho)$. We defer the precise statement to Sec.~\ref{sec:LLT_stable}; we establish in particular in Theorem~\ref{thm:lls1} the estimate~\eqref{eq:LTT_forme_generale} when $x_{n} - n \rho G'(\rho)/G(\rho)$ divided by the usual scaling factor of the random walk, of order $n^{1/\alpha}$, tends to $-\infty$.

\subsection{Applications to biconditioned random trees}
Recall the model of Bienaym\'e--Galton--Watson plane trees, representing the genealogy of a family in which the individuals reproduce independently of each other according to a given offspring distribution. A lot of effort in the last decades has been put on understanding the scaling limits of such trees conditioned to be large. A classical conditioning involves the total number of vertices, for which the most famous results are certainly due to Aldous~\cite{Ald93} and Duquesne~\cite{Duq03}.  
Often motivated by applications to random combinatorial models, many other types of conditionings have also been considered, involving for instance the height~\cite{LG10} or the total number of leaves and more generally the total number of vertices with fixed outdegrees~\cite{Kor12,Riz15}.

One of our applications concerns random plane  trees conditioned by both their total number of vertices and leaves, say $n$ and $K_{n}$ respectively. This conditioning has been first considered by Labarbe \& Marckert~\cite{LM07} in the case of the uniform distribution (which amounts to taking a geometrical offspring distribution),
by relying on explicit counting formulas, in part motivated by applications to parallelogram polyominoes. In this work, we are interested in general biconditioned Bienaymé--Galton--Watson trees.
To the best of our knowledge, this is considered for the first time. Using new techniques, we establish scaling limits of their \L ukasiewicz path. This is the first step towards the understanding of the geometry of such trees;  see  the open question at the end of this introduction.

Indeed, this path has the same law as a random walk with increments in $\{-1, 0, 1, 2, \dots\}$, conditioned to first hit $-1$ at time $n$ \emph{and} to make $K_{n}$ negative steps in total. By applying the cycle lemma (also known as Vervaat transform) one can remove the positivity constraint and end up studying a random walk conditioned to lie at position $-1$ at time $n$ and to have made $K_{n}$ negative steps. We then split the negative and nonnegative contributions; for the reverse operation one roughly speaking `shuffles' these contributions. Indeed, the positions of the negative steps are just uniformly chosen amongst all the possibilities, which is tractable and is a well-studied framework, whilst the path obtained by removing these jumps is now a nondecreasing random walk, only conditioned to lie at the value $K_{n}-1$ at time $n-K_{n}$. This is essentially the new key idea that allows to study general offspring distributions and to transfer results from a `conditioned' setting to a `biconditioned' setting, see Lemma~\ref{lem:separation_sauts} for a precise statement. We are then in position to apply our preceding results on bridge-conditioned random walks. In each of the three regimes presented in the previous subsection, we obtain that the \L ukasiewicz path of the trees converges in distribution towards the Brownian excursion.
When the offspring distribution belongs to the domain of attraction of a stable law with index in $(1,2)$, under a fine tuning of the proportion of leaves, we obtain instead in the limit the excursion of a stable L\'evy process with no negative jump and with a (positive or negative) drift. We refer to Sec.~\ref{sec:Luka} for precise statements.

\subsection{Scaling limits of  biconditioned random planar maps}  A (planar) map is the embedding of a finite, connected multigraph in the two-dimensional sphere, viewed up to orientation-preserving homeomorphisms; we shall always root these graphs by distinguishing one oriented edge. The embedding allows to define the \emph{faces} of the map which are the connected components of the complement of the graph on the sphere, and the \emph{degree} of a face is the number of edges incident to it, counted with multiplicity: an edge incident on both sides to the same face contributes twice to its degree. See Figure~\ref{fig:bijection_arbre_carte} right for an example of a map with six faces.

Le~Gall~\cite{LG13} and Miermont~\cite{Mie13} simultaneously closed a series of works by proving that if one samples a quadrangulation (i.e.~all faces have degree $4$) with $n$ faces uniformly at random, say $m_{n}$, then its set of vertices $V(m_{n})$ equipped with the graph distance $\dgr$ rescaled by a factor $n^{-1/4}$ converges in distribution in the Gromov--Hausdorff topology to a random compact metric space called the \emph{Brownian map}. The latter is a random metric space, almost surely homeomorphic to the $2$-sphere~\cite{LGP08,Mie08} and with Hausdorff dimension $4$~\cite{LG07}. 
In this paper we shall also consider the uniform probability measure $\pgr$ and the Gromov--Hausdorff--Prokhorov topology (see e.g.~\cite[Sec.~6]{Mie09} for details on this topology).
Following this result, many other discrete models of `large' random maps have been shown to converge towards the Brownian map. Most of them focus on the technically simpler case of \emph{bipartite} maps, in which all faces have even degree, with two notable exceptions~\cite{ABA19, BJM14}. 
We will as well restrict to the bipartite case.

In Sec.~\ref{sec:cartes}, motivated by a prediction of Fusy \& Guitter~\cite{FG14},  we will be interested in random planar maps with random face degrees, conditioned to have a large fixed number of vertices, edges, and faces at the same time; by Euler's formula, there are actually only two degrees of freedom, hence the name \emph{biconditioned} planar maps. More precisely, Let $(K_{n})_{n\ge 1}$ be a sequence of integers and, in order to discard degenerate cases, we shall always assume that both $K_{n}$ and $n-K_{n}$ tend to infinity. Let us denote by $\Map_{n,K_{n}}$ the set of all rooted bipartite planar maps with $n-1$ edges and $K_{n}+1$ vertices; by Euler's formula, all maps in $\Map_{n,K_{n}}$ have $n-K_{n}$ faces. The combination of the papers~\cite{BDG04} and~\cite{JS15} relate such maps to decorated (`labelled') trees with $n$ vertices and $K_{n}$ leaves and allows to study random maps related to Bienaym\'e--Galton--Watson trees that we discussed.

Let us defer our general statements to Sec~\ref{sec:results} and focus here on the particular case of the uniform distribution.
This framework is already new and exhibits an intriguing feature: Once appropriately scaled, such random maps always converge to the Brownian map.
Specifically, for $0 < x <1$, set
\[S(x)=\frac{(1-x) (3+x+\sqrt{(1-x)(9-x)})}{12 x}.\]
Note that $S$ is continuous, decreasing, and $S(x) \sim 1/(2x)$ and $S(1-x) \sim x/3$ as $x \to 0$.

\begin{thm}\label{thm:cartes_unif}
Let $(K_{n})_{n \ge 1}$ be integers such that $K_{n} \to \infty$ and $n-K_{n} \to \infty$ and let $M_{n, K_{n}}$ be a bipartite map with $n-1$ edges and $K_{n}+1$ vertices sampled uniformly at random.  Then the convergence
\[\left(V(\map_{n,K_n}),   \left( S \left(  \frac{K_{n}}{n}\right) \frac{9}{4n} \right)^{1/4} \dgr, \pgr\right)  \cvloi (\Bmap, \dBmap, \pBmap)\]
holds in distribution for the Gromov--Hausdorff--Prokhorov topology, where $(\Bmap, \dBmap, \pBmap)$ is the standard Brownian map.
\end{thm}

\begin{rem}\label{rem:Fusy_Guitter}
\begin{enumerate}
\item It is interesting to notice in Theorem~\ref{thm:cartes_unif} that, whilst the scaling factor depends on $K_{n}$ and $n$, the scaling limit is always the Brownian map. This theorem is in spirit similar to a result of Labarbe \& Marckert~\cite{LM07}, which shows that if $s(x)=2(1-x)/x$, then a uniformly random plane tree with $n$ edges and $K_{n}$ leaves, normalised by $(n  s(K_{n}/n))^{-1/2}$ converges to Aldous's Brownian CRT coded by the standard Brownian excursion (under the same non-degeneracy assumption $K_{n} \to \infty$ and $n-K_{n} \to  \infty$). 
Our work has two main differences:  first, we do not restrict to the uniform distribution; as a matter of fact, even for uniform maps, non-uniform plane trees appear because of bias due to the labels.
Second, we only focus on the \L ukasiewicz path of trees, and not on their contour function.

\item Theorem~\ref{thm:cartes_unif} confirms a prediction of Fusy \& Guitter~\cite[Sec.~6]{FG14} that typical distances in uniform random maps with $n$ edges and $n^{b}$ faces are of order $n^{(2-b)/4}$, whereas distances in uniform random maps with $n$ edges and $n^{c}$ vertices are of order $n^{c/4}$; furthermore our result shows that  in both cases the limit is the Brownian map.

\item Finally, Abraham~\cite{Abr16} proved that bipartite maps with $n$ edges sampled uniformly random converge in distribution to the Brownian map once rescaled by $(2n)^{-1/4}$. Since a map with $n$ edges typically has about $2n/3$ vertices (see e.g.~\cite[Sec.~6]{Abr16}), and $S(2/3) = 2/9$, then this is consistent with our result.
Note that for general, possibly non-bipartite, uniformly random maps with $n$ edges, the scaling constant is different, see~\cite{BJM14}.
\end{enumerate}
\end{rem}

When the trees have an offspring distribution in the domain of attraction of a stable law, we shall see that the corresponding biconditioned maps, under a fine tuning of the number of vertices, converge \emph{after extraction of a subsequence}. These new metric spaces, which we call $(\alpha, \lambda)$-\emph{skewed-stable maps}, are formed by a two-parameter family with $\alpha \in (1,2)$ and $\lambda\in\R$ and will be the main subject of investigation of the companion paper~\cite{dimensions}. They are close in spirit to the so-called stable maps of Le~Gall \& Miermont~\cite{LGM11} (which would be obtained for $\lambda=0$), and are constructed as there with the difference that there is an additional (positive or negative) drift to the (decorated) stable L\'evy process with no negative jump.  In~\cite{dimensions}, we actually consider a much broader class of spectrally positive L\'evy processes and we identify the Hausdorff, packing, and Minkowski dimensions of the associated metric spaces. In the particular case of the skewed-stable maps that appear here, as in~\cite{LGM11}, their dimensions are all equal to $2\alpha$ almost surely, for every value of $\lambda$. Finally,  we believe that for each $\alpha$ fixed, this family interpolates between the Brownian map and the Brownian \emph{tree} in the sense that, suitably rescaled, the skewed-stable maps converge in distribution towards the former as $\lambda\to-\infty$ and to the latter as $\lambda\to\infty$.

\subsection{Open questions}
This work leaves open several questions that we plan to investigate, let us briefly present two of them. 
First, from the point of view of scaling limits, several regimes are yet not covered, especially the stable regime with index in $(0, 1]$. When the index equals $1$, i.e.~in the Cauchy regime, trees~\cite{KR19} and maps~\cite{Mar19} conditioned only by their number of edges fall into a condensation regime but we believe that, as here, if one forces them to have less leaves and vertices respectively than usual, one can get different limits. Our techniques still apply and one only needs appropriate local limit estimates as those in Sec.~\ref{sec:LLT}.

Another direction of research focuses on biconditioned Bienaym\'e--Galton--Watson trees.
It is natural to ask for analogous results to our main results on maps for such trees. Only the uniform distribution has been considered so far, by Labarbe \& Marckert~\cite{LM07}; the point is that the analysis of maps only relies on the \L ukasiewicz path of the trees, whereas the convergence of the trees requires to study their height process which is much more involved for general Bienaym\'e--Galton--Watson trees. Nonetheless our results on the \L ukasiewicz paths (Theorems~\ref{thm:CVluka} and~\ref{thm:CVluka_distorted}) are the first step towards understanding the behaviour of the height processes and we plan to study more these trees in the near future.

\subsection{Plan of the paper}
This paper is organised as follows. First, in Sec.~\ref{sec:LLT} we focus on the local limit estimates: we prove Theorem~\ref{thm:LTT_forme_generale_hors_stable} and state and prove similar estimates in the stable regimes. 
Then in Sec.~\ref{sec:marches} we apply these results to obtain invariance principles for nondecreasing random walk bridges.
In Sec.~\ref{sec:arbres} we discuss application to biconditioned random trees and invariance principles for their \L ukasiewicz paths in Sec.~\ref{sec:Luka} by relating them with nondecreasing bridges as discussed above.
Then in Sec.~\ref{sec:cartes} we first present the general model of Boltzmann planar maps, and then precisely state our results pertaining to biconditioned maps in Sec.~\ref{sec:results}, whose proof readily follows from the results on trees. Finally we defer to Appendix~\ref{sec:preuves_LLT} some technical proofs on random bridges needed for the random maps.

\paragraph{Acknowledgement.} We warmly thank Lo\"ic Richier for stimulating discussions, as well as Mickaël Maazoun for discussions at early stages of this work.

\section{Local Limit Theorems}
\label{sec:LLT}

In this section, we first prove Theorem~\ref{thm:LTT_forme_generale_hors_stable} (ii) and (iii) (recall that (i) is already known, see~\cite[Theorem~6.1.5]{BB08}) and then state and prove an analogous result for walks attracted to a stable law as alluded to in the introduction.
Recall the framework of Sec.~\ref{sec:intro_LLT}:  $\xi$ is a random variable supported by $\Z_{\ge0}$, and, in order to avoid periodicity issues, we  assume throughout that its support is not included in a sublattice of $\Z$. We consider a random walk, say $S_{n} = \xi_{1} + \dots + \xi_{n}$ for every $n \ge 1$, where $(\xi_{i})_{i\ge 1}$ is a sequence of i.i.d. copies of $\xi$.
Recall that we denote by $G(s) = \sum_{k \ge 0} s^{k} \P(\xi=k)$ the generating function of $\xi$, and we let $\rho \ge 1$ denote its radius of convergence. For every $b \in (0, \rho)$, let $\xi^{(b)}$ have the law with generating function $G^{(b)}(s) = G(bs)/G(b)$, so $\E[\xi^{(b)}] = b G'(b)/G(b)$ and let $S^{(b)}$ denote the associated random walk.

Recall that we aim at estimating probabilities of the form $\P(S_{n} = x_{n}+ k)$ for a given a sequence $x_{n} \to \infty$; by~\eqref{eq:tilt} we may replace $S_{n}$ by $S^{(b_{n})}$ with $b_{n}$ such that $\E[\xi^{(b_{n})}] = b_{n} G'(b_{n})/G(b_{n}) = x_{n}/n$ (provided it exists) so that the event of taking value $x_{n}$ at time $n$ is more `typical'.
The following expressions of the moments of $\xi^{(b_{n})}$ will be useful in this section and in the next ones:
\begin{equation}\label{eq:moments}
\begin{split}
&\Es{\xi^{(b_{n})}} = \frac{b_{n} G'(b_{n})}{G(b_{n})} =  \frac{x_{n}}{n},
\qquad\qquad
\Es{(\xi^{(b_{n})})^{2}} = \frac{b_{n}^{2} G^{(2)}(b_{n}) + b_{n} G'(b_{n})}{G(b_{n})},
\\
&\Es{(\xi^{(b_{n})})^{3}} = \frac{b_{n}^{3} G^{(3)}(b_{n}) + 3 b_{n}^{2} G^{(2)}(b_{n}) - 2 b_{n} G'(b_{n})}{G(b_{n})},
\\
&\Es{(\xi^{(b_{n})})^{4}} = \frac{b_{n}^{4} G^{(4)}(b_{n}) + 6 b_{n}^{3} G^{(3)}(b_{n}) - 11 b_{n}^{2} G^{(2)}(b_{n}) + 6 b_{n} G'(b_{n})}{G(b_{n})}.
\end{split}
\end{equation}

\subsection{The small endpoint regime}
\label{sec:small_endpoint}

In this subsection we consider the case $x_{n}/n \to 0$ and we prove Theorem~\ref{thm:LTT_small_endpoint}: as soon as $\Pr{\xi=0}>0$, $\Pr{\xi=1}>0$, if $b_{n}$ is such that $b_{n} G'(b_{n})/G(b_{n}) = x_{n}/n$ then we have
\[\sup_{k \ge -x_{n}} \left| \sqrt{x_{n}} \cdot \frac{b_{n}^{x_{n}+k}}{G(b_{n})^{n}} \cdot \Pr{{S}_{n} = x_{n}+k} -   \frac{1}{\sqrt{2\pi}}\exp\left(-\frac{k^{2} }{2 x_{n}}\right)\right| \cv 0.\]

\begin{proof}[Proof of Theorem~\ref{thm:LTT_small_endpoint}]
Using~\eqref{eq:tilt}, let us rewrite the claim as
\begin{equation}\label{eq:mq}
\sup_{k  \in \Z} \left|  \sqrt{x_{n}} \cdot \Pr{{S}_{n}^{(b_{n})} = x_{n}+k} -   \frac{1}{\sqrt{2\pi}}\exp\left(-\frac{k^{2} }{2 x_{n}}\right)\right| \cv 0.
\end{equation}
One could mimic the proof of the standard local limit theorem, but we prefer a shorter and less technical approach based on~\cite[Theorem~1.2]{DM95}. 
We first need to check that a central limit theorem holds. Recall (see e.g.~\cite[Lemma~3.3.7]{Dur10}) that if $Y$ is a random variable with a finite third moment, then
\[\left|\Es{\e^{\ic Y}} - \left(1 + \ic \E[Y] - \frac{1}{2} \E[Y^{2}]\right)\right| \le \frac{1}{6} \E[|Y|^{3}].\]
To estimate the moments of $S_{n}^{(b_{n})}$, 
set $p(0)=\Pr{\xi=0}$ and $p(1)=\Pr{\xi=1}$ and observe that since we assume $\poids(1) > 0$, then, for every $i \ge 2$, as $x \to 0$,
\begin{equation}\label{eq:derivees_superieures_G}
x G'(x) \sim \poids(1) x
\qquad\text{and}\qquad
x^{i} G^{(i)}(x) = O(x^{i}).
\end{equation}
Then by definition of $b_{n}$,
\[b_{n} \equi \frac{\poids(0)}{\poids(1)} \frac{x_{n}}{n},\]
and it follows from~\eqref{eq:moments} that 
\[\E\big[(\xi^{(b_n)})^{2}\big] = \frac{b_{n}^{2} G^{(2)}(b_{n}) + b_{n} G'(b_{n})}{G(b_{n})}
= \frac{x_{n}}{n} \left(\frac{b_{n}^{2} G^{(2)}(b_{n})}{b_{n} G'(b_{n})} + 1\right)
\equi \frac{x_{n}}{n},\]
and
\[\E\big[(\xi^{(b_n)})^{3}\big] 
= O\left(\frac{x_{n}}{n}\right).\]
Hence, for any $u \in \R$,
\[\Es{\e^{\ic u \xi^{(b_n)}}}  = 1 + \ic \frac{x_{n} u}{n} - \frac{x_{n} u^{2}}{2n} + o\left(\frac{x_{n}}{n}\right),\]
which implies that\[\Es{\exp\left(\ic u \frac{{S}^{(b_{n})}_{n} - x_{n}}{\sqrt{x_{n}}}\right)}
\cv \exp\left(- \frac{u^{2} t}{2}\right).\]
Next, in the notation from~\cite{DM95}, for every $n \ge 1$, we have $Q_{n} = n \sum_{k \in \Z} \min\{\P(\xi^{(b_n)} = k), \P(\xi^{(b_n)} = k+1)\} \ge n \min\{\P(\xi^{(b_n)} = 0), \P(\xi^{(b_n)} = 1)\}$. Observe that
\[\P(\xi^{(b_n)} = 0) = \frac{\poids(0)}{G(b_{n})} \cv 1,\]
and
\[\P(\xi^{(b_n)} = 1) = \frac{b_{n} \poids(1)}{G(b_{n})} \equi \frac{\poids(1)}{\poids(0)} b_{n} \equi \frac{x_{n}}{n}.\]
Therefore $\limsup_{n} x_{n} / Q_{{n}} < \infty$ and Theorem~1.2 in~\cite{DM95}  implies~\eqref{eq:mq}. This completes the proof.
\end{proof}

\subsection{The large endpoint regime}
\label{sec:large_endpoint}

We next turn to Theorem~\ref{thm:LTT_large_endpoint}; we henceforth assume that $G$ is $\Delta$-analytic and that there exist $c, \rho, \alpha >0$ such that
$G(\rho-z) \sim c z^{-\alpha}$ as $z \to 0$ with $\mathrm{Re}(z)>0$.
We let $x_{n}$ be such that $x_{n}/n \to \infty$ and let $b_{n}$ be such that $b_{n} G'(b_{n})/G(b_{n}) = x_{n}/n$ and we show that
\[\sup_{k \ge -x_{n}} \left|\frac{x_{n}}{\sqrt{\alpha n}} \cdot  \frac{b_{n}^{x_{n}+k}}{G(b_{n})^{n}}\Pr{{S}_{n} = x_{n}+k} -    \frac{1}{\sqrt{2\pi}}\exp\left(- \frac{\alpha n}{x_{n}^{2}}  k^{2}\right)\right| \cv 0.\]
We shall do so by adapting the classical proof of the local limit estimate: we first establish a central limit theorem and then use the inverse Fourier transform, cut the integrals and bound each piece appropriately.

Let us start with some estimates. First, by $\Delta$-analyticity, for every $k \ge 1$, if $G^{(k)}$ denotes the $k$'th derivative of $G$, then (see e.g.~\cite[Theorem VI.8]{FS09})
\begin{equation}\label{eq:Gderiv}
G^{(k)}(\rho-z)   \quad \mathop{\sim}_{\substack{z \to 0\\ \mathrm{Re}(z)>0}} \quad  \alpha (\alpha+1) \cdots (\alpha+k-1) \frac{c}{z^{\alpha+k}}.
\end{equation}
Writing  $b_{n} = \rho (1 - \varepsilon_{n})$, by definition of $b_{n}$ we have
\begin{equation}\label{eq:b_n_negative_index}
b_{n} = \frac{x_{n}}{n} \frac{G(\rho (1 - \varepsilon_{n}))}{G'(\rho (1 - \varepsilon_{n}))}
\sim \frac{x_{n}}{n} \frac{c (\rho\varepsilon_{n})^{-\alpha}}{\alpha c (\rho\varepsilon_{n})^{-(\alpha+1)}} =
\frac{x_{n}}{n} \frac{\rho\varepsilon_{n}}{\alpha}.
\end{equation}
Since $b_{n} \to \rho$, we conclude that
$\varepsilon_{n} \sim \alpha n / x_{n}$.

Set ${\sigma}_{n}^{2} = \Var(\xi^{(b_n)})$. Recalling from~\eqref{eq:moments} the expression of the first moments of $\xi^{(b_n)}$, we  infer that as $n \to \infty$,
\begin{equation}\label{eq:moments2-new}
\E\big[(\xi^{(b_n)})^{2}\big] 
\sim \frac{\alpha(\alpha+1)}{\varepsilon_{n}^{2}}
\sim \frac{\alpha+1}{\alpha} \frac{x_{n}^{2}}{n^{2}}
\qquad\text{so}\qquad
{\sigma}_{n}^{2} \sim \frac{1}{\alpha} \frac{x_{n}^{2}}{n^{2}}
,\end{equation}
as well as
\begin{equation}\label{eq:moments3-new}
\E[(\xi^{(b_n)})^{3}] 
\sim \frac{\alpha (\alpha+1) (\alpha+2)}{\varepsilon_{n}^{3}}
\sim \frac{(\alpha+1) (\alpha+2)}{\alpha^{2}} \frac{x_{n}^{3}}{n^{3}}
.\end{equation}
Observe that
\begin{equation}\label{eq:limite_moments23-new}
\frac{\E[(\xi^{(b_n)})^{2}]}{{\sigma}_{n}^{2}} \cv \frac{\alpha+1}{2}
\qquad\text{and}\qquad
\frac{\E[(\xi^{(b_n)})^{3}]}{{\sigma}_{n}^{3}} \cv \frac{(\alpha+1) (\alpha+2)}{\sqrt{\alpha}}
.\end{equation}

Finally, to simplify notation, let $\phi_{n}(t)=\E[\exp(\ic t \xi^{(b_n)})]$ denote the characteristic function of $\xi^{(b_n)}$.

\begin{lem}\label{lem:techn}
The following assertions hold.
\begin{enumerate}
\item For every $\eta'>0$, for every $\eta' \sigma_{n} \le  |s| \le  \pi {\sigma}_{n}$, for every $n$ sufficiently large, $|\phi_{n}(s/{\sigma}_{n})| \leq 1/\sigma_{n}$.
\item There exists $\eta'>0$ such that for every $\eta >0$, there exists $c_{\eta} \in (0,\eta^{\alpha})$  such that for every $n$ sufficiently large, for every $\eta  \le  |s| \le  \eta' {\sigma}_{n}$ we have $|\phi_{n}(s/{\sigma}_{n})| \le c_{\eta} |s|^{-\alpha}$.
\item There exist $\eta,c_{1}>0$ such that for every $0 \le  |s| \le  \eta$ and $n \ge  1$ we have $|\phi_{n}(s / {\sigma}_{n})| \le \exp\left(- c_{1} s^{2}\right)$.
\end{enumerate}
\end{lem}

\begin{proof}
Recall that for every $s \in \R$, we have $\phi_{n}(s) = G(b_{n} \e^{\ic s}) / G(b_{n})$ with $b_{n} = \rho (1 - \varepsilon_{n})$,  so by our assumption, $G(b_{n}) \sim c (\rho \varepsilon_{n})^{-\alpha}$.
Moreover, by combining~\eqref{eq:b_n_negative_index} and~\eqref{eq:moments2-new}, we obtain $\varepsilon_{n} \sim \sqrt{\alpha} / {\sigma}_{n}$, hence $G(b_{n}) \sim c (\rho \sqrt{\alpha})^{-\alpha} \sigma_{n}^{\alpha}$. Also, since $G$ is $\Delta$-analytic, the convergence $|G(b_{n} \e^{\ic t})| \to |G(\rho \e^{\ic t})|$ holds uniformly for $t$ in compact subsets of $(-\pi,\pi) \backslash \{0\}$, so there exists a constant $c_{0}>0$ such that 
for every $\eta'>0$, for every $\eta' {\sigma}_{n} \leq |s| \leq \pi \sigma_{n}$, and for $n$ sufficiently large, $|\phi_{n}(s/{\sigma}_{n})| \leq c_{0}/\sigma_{n}^{\alpha} \leq 1/\sigma_{n}$.

Let us turn to the second claim. Fix $\delta>0$. There exists $\eta'>0$ such that for any $z$ such that $|z| < 2\eta'$ and $\mathrm{Re}(z)>0$, it holds that $c(1+\delta)^{-1} |z|^{-\alpha} \le |G(\rho-z)| \le c(1+\delta) |z|^{-\alpha}$. Therefore, for every $|t| < \eta'$, for every $n$ large enough,
\[|\phi_{n}(t)|
= \left|\frac{G(b_{n} \e^{\ic t})}{G(b_{n})}\right|
\le \frac{c(1+\delta) |\rho - b_{n} \e^{\ic t}|^{-\alpha}}{c (1+\delta)^{-1} (\rho \varepsilon_{n})^{-\alpha}}
= (1+\delta)^{2} \left(\frac{|1 - (1 - \varepsilon_{n}) \e^{\ic t}|}{\varepsilon_{n}}\right)^{-\alpha}
.\]
Straightforward calculations show that if $\eta'$ is small enough, then for every $t \in [-\eta', \eta']$,
\[|1 - (1 - \varepsilon_{n}) \e^{\ic t}|^{2}
= 1 - 2 (1 - \varepsilon_{n}) \cos t + (1 - \varepsilon_{n})^{2}
= 2 (1-\cos t) (1 - \varepsilon_{n}) + \varepsilon_{n}^{2}
\ge \frac{t^{2}}{1+\delta} + \varepsilon_{n}^{2}
.\]
Moreover, combining~\eqref{eq:b_n_negative_index} and~\eqref{eq:moments2-new}, we obtain $\varepsilon_{n}^{2} \sim \alpha / {\sigma}_{n}^{2}$, therefore for $|t| < \eta'$, for every $n$ large enough,
\[|\phi_{n}(t)|
\le (1+\delta)^{2} \left(\frac{|1 - (1 - \varepsilon_{n}) \e^{\ic t}|^{2}}{\varepsilon_{n}^{2}}\right)^{-\alpha/2}
\le (1+\delta)^{2} \left(\frac{(t {\sigma}_{n})^{2}}{(1+\delta)^{2} \alpha^{2}} + 1\right)^{-\alpha/2}
.\]
We infer that for every $\eta \le |s| \le \eta' {\sigma}_{n}$,
\[|\phi_{n}(s/{\sigma}_{n})|
\le (1+\delta)^{2} \left(\frac{1}{(1+\delta)^{2} \alpha^{2}} + \frac{1}{\eta^{2}}\right)^{-\alpha/2} |s|^{-\alpha}
.\]
For $\delta = 0$, the right-hand side is smaller than $\eta^{\alpha} |s|^{-\alpha}$ so by continuity, it remains true for $\delta>0$ small enough and the first claim follows.

Let us finally turn to the third claim. As in the proof of Theorem~\ref{thm:LTT_small_endpoint} (see e.g.~\cite[Lemma~3.3.7]{Dur10}), we start  with writing
\begin{equation}\label{eq:approx_phi_moments-new}
\phi_{n}(t) = 1 + \ic \E[\xi^{(b_n)}] t - \E[(\xi^{(b_n)})^{2}] \frac{t^{2}}{2} + R^{n}_{1}(t)
\qquad\text{with}\qquad
|R^{n}_{1}(t)| \le \frac{|t|^{3}}{6} \E[|\xi^{(b_n)}|^{3}].
\end{equation}
Consequently, by writing $|\phi_{n}(t)|^{2} = \phi_{n}(t) \overline{\phi_{n}(t)}$, we infer that
\begin{equation}\label{eq:approx_phi_carre_moments-new}
\begin{split}
&\big|\phi_{n}(t)\big|^{2} = 1 - \Var(\xi^{(b_n)}) t^{2} + |R^{n}_{2}(t)|,
\\
\text{where}\qquad 
&|R^{n}_{2}(t)| = \left(\E[(\xi^{(b_n)})^{2}] \frac{t^{2}}{2}\right)^{2} + 2\left(1 - \E[(\xi^{(b_n)})^{2}] \frac{t^{2}}{2}\right) |R^{n}_{1}(t)| + |R^{n}_{1}(t)|^{2}
.\end{split}
\end{equation}
Combining~\eqref{eq:approx_phi_carre_moments-new} with~\eqref{eq:approx_phi_moments-new} and~\eqref{eq:limite_moments23-new} we see that $|R^{n}_{2}(t/{\sigma}_{n})| = O(|t|^{3})$ so there exist $\eta,c_{1}>0$ such that for every $0 \le  |t| \le  \eta$ and $n \ge  1$, we have $|R^{n}_{2}(t/{\sigma}_{n})| \le (1-2c) t^{2}$. Consequently, appealing to~\eqref{eq:approx_phi_carre_moments-new} again, for every $0 \le  |t| \le  \eta$ and $n \ge  1$,
\[\ln \big|\phi_{n}(t/{\sigma}_{n})\big| 
\le \frac{1}{2} \ln\left(1 - t^{2} + |R^{n}_{2}(t/{\sigma}_{n})|\right)
\le \frac{1}{2} \ln\left(1 - 2c_{1} t^{2}\right)
\le -c_{1} t^{2}
,\]
and the third claim follows.
\end{proof}

We are now in position to establish Theorem~\ref{thm:LTT_large_endpoint}.

\begin{proof}[Proof of Theorem~\ref{thm:LTT_large_endpoint}]
By~\eqref{eq:tilt}, our aim is to show that 
\[\sup_{k \in \Z} \left|\frac{x_{n}}{\sqrt{\alpha n}} \cdot \Pr{{S}_{n}^{(b_{n})} = x_{n}+k} -    \frac{1}{\sqrt{2\pi}}\exp\left(- \frac{\alpha n}{x_{n}^{2}}  k^{2}\right)\right| \cv 0.
\]
Recall the approximation~\eqref{eq:approx_phi_moments-new}, then by~\eqref{eq:moments2-new} and~\eqref{eq:moments3-new}, there is a constant $C >0$ such that $|R^{n}_{1}(t/(\sigma_{n} \sqrt{n}))| \le  C |t|^{3}/n^{3/2}$ for every $t \in \R$,  $n \ge  1$.  Since $ n  R^{n}_{1}(t/(\sigma_{n}\sqrt{n})) \to 0$  uniformly for $t$ in compact subsets of $\R$, it readily follows that
\begin{equation}\label{eq:cv1b}
\Es{\exp\left(\ic t \frac{{S}_{n}^{(b_{n})}-x_{n} }{\sigma_{n} \sqrt{n}}\right)}  \cv \exp\left(- \frac{t^{2}}{2}\right),
\end{equation}
uniformly for $t$ in compact subsets of $\R$.

We then follow the steps of the analytic proof of the standard local limit theorem for a sequence of independent identically distributed random variables. The main difficulty is that the distribution of $\xi^{(b_n)}$ depends on $n$. 
As in the preceding subsection, to simplify notation, set $f(t)= \e^{- t^{2}/2}$ for $t \in \R$, so by Fourier inversion, the Gaussian density reads
\[p(x)= \frac{1}{2\pi} \int_{-\infty}^{\infty} \e^{- \ic t x} f(t) {\d}t,
\qquad x \in \R.\]
Also, for $k \in \Z$,
\[ \Pr{{S}_{n}^{(b_{n})}=k} = \frac{1}{2\pi} \int_{-\pi}^{\pi} \e^{- \ic t k}\phi_{n}(t)^{ n} \d t.\]
In the following, we only consider values of $x \in \R$ such that $x_{n}+x  \sigma_{n}\sqrt{n}$ is an integer.
For such $x$,
\[ \sigma_{n}\sqrt{n} \cdot \Pr{{S}_{n}^{(b_{n})}= x_{n}+x  \sigma_{n}\sqrt{n}}=  \frac{1}{2\pi} \int^{\pi \sigma_{n}\sqrt{n}}_{-\pi \sigma_{n}\sqrt{n}} \e^{- \ic t x} \left( \phi_{n} \left(  \frac{t}{\sigma_{n}\sqrt{n}} \right)   \e^{- \ic t  \frac{1}{ \sigma_{n}{n}^{3/2}}} \right)^{n} \d t.\]
Therefore for every fixed $A, \eta>0$, for $n$ sufficiently large,
the difference
\[ \left| \sigma_{n}\sqrt{n} \Pr{{S}_{n}^{(b_{n})}= x_{n}+x  \sigma_{n}\sqrt{n}} -f( x)\right|\]
is upper bounded by $1/(2 \pi)$ times the modulus of the sum of the following five integrals:
\begin{align*}
I^{(1)}_{A}(x,n) &= \int_{-A}^{A} \e^{- \ic  t x} \left(   \left(  \phi_{n} \left(  \frac{t}{\sigma_{n}\sqrt{n}} \right)   \e^{- \ic t  \frac{1}{ \sigma_{n}{n}^{3/2}}} \right)^{n}- f(t) \right) \d t,
\\
I^{(2)}_{A,\eta}(x,n) &= \int_{A<|t|< \eta \sqrt{n}} \e^{- \ic t x - \ic t \frac{1}{ \sigma_{n}\sqrt{n}}} \phi_{n} \left(  \frac{t}{\sigma_{n}\sqrt{n}} \right) ^{ n}\d t,
\\
I^{(3)}_{\eta,\eta'}(x,n) &= \int_{\eta \sqrt{n}<|t|< \eta' \sigma_{n} \sqrt{n}} \e^{- \ic t x - \ic t \frac{1}{ \sigma_{n}\sqrt{n}}} \phi_{n} \left(  \frac{t}{\sigma_{n}\sqrt{n}} \right) ^{ n}\d t,
\\
I^{(4)}_{\eta'}(x,n) &= \int_{\eta' \sigma_{n} \sqrt{n}<|t|< \pi \sigma_{n}\sqrt{n}} \e^{- \ic t x - \ic t \frac{1}{ \sigma_{n}\sqrt{n}}} \phi_{n} \left(  \frac{t}{\sigma_{n}\sqrt{n}} \right) ^{ n}\d t,
\\
I^{(5)}_{A}(x) &= \int_{|t|>A} \e^{- \ic t x} f(t) \d t.
\end{align*}
We shall check that for every fixed $\epsilon>0$, there exist $A,\eta>0$ such that for every $n$ sufficiently large, for every $x \in \R$ (with the above integrality condition),
the modulus of each of the five integrals is less than $\epsilon$.

 We choose $\eta,\eta',c_{1},c_{\eta}>0$ such that the conclusions of Lemma~\ref{lem:techn} hold and then $A >0$ such that
\begin{equation}\label{eq:Ab}
2 \int_{A}^{\infty } \e^{- t^{2}/2} {\d}t < \epsilon \qquad \textrm{ and }  \qquad 2 \int_{A}^{\infty} \e^{-c_{1} t^{2} }\d t <  \epsilon.
\end{equation}

\emph{Bounding $|I^{(1)}_{A}(x,n)|$:} Since the convergence~\eqref{eq:cv1b} holds  uniformly on compact subsets of $\R$, then for   $n$ sufficiently large and $x \in \R$, the modulus of the integrand tends to $0$ as $n \to \infty$ uniformly in $x \in \R$, so indeed $|I^{(1)}_{A}(x,n)| \le  \epsilon$ for $n$ large enough.

 \emph{Bounding $|I^{(2)}_{{A},\eta }(x,n)|$:}  By Lemma~\ref{lem:techn}(iii), we have  for   $n$ sufficiently large and any $x \in \R$,
 \[|I^{(2)}_{{A},\eta }(x,n)| \le  \int_{A < |t| < \eta  \sqrt{n}} \e^{-c_{1} n (\frac{t}{ \sqrt{n}})^{2}} \d t  
 \le 2  \int_{A}^{\infty} \e^{-c_{1} {t^{2} }{} } \d t,  \]
and the last integral is less than $\epsilon$ by~\eqref{eq:Ab}.

  \emph{Bounding $|I^{(3)}_{\eta,\eta'}(x,n)|$:} By Lemma~\ref{lem:techn}(ii),  for   $n$ sufficiently large and $x \in \R$,
  \[|I^{(3)}_{\eta,\eta'}(x,n)| \le  2  \sqrt{n}  \int_{\eta }^{\eta' \sigma_{n}} | \phi_{n}  (s/\sigma_{n}) |^{n} \d s \le   2  \sqrt{n}\int_{\eta }^{\infty} \frac{c_{\eta}^{n}}{s^{\alpha n}}\d s =  2  \sqrt{n}  \frac{c_{\eta}^{n}}{(\alpha n-1)\eta ^{\alpha n-1}},\]
  which tends to $0$ since  $c_{\eta}<\eta^{\alpha}$.
  
  \emph{Bounding $|I^{(4)}_{\eta'}(x,n)|$:}   By Lemma~\ref{lem:techn}(i), for   $n$ sufficiently large and $x \in \R$, 
  \[ |I^{(4)}_{\eta'}(x,n) | \leq  2  \sqrt{n}  \int_{\eta' \sigma_{n} }^{\pi \sigma_{n}} | \phi_{n}  (s/\sigma_{n}) |^{n} \d s \le 2 \pi \sqrt{n} \sigma_{n}   \frac{1}{\sigma_{n}^{n}}  \cv 0.\]
  
  \emph{Bounding $|I^{(5)}_{A}(x)|$:} By the choice of $A$ in~\eqref{eq:Ab}, we directly conclude that $|I^{(4)}_{A}(x)| \le  \epsilon$ for every $x \in \R$.
This completes the proof.
\end{proof}

\subsection{Stable regime}
\label{sec:LLT_stable}

We finally consider the case where $x_{n}/n$ converges to the right-edge $\rho G'(\rho) / G(\rho)$ when the latter is finite (recall that that case was excluded in Theorem~\ref{thm:LTT_bulk}).
We assume throughout that $\xi$ has finite mean $m>0$, is aperiodic and belongs to the domain of attraction of a stable law with index $\alpha \in (1, 2]$, which means that there exists a function $L$, slowly varying at infinity, such that for every $s \in (0,1)$,
\begin{equation}
\label{eq:attraction}
G(s)=\E[s^{\xi}] = 1 - m + ms + (1-s)^{\alpha} L\left(\frac{1}{1-s}\right).
\end{equation}
Recall that $L$ is slowly varying if for every $c>0$, it holds $L(cx)/L(x) \to 1$ as $x \to \infty$.

Before stating limit theorems, we introduce scaling sequences. If $\Var(\xi)<\infty$ (which implies $\alpha=2$), we set $r_{n}=\sqrt{n \Var(\xi) {/2}}$.  In this case, $L$ has a finite limit at $\infty$, and setting $\ell=\lim_{\infty} L$, we have $\Var(\xi)=2\ell+m-m^{2}$. If $\Var(\xi)=\infty$, we let $(r_{n})_{n\ge 1}$ be a sequence satisfying $r_{n}^{\alpha} \sim n L(r_{n})$ as $n\to \infty$.  In the latter case, it is sometimes useful to express $L$ in terms of the truncated variance of $\xi$ (although we will not use it here): as in e.g.~\cite[Lemma~4.7]{BS15} (which covers the case $m=1$, but the argument in the general case is the same),  if we set $L_{2}(n) = n^{\alpha-2} \Var(\xi \ind{\xi \le n})$, then $L_{2}$ is slowly varying and as $n\to \infty$ we have $L(n) \sim \frac{\Gamma(3-\alpha)}{\alpha (\alpha-1)} L_{2}(n)$ (this is specific to the case $\Var(\xi)=\infty$).

We now introduce some notation. Let $X^{\alpha}$ denote the $\alpha$-stable process with no negative jump, whose law is characterised by $\E[\exp(-q X^{\alpha}_{t})] = \exp(t q^{\alpha})$ for every $t,q>0$; note that $X^{2}$ has the law of $\sqrt{2}$ times a standard Brownian motion. For every $t>0$,  let $d^{\alpha}_{t}$ denote the density of $X^{\alpha}_{t}$.
The classical local limit theorem (see e.g.~\cite[Theorem~4.2.1]{IL71}, combined with~\cite[Eq.~(7) and (41)]{Kor17} for the scaling constants) reads
\begin{equation}\label{eq:LLT_stable_classique}
\sup_{k \in \Z} \left|r_{n} \cdot \Pr{S_{n} = mn+k} - d^{\alpha}_{1}\left(\frac{k}{r_{n}}\right)\right| \cv 0.
\end{equation}
This gives precise first-order asymptotic estimates uniformly for $k/r_{n}$ belonging to a compact subset of $\R$. The last limit theorem of this section aims at giving a precise first-order asymptotic estimate for the quantity $\Pr{S_{n}=x_{n}}$ in the large deviation regime when $x_{n}/n \to m$ but $(x_{n}-mn)/r_{n} \to -\infty$.

\begin{thm}\label{thm:lls1}
Assume that~\eqref{eq:attraction} holds. Let $(\lambda_{n})$ be a sequence such that $\lambda_{n}/r_{n} \to -\infty$ and $\lambda_{n} /n \to 0$. Set $x_{n}=mn+\lambda_{n}$.
For every $n$ sufficiently large, let $b_{n}$ be such that $b_{n} G'(b_{n})/G(b_{n}) = x_{n}/n$. Let $\varepsilon_{n} > 0$ be such that $b_{n} = 1-\varepsilon_{n}$ and define
\[\scale_{n} = \sqrt{\frac{(\alpha-1) |\lambda_{n}|}{\varepsilon_{n}}}.\]
Then
\[\sup_{k \ge -x_{n}} \left|\scale_{n} \cdot \frac{b_{n}^{x_{n}+k}}{G(b_{n})^{n}} \cdot \Pr{S_{n} = x_{n}+k} -  \frac{1}{\sqrt{2\pi}}\exp\left(-\frac{k^{2}}{2\scale_{n}^{2}}\right)\right| \cv 0,\]
\end{thm}

Roughly speaking, when $\Var(\xi)<\infty$, this estimate, with $\sqrt{n\Var(\xi^{(b_{n})})}$ instead of $\scale_{n}$, follows from~\cite[Theorem 6.1.5]{BB08} and when $\Var(\xi)=\infty$ this follows from~\cite[Theorem  3.1]{JP87}. The identification of $\scale_{n}$ requires asymptotics estimates on the moments of $\xi^{(b_{n})}$, which is the core of the proof of Theorem~\ref{thm:lls1}.

\begin{proof}[Proof of Theorem~\ref{thm:lls1}] We start with some preliminary estimates. By standard arguments (see e.g.~\cite[Eq. 44]{Kor17}), it follows from~\eqref{eq:attraction} that when $u \downarrow 0$, we have 
\[G'(1-u)= m - \alpha u^{\alpha-1} L(1/u)(1+o(1)).\]
By taking $u = \varepsilon_{n}$, we get
\begin{align*}
\frac{\lambda_{n}}{n} = \frac{x_{n}}{n} - m
&= \frac{(1-\varepsilon_{n}) G'(1-\varepsilon_{n})}{G(1-\varepsilon_{n})} - m
\\
&= \frac{(1-\varepsilon_{n}) (m - \alpha \varepsilon_{n}^{\alpha-1} L(1/\varepsilon_{n})(1+o(1)))}{1 - m \varepsilon_{n} + \varepsilon_{n}^{\alpha} L(1/\varepsilon_{n})} - m
\\
&= - \varepsilon_{n} \left(m - m^{2} + \alpha \varepsilon_{n}^{\alpha-2} L(1/\varepsilon_{n})\right) (1+o(1))
.\end{align*}
If $\Var(\xi)<\infty$, if follows that
\begin{equation}\label{eq:varfini}
\frac{|\lambda_{n}|}{n \varepsilon_{n}} \cv \Var(\xi).
\end{equation}
Now assume that $\xi$ has infinite variance,  which entails that $\lim_{\infty} L=\infty$ in the case $\alpha=2$. Then
\begin{equation}\label{eq:equiv_scaling_distorted-new}
\frac{|\lambda_{n}|}{\varepsilon_{n}} \sim  \alpha n \varepsilon_{n}^{\alpha-2} L(1/\varepsilon_{n}), \qquad \scale_{n}^{2} \sim \alpha(\alpha-1) n \varepsilon_{n}^{\alpha-2}L(1/\varepsilon_{n}).
\end{equation}
Next, we claim that $r_{n}\varepsilon_{n} \to \infty$. Indeed, from $r_{n}^{\alpha} \sim n L(r_{n})$, it follows that
\[|\lambda_{n}|/r_{n} 
\sim \alpha (r_{n}\varepsilon_{n})^{\alpha-1} \frac{L(1/\varepsilon_{n})}{L(1/r_{n})}
.\]
Our claim follows: if the sequence $(r_{n}\varepsilon_{n})$ does not tend to infinity, it would have a subsequence with a finite limit, and then the right-hand side would have a finite limit along this subsequence, which contradicts the fact that $|\lambda_{n}|/r_{n} \to \infty$. In particular, $\scale_{n} \to \infty$.

We now establish the theorem by checking that
\begin{equation}
\label{eq:ll}
\sup_{k \ge -x_{n}} \left|\scale_{n}  \cdot \Pr{S^{(b_{n})}_{n} = x_{n}+k} -  \frac{1}{\sqrt{2\pi}}\exp\left(-\frac{k^{2}}{2\scale_{n}^{2}}\right)\right| \cv 0.
\end{equation}
When $\xi$ has finite variance,~\cite[Theorem 6.1.5]{BB08} (see in particular the second paragraph after its statement) gives~\eqref{eq:ll} with $\sqrt{\Var(\xi^{(b_{n})})}$ instead of $\scale_{n}$, and the estimate~\eqref{eq:varfini} yields $\scale_{n} \sim  \sqrt{n \Var(\xi^{(b_{n})})}$.

Now assume that $\xi$ has infinite variance. We check that we can apply~\cite[Theorem 3.1]{JP87}; the first step is to check that a central limit theorem holds. We start with estimating the moments of $\xi^{(b_n)}$. 
For every $k\ge 1$,
\begin{equation}\label{eq:derivees_G_distorted}
G^{(k)}(1-u) = \alpha(\alpha-1) \cdots (\alpha-k+1) u^{\alpha-k} L(1/u) (1+o(1)).
\end{equation}
Then we infer from~\eqref{eq:moments} that
\begin{equation}\label{eq:moment_3_scaling_distorted-new}
\begin{split}
\E\big[(\xi^{(b_n)})^{2}\big] 
&\equi \alpha(\alpha-1) \varepsilon_{n}^{\alpha-2} L(1/\varepsilon_{n}),
\\
\E\big[(\xi^{(b_n)})^{3}\big] 
&\equi \alpha(\alpha-1)(\alpha-2) \varepsilon_{n}^{\alpha-3} L(1/\varepsilon_{n}).
\end{split}
\end{equation}
Thus, using~\eqref{eq:equiv_scaling_distorted-new}, 
\[n \frac{\E[(\xi^{(b_n)})^{2}]}{\scale_{n}^{2}} \cv 1,\]
and
\[n \frac{\E[(\xi^{(b_n)})^{3}]}{{\scale}_{n}^{3}} 
\equi \frac{1}{\sqrt{n}}\frac{(\alpha-2)}{\sqrt{\alpha(\alpha-1) \varepsilon_{n}^{\alpha} L(1/\varepsilon_{n})}}
\equi \frac{(\alpha-2)}{\sqrt{\alpha-1}} \frac{1}{\sqrt{|\lambda_{n}| \varepsilon_{n}}}
.\]
In particular, $n {\E[(\xi^{(b_n)})^{3}]}/{{\scale}_{n}^{3}}  \to 0$ since  $r_{n}\varepsilon_{n} \to \infty$.

Now, as above, write
\[\E[\e^{\ic t \xi^{(b_n)}}]
=1+\ic \E[\xi^{(b_n)}]t- \E[(\xi^{(b_n)})^{2}] \frac{t^{2}}{2}+R^{n}_{1}(t)
\quad\text{with}\quad
|R^{n}_{1}(t)| \le \frac{1}{6} |t|^{3} \E[(\xi^{(b_n)})^{3}],\]
for every $n \ge  1$ and $t \in \R$. The previous estimates yield the existence of a constant $C >0$ such that $n|R^{n}_{1}(t/\scale_{n}) | \le  C |t|^{3}/\sqrt{|\lambda_{n}| \varepsilon_{n}}$ for every $t \in \R$,  $n \ge  1$.  Since $ n  R^{n}_{1}(t/\scale_{n}) \to 0$  uniformly for $t$ in compact subsets of $\R$, it readily follows that
\[\E\bigg[\exp\bigg(\ic t \frac{{S}_{n}^{(b_{n})}-x_{n} }{\scale_{n}}\bigg)\bigg]  \cv \e^{- \frac{t^{2}}{2}}.
\]

Finally, we check that the other conditions of Theorem  3.1 in~\cite{JP87} are satisfied  (using the notation of~\cite{JP87}) with $F_{n}$ being the distribution of $\xi^{(b_{n})}$. The random variable $\xi ^{(b_{n})}$ converges in distribution to $\xi$, which is stochastically compact (using the terminology of~\cite{JP87})  since it belongs to the domain of attraction of a stable law. We may therefore apply~\cite[Theorem 3.1]{JP87}. To identify the scaling constant, note that $\Pr{S_{n} \leq x_{n}} \to 0$, so Theorem~2.1 and Lemma~4.1(c) in~\cite{JP87} give the desired result (the quantity $V(\lambda_{n})$ {in their notation, with $\lambda_{n}=b_{n}$ here,} is precisely the variance of $\xi^{(b_{n})}$). This establishes~\eqref{eq:ll} and completes the proof. 
\end{proof}

\section{Scaling limits of nondecreasing random bridges}
\label{sec:marches}

In this section we establish invariance principles, as $ n \to \infty$, for nondecreasing random walk bridges, which are random paths which start at $0$, whose increments are nonnegative integers, and which end at time $n$ at a fixed value $x_{n}$. This can be viewed as extensions in different regimes of a result of Liggett~\cite{Lig68}, who roughly speaking focused on bridge random walks in the domain of attraction of a stable law in the `bulk regime', where the value of the endpoint is of the same order as the flucutations of the random walk (this corresponds to Theorem~\ref{thm:stable_marches_drift} below).

\subsection{Discrete and continuum bridges}
\label{sec:def_ponts}

Let us first precisely introduce the model and its continuum limits, before stating our results in Sec.~\ref{sec:results_marches} and proving them in Sec.~\ref{sec:LLT_implique_excursion_brownienne} and~\ref{sec:preuve_condensation} by relying on the local estimates from Sec.~\ref{sec:LLT}.

\subsubsection{Discrete bridges}
For every $n \ge 1$, let $x_{n} \in \N$ and define $\mathcal{S}^{n}_{x_{n}}$ to be the set of paths $s = (s_{k} ; 0 \le k \le n)$ such that
\[s_{0} = 0, \qquad s_{n} = x_{n}, \qquad\text{and}\qquad s_{k} - s_{k-1} \in \Z_{\ge0} \enskip\text{for every}\enskip 1 \le k \le n.\]
Let $\boldsymbol{\poids} = (\poids(i))_{i \ge 0}$ be a sequence of nonnegative real numbers and define a measure $w^{\boldsymbol{\poids}}$ on $\mathcal{S}^{n}_{x_{n}}$ by setting for every $s \in \mathcal{S}^{n}_{x_{n}}$,
\[w^{\boldsymbol{\poids}}(s) = \prod_{k=1}^{n} \poids(s_{k}-s_{k-1}).\]
We shall always implicitly assume that $n$ and $x_{n}$ are compatible with the support of $\boldsymbol{\poids}$, in the sense that $\mathcal{S}^{n}_{x_{n}}$ has nonzero total $w^{\boldsymbol{\poids}}$-weight, and we also assume without further notice that $ \{k \geq 0: \poids(k)>0\}$ is not included in a sublattice of $\Z$, i.e.~the largest $h>0$ such that there exists $a \in \R$ such that $ \{k \geq 0: \poids(k)>0\} \subset a+h\Z$ is $h=1$. The results carry through without such an aperiodicity condition after mild adaptations.

Let us denote by $\P^{\boldsymbol{\poids},n}_{x_{n}} = w^{\boldsymbol{\poids}}(\cdot)/w^{\boldsymbol{\poids}}(\mathcal{S}^{n}_{x_{n}})$ the associated probability distribution on $\mathcal{S}^{n}_{x_{n}}$. Note that if $\boldsymbol{\poids}$ is a probability measure, then $\P^{\boldsymbol{\poids},n}_{x_{n}}$ is just the law of the bridge of a $\boldsymbol{\poids}$-random walk conditioned to end at $x_{n}$ at time $n$. 
We shall denote throughout this paper by $S^{n} = (S^{n}_{k} ; 0 \le k \le n)$ a random path sampled according to $\P^{\boldsymbol{\poids},n}_{x_{n}}$ 
(we write $S^{n}$ with $x_{n}$ implicit to lighten the notation).
By analogy with random trees (see Sec.~\ref{sec:arbres_aleatoires} below), we call such random paths \emph{simply generated bridges}. 
As shown in the next lemma, we could have equivalently worked with random walk bridges, but in view of applications to Boltzmann planar maps, the simply generated framework is more convenient.

We  introduce some notation which was also used in Sec.~\ref{sec:LLT}. We emphasize that as opposed to there, the sequence $\boldsymbol{\poids}$ here is not necessarily a probability measure. Set
\begin{equation}\label{eq:serie_gen_poids_marches}
G(z)=\sum_{k=0}^{\infty}  \poids(k)  z^{k}
\qquad\text{and}\qquad
\Psi(z) = \frac{zG'(z)}{G(z)}.
\end{equation}
Denote by $\rho$ the radius of convergence of $G$, which we will always assume to be positive. Let $i_{\boldsymbol{\poids}} = \inf\{i \ge 0 : \poids(i) > 0\}$ and observe that $\mathcal{S}^{n}_{x_{n}}=\varnothing$ if $x_{n} < i_{\boldsymbol{\poids}} n$.
It is a simple matter to check that $\Psi$ is increasing on $(0, \rho)$, we extend it by continuity with $\Psi(0) = i_{\boldsymbol{\poids}}$ and $\Psi(\rho) \in (0, \infty]$. Note that $\Psi(\rho) = \infty$ if and only if $G'(\rho) = \infty$. When $G'(b)<\infty$, $\Psi(b)$ is simply the expectation of a random variable with generating function $z \mapsto {G(b z)}/{G(b)}$.
In the context of bridges, the exponential tilting, which was presented in the introduction, leaves the law of the path invariant, as is made explicit by the next result.

\begin{lem}\label{lem:tilting}
Let $\boldsymbol{\poids}$ be a non-negative sequence and $a,b>0$, and define another sequence $\boldsymbol{\nu}$ by setting $\nu(k) = a b^{k} \poids(k)$ for every $k\ge0$. 
\begin{enumerate}
\item The laws $\P^{\boldsymbol{\poids},n}_{x_{n}}$ and $\P^{\boldsymbol{\nu},n}_{x_{n}}$ on $\mathcal{S}^{n}_{x_{n}}$ are equal.

\item Furthermore, for every choice of $b \in (0,\rho]$ such that $G'(b) < \infty$, taking $a = 1/G(b)$ provides a probability measure with mean $\Psi(b)$.
\end{enumerate}
\end{lem}

\begin{proof}
For the first assertion, for every path $s \in \mathcal{S}^{n}_{x_{n}}$, it holds
\[w^{\boldsymbol{\nu}}(s) = \prod_{k=1}^{n} a b^{s_{k}-s_{k-1}} \poids(s_{k}-s_{k-1})
= a^{n} b^{x_{n}} w^{\boldsymbol{\poids}}(s),\]
and since $a^{n} b^{x_{n}}$ does not depend on $s$, then indeed, when rescaled by the total mass, the laws $\P^{\boldsymbol{\poids},n}_{x_{n}}$ and $\P^{\boldsymbol{\nu},n}_{x_{n}}$ are equal.
The second assertion is immediate and actually motivated our definition of $\Psi$.
\end{proof}

Therefore, whatever the original weight sequence $\boldsymbol{\poids}$ used to generate the random path $S^{n}$ in the first place is, as soon as its generating series has a nonzero radius of convergence, we can always tilt the weights in order to view $S^{n}$ as a random walk conditioned on its value at time $n$.
Moreover, the family of such probability measures is parameterised by their mean value, which lies in $(i_{\boldsymbol{\poids}}, \Psi(\rho)]$, where the right-edge is included if and only if $G'(\rho) < \infty$.
Note that such a law has generating function $G(b\,\cdot)/G(b)$, and is the law of $\xi^{(b)}$ in the notation of Sec.~\ref{sec:LLT}; let us also recall the first moments of this law, given in~\eqref{eq:moments}.

The main results of this section are convergence of these paths, suitably rescaled, towards bridges of L\'evy processes that we now introduce.

\subsubsection{Continuum Lévy bridges}

For any interval $I \subset \R_{+}$, we shall denote by $\D(I, \R)$ the space of real-valued c\`adl\`ag functions on $I$, equipped with the Skorokhod $J_{1}$ topology (see Chapter~VI in~\cite{JS03} for background). For any $\alpha \in (1,2]$, we shall denote by $X^{\alpha}$ the $\alpha$-stable L\'evy process with no negative jump, whose law is characterised by the Laplace transform $\E[\exp(-q X^{\alpha}_{t})] = \exp(t q^{\alpha})$ for every $q,t > 0$; note that $X^{2}$ has the law of $\sqrt{2}$ times a standard Brownian motion. Similarly to the discrete setting, we can make sense of the law of $(X^{\alpha}_{t}-ct ; 0 \le t \le 1)$ conditioned on the event $\{X^{\alpha}_{1} - c=0\}$ with $c \in \R$. We refer to e.g.~\cite{Ber96}, Chapter~XIII, and especially Sec.~3 there for details (which focuses in the case $c=0$, but the arguments extend readily).

Indeed, recall that the law of $X^{\alpha}_{t}$ is absolutely continuous with respect to the Lebesgue measure, and we shall denote by $d^{\alpha}_{t}$ its density; we stress that this function is explicit only in the case $\alpha=2$, in which case $d^{2}_{t}(x) = (4\pi t)^{-1/2} \exp(-x^{2}/(4t))$. 
Then for every $c \in \R$, we can define a process $(X^{\alpha,c,\br}_{t} ; 0 \le t \le 1)$, which is informally, the process $(X^{\alpha}_{t}-ct ; 0 \le t \le 1)$  conditioned to end at $0$ at time $1$, whose law is characterised by the following absolute continuity relation: for every $0<u<1$ and every bounded continuous functional $F: \D([0,u]) \to \R$,
\begin{equation}\label{eq:lawbridgestable}
\Es{F\left(X^{\alpha,c,\br}_{s} ; 0 \le s \le u \right)} = \Es{F\left(X^{\alpha}_{s}-cs ; 0 \le s \le u \right) \frac{d^{\alpha}_{1-u}(c-X^{\alpha}_{u})}{d^{\alpha}_{1}(c)}}.
\end{equation}
Finally, we shall denote by $B^{\br}$ the process $2^{-1/2} X^{2,0}$, which is the standard Brownian bridge going from $0$ to $0$ in one unit of time.
One can check, using Girsanov's formula, that in fact $X^{2,c,\br}$ always has the law of $\sqrt{2}$ times this standard Brownian bridge, for any $c \in \R$. This is not the case when $\alpha \in (1,2)$.

\subsection{Statements of the main results}
\label{sec:results_marches}

We turn to our main scaling limit results concerning the random simply generated bridges $S^{n}$ under $\P^{\boldsymbol{\poids},n}_{x_{n}}$, for which $S^{n}_{n} = x_{n}$. Up to extraction of a subsequence, we may assume without loss of generality that $x_{n} / n \to \gamma \in [0, \infty]$ as $n \to \infty$. 
We shall canonically view our discrete path $S^{n}$ as a c\`adl\`ag function on $[0,1]$, equipped with the Skorokhod $J_{1}$ topology; recall that when the limit is a continuous path, convergence in this topology is equivalent to the uniform convergence. For every $k \in \{1, \dots, n\}$, we shall denote by $X^{n}_{k} = S^{n}_{k} - S^{n}_{k-1}$ the $k$'th increment of $S^{n}$.
In view of applications to random maps, we also include in our statements a law of large numbers for the  sum of the square of the increments.

Recall the first moments of $\xi^{(b)}$ from~\eqref{eq:moments}, let us denote by $m_{b} = \frac{b G'(b)}{G(b)}$ and $\sigma^{2}_{b} = \frac{b^{2} G^{(2)}(b) + b G'(b)}{G(b)} - (\frac{b G'(b)}{G(b)})^{2}$ respectively the mean and variance of $\xi^{(b)}$.
We deduce from Theorem~\ref{thm:LTT_forme_generale_hors_stable} the following limits.

\begin{thm}
\label{thm:CVmarches}
Let $(S^{n}_{k} ; 0 \le k \le n)$ have the law $\P^{\boldsymbol{p},n}_{x_{n}}$ for every $n \ge 1$. 
The convergences
\[\left(\frac{1}{\sqrt{\scale_{n}}} \left(S^{n}_{\floor{nt}} - x_{n} t\right) ; 0 \le  t \le  1\right) \cvloi B^{\br}
\qquad\text{and}\qquad
\frac{1}{\scale_{n}} \sum_{k=1}^{n} (X^{n}_{k})^{2} \cvproba C,\]
hold in each of the following cases:
\begin{thmlist}
\item\label{thm:gencrit_marches} There exists $\gamma \in  (i_{\boldsymbol{\poids}}, \Psi(\rho))$ such that $x_{n}/n \to \gamma$, and $\scale_{n} = n \sigma^{2}_{b}$ and $C = m_{b}/\sigma^{2}_{b}$ with $b=\Psi^{-1}(\gamma)$.

\item\label{thm:gencrit0_marches} $x_{n}/n \to 0$ and $x_{n} \to \infty$, $\poids(0), \poids(1)>0$, $\scale_{n} = x_{n}$, and finally $C=1$.

\item\label{thm:nongenneg_marches} $x_{n}/n \to \infty$, $G$ is $\Delta$-analytic and there exist $\alpha, c>0$ such that $G(\rho-z) \sim c/z^{\alpha}$ as $z \to 0$ with $\mathrm{Re}(z)>0$,  $\scale_{n} = x_{n}^{2}/(\alpha n)$, and finally $C=\alpha+1$.
\end{thmlist}
\end{thm}

As already mentioned, the functional convergence in the first case is essentially contained in~\cite{Lig68}.
On the other hand, to the best of our knowledge, the two other cases are new.
Theorem~\ref{thm:gencrit0_marches} is rather intuitive: if we force the path to end at a very low value, then first $\poids(0)$ has to be nonzero, and as alluded after Theorem~\ref{thm:LTT_forme_generale_hors_stable} the assumption $\poids(1)>0$ prevents some periodicity in the large scale, and finally in the limit the path must not make any large jump, hence a Brownian limit.
On the other hand, in Theorem~\ref{thm:nongenneg_marches}, we condition the path to end at a very high value, and still we obtain a Brownian limit; this unusual behaviour comes from the assumptions on the weights (which, as noted in the introduction, naturally appear when looking at uniform random maps), and Theorem~\ref{thm:stable_marches_condensation} below exhibits a very different behaviour.

\smallskip
In our last main result on these random paths, we study  the boundary case $x_{n}/n \to \Psi(\rho)$ when the latter is finite. To do so, we assume that the probability distribution with generating function $G(\rho \,\cdot)/G(\rho)$ belongs to the domain of attraction of a stable law with index $\alpha \in (1, 2]$, i.e.~that there exist $\alpha \in (1,2]$, $\rho>0$, $m>0$, and a slowly varying function $L$ such that for every $0 \le  s <  \rho$,
\[G(s)=G(\rho) \left(1-m+\frac{ms}{\rho}+\left(1-\frac{s}{\rho}\right)^{\alpha} L\left(\frac{1}{1-s/\rho}\right)\right).\]
In the case $G^{(2)}(\rho) < \infty$, i.e.~when $\alpha=2$ and $L$ admits a finite limit at infinity, say $\ell$, the law with generating function $G(\rho \,\cdot)/G(\rho)$ has a finite variance, equal to
\[\sigma^{2} = 2\ell + m - m^{2}
= \frac{\rho^{2} G^{(2)}(\rho)}{G(\rho)} + \frac{\rho G'(\rho)}{G(\rho)} - \left(\frac{\rho G'(\rho)}{G(\rho)}\right)^{2}
.\]
In this case we set $r_{n} = \sqrt{n \sigma^{2}/2}$ for every $n\ge 1$, otherwise, we let $(r_{n})_{n \ge 1}$ be a sequence such that $r_{n}^{\alpha} \sim n L(r_{n})$ as $n \to \infty$.
It is well-known (see e.g.~\cite[Theorem XVII.5.]{Fel71} and~\cite[Theorem 16.14]{Kal02}) that if $(S_{n})_{n\ge 1}$ is a random walk with i.i.d. increments with generating function $G(\rho \,\cdot)/G(\rho)$, then the convergence
\[\left(\frac{1}{r_{n}} \cdot \left(S_{\floor{nt}} - mnt\right) ; t \ge 0\right)  \cvloi X^{\alpha},\]
holds in distribution in $\D([0,\infty), \R)$. See the remarks around~\eqref{eq:LLT_stable_classique} for the scaling.

Recall from Sec.~\ref{sec:def_ponts} that for $\lambda\in\R$, we denote by $X^{\alpha,\lambda,\br}$ the process $(X^{\alpha}_{t}-\lambda t)_{t}$ conditioned to end at $0$ at time $1$.
Let us also denote by $\Delta X^{\alpha,\lambda,\br}_{t} = X^{\alpha,\lambda,\br}_{t}-X^{\alpha,\lambda,\br}_{t-}$ the value of the `jump' of $X^{\alpha,\lambda,\br}$ at time $t \in [0,1]$, which is null except for a countable set of times.

The functional convergence in distribution in item (iii) below is already known (see the discussion just afterwards), but we include it here for completeness.

\begin{thm}
\label{thm:stable_marches}
Under the preceding assumptions, if $x_{n}$ takes the form
\[x_n = m n + \lambda_{n},\]
where  
$\lambda_{n}/r_{n}$ converges to some $\lambda \in [-\infty,\infty]$, then three cases occur.
\begin{thmlist}
\item\label{thm:stable_marches_drift} If $\lambda \in \R$, then
\[ \left( \frac{1}{r_{n}} \cdot \left(S^{n}_{\floor{nt}} - x_{n} t\right) ; 0 \le  t \le  1\right)  \cvloi X^{\alpha,\lambda,\br}.\]
In addition, 
\[r_{n}^{-2} \sum_{k=1}^{n} (X^{n}_{k})^{2} \cvloi
\begin{dcases}
 2 + \frac{2 m^{2}}{\sigma^{2}} \ind{G^{(2)}(\rho) < \infty} &\textrm{ if } \alpha = 2,\\
\sum_{t \in [0,1]} (\Delta X^{\alpha,\lambda,\br}_{t})^{2}  &\textrm{ if } \alpha<2.
\end{dcases} 
\]

\item\label{thm:stable_marches_gaussien} If $\lambda = -\infty$ and if $\lambda_{n}/n \to 0$, then with $\varepsilon_{n}$ such that $\Psi(\rho(1-\varepsilon_{n})) = x_{n}/n$, we have
\[ \left(\sqrt{\frac{1}{\alpha-1} \frac{\varepsilon_{n}}{|\lambda_{n}|}} \cdot \left(S^{n}_{\floor{nt}} - x_{n} t\right) ; 0 \le  t \le  1\right)  \cvloi B^{\br}.\]
In addition, 
\[\frac{\varepsilon_{n}}{|\lambda_{n}|} \sum_{k=1}^{n} (X^{n}_{k})^{2} \cvproba \alpha-1 + \frac{m^{2}}{\sigma^{2}} \ind{G^{(2)}(\rho)<\infty}.\]

\item\label{thm:stable_marches_condensation} If $\lambda = \infty$, we further assume that $\poids(k)\rho^{-k}$ is regularly varying at infinity with some index $-\beta < -2$. 
If $G^{(2)}(\rho)<\infty$ (which implies $\beta\ge3$), we  furthermore assume that there exists $c>(\beta-3)/\sigma^{2}$ such that $\lambda_{n} \geq \sqrt{c \ln(n)}$. 
Then
\[\left(\frac{1}{\lambda_{n}} \cdot \left(S^{n}_{\floor{nt}} - x_{n} t\right) ; 0 \le  t \le  1\right)  \cvloi (\ind{U \leq t} - t; 0 \leq t \leq 1),\]
where $U$ is uniformly distributed on $[0,1]$. 
In addition,
\[ \frac{\max \{X^{n}_{1},X^{n}_{2}, \ldots, X^{n}_{n}\}^{2} }{\sum_{k=1}^{n} (X^{n}_{k})^{2} }  \cvproba1.\]
\end{thmlist}
\end{thm}

Let us make some comments. In the first regime, we condition the path $S^{n}$ to end at a value which only differs from its average behaviour by some constant $\lambda$ times $r_{n}$, which is the scaling of the natural fluctuations as we recalled before the statement, so the scaling remains the same and the limit is simply the same L\'evy process as in the unconditioned case, but conditioned to end at $\lambda$. In the second regime however we condition the path to end at a much lower value (we assume $\lambda_{n}/n \to 0$ as otherwise we fall in the bulk regime of Theorem~\ref{thm:gencrit_marches}) and this completely changes the scaling, from $1/r_{n}$ to $\sqrt{\varepsilon_{n}/|\lambda_{n}|}$, which, as we shall later see, is of order $1/r_{n}$ in the finite variance case, but larger otherwise;
this wipes away the natural behaviour of $G$ and it yields in fact a Gaussian limit at this scaling, which can be understood intuitively, as for Theorem~\ref{thm:gencrit0_marches}, by the fact that, in order to end at a low value, the path must avoid any large jump.

Finally, in the last regime, a `one big jump principle' holds, which contrasts with Theorem~\ref{thm:nongenneg_marches}. The functional convergence in this regime is already known (see e.g.~\cite[Theorem 6.4.1]{BB08}). However we will need a more detailed analysis, which follows from a direct application of the results from~\cite{AL11,DDS08}:
In addition to a unique giant increment, of size asymptotically equivalent to $\lambda_{n}$ (as the combination of the two convergences shows), occurring at a random time asymptotically uniformly distributed, the $n-1$ other increments are close (in total variation) to just i.i.d. random variables, with generating function $G(\rho \,\cdot)/G(\rho)$, so the path obtained by removing this giant increment stays in the scale $r_{n}$. 
We refer to Sec.~\ref{sec:preuve_condensation} below for details.

\subsection{Functional convergence from a local limit estimate}
\label{sec:LLT_implique_excursion_brownienne}

Let us now prove how the L\'evy bridges appear as limits of our discrete paths, based on the local limit estimates from Sec.~\ref{sec:LLT}. Since the same argument applies to prove most of them, we present a general result that we then apply to our case. We denote by $X$  a stable Lévy process with a linear drift, with a density $d_{t}$ at time $t$, and we let $X^{\br}$ denote a version of $X$ conditioned on $X_{1} = 0$, whose definition was given in Sec.~\ref{sec:def_ponts}.
For the proof of Theorems~\ref{thm:gencrit_marches},~\ref{thm:gencrit0_marches},~\ref{thm:nongenneg_marches}, and~\ref{thm:stable_marches_gaussien} we shall take $X$ to simply be a standard Brownian motion without drift.

\begin{prop}[Subcritical conditioning]
\label{prop:LLT_implique_convergence_loi_varie}
Suppose that $x_{n}/n < \Psi(\rho)$ for every $n$ large enough and let $b_{n} = \Psi^{-1}(x_{n}/n)$. Assume that there exists a sequence $\scale_{n} \to \infty$ such that for every $t > 0$,
\begin{equation}\label{eq:LLT_S_varie}
\sup_{k \ge 0} \left|\scale_{n} \cdot \frac{b_{n}^{k}}{G(b_{n})^{\floor{nt}}} \cdot \Pr{S^{n}_{\floor{nt}} = k} - d_{t}\left(\frac{k-x_{n}t}{\scale_{n}}\right)\right| \cv 0.
\end{equation}
Then the convergence
\[\left( \frac{1}{\scale_{n}} \left(S^{n}_{\lfloor n t \rfloor} - x_{n} t\right)  ; 0 \le  t \le  1\right)   \cvloi X^{\br}\]
holds in distribution in $\D([0,1],\R)$.
\end{prop}

\begin{proof}
By construction of $b_{n}$, we can define a random variable $\widehat{X}^{n}$ by setting for every $k \ge 0$,
\[\P\big(\widehat{X}^{n} = k\big) = \frac{1}{G(b_{n})} b_{n}^{k} \poids(k),
\qquad\text{so}\qquad
\E\big[\widehat{X}^{n}\big]
= \Psi(b_{n})
= \frac{x_{n}}{n}.\]
Furthermore, according to Lemma~\ref{lem:tilting}, the random walk $\widehat{S}^{n} = (\widehat{S}^{n}_{i} ; 0 \le i \le n)$ with step distribution $\widehat{X}^{n}$, when conditioned on $\widehat{S}^{n}_{n} = x_{n}$, has the same law as our original path $S^{n}$.
Now the local limit estimate~\eqref{eq:LLT_S_varie} can be rewritten as
\begin{equation}\label{eq:LLT_S_varie_bis}
\sup_{k \ge 0} \left|\scale_{n} \cdot \Pr{\widehat{S}^{n}_{\floor{nt}} = k} - d_{t}\left(\frac{k-x_{n}t}{\scale_{n}}\right)\right| \cv 0.
\end{equation}
Then this implies the convergence in distribution of $\scale_{n}^{-1} (\widehat{S}^{n}_{\floor{nt}} - x_{n} t)$ to $X_{t}$, under the unconditional law, for any $t>0$ fixed, which, by e.g.~\cite[Theorem~16.14]{Kal02}, actually suffices to conclude to the convergence of the whole path, still under the unconditional law, namely
\begin{equation}\label{eq:cv_S_non_conditionne}
\left(\frac{1}{\scale_{n}} \left(\widehat{S}^{n}_{\floor{nt}} - x_{n} t\right) ; 0 \le t \le 1\right) \cvloi (X_{t} ; 0 \le t \le 1).
\end{equation}

We deduce the convergence of bridges on any time interval $[0,u]$ with $u \in (0,1)$ by absolute continuity. Indeed, let us set $\varphi^{n}_{i}(k) = \P(\widehat{S}^{n}_{i} = k)$ for every $i$ and $k$ and let $F : \D([0,u],\R) \to \R$ be a bounded continuous functional, then by~\eqref{eq:LLT_S_varie_bis} first and then~\eqref{eq:cv_S_non_conditionne},
\begin{align*}
&\Es{F\left(\frac{1}{\scale_{n}} \left(\widehat{S}^{n}_{\floor{nt}} - x_{n} t\right) ; 0 \le t \le u\right) \;\middle|\; \widehat{S}^{n}_{n}=x_{n}}
\\
&= \Es{F\left(\frac{1}{\scale_{n}} \left(\widehat{S}^{n}_{\floor{nt}} - x_{n} t\right) ; 0 \le t \le u\right) \frac{\varphi^{n}_{n-\floor{nu}}(x_{n} -\widehat{S}^{n}_{\floor{nu}})}{\varphi^{n}_{n}(x_{n})}}
\\
&= \Es{F\left(\frac{1}{\scale_{n}} \left(\widehat{S}^{n}_{\floor{nt}} - x_{n} t\right) ; 0 \le t \le u\right) \frac{d_{1-u}(- \scale_{n}^{-1} (\widehat{S}^{n}_{\floor{nu}} - x_{n} u))}{d_{1}(0)}} (1+o(1))
\\
&= \Es{F\left(X_{t} ; 0 \le t \le u\right) \frac{d_{1-u}(-X_{u})}{d_{1}(0)}} (1+o(1)).
\end{align*}
Notice that $(\widehat{S}^{n}_{n}-\widehat{S}^{n}_{n-i})_{ 0 \le  i \le  n}$ has the same distribution as $(\widehat{S}^{n}_{i})_{ 0 \le  i \le  n}$ under $\P(\,\cdot \mid \widehat{S}^{n}_{n}=x_{n})$, so the convergence of bridges also holds on the time interval $[u,1]$. In particular, the portion of the path on $[u,1]$ is tight and we may extend the preceding convergence in distribution of the bridges to the whole interval $[0,1]$. 
This shows our claim when $S^{n}$ is replaced by the bridge of $\widehat{S}^{n}$ and we conclude the proof by recalling that these paths have the same law.
\end{proof}

With this result at hand and the local estimates from Sec.~\ref{sec:LLT}, we can easily prove some of our main results.
\begin{proof}[Proof of Theorems~\ref{thm:CVmarches} and~\ref{thm:stable_marches_gaussien}]
Let us first prove Theorem~\ref{thm:gencrit_marches}; under its assumptions, we deduce from Theorem~\ref{thm:LTT_bulk} and Proposition~\ref{prop:LLT_implique_convergence_loi_varie}, the convergence
\[\left( \frac{1}{\sqrt{\widehat{\sigma}_{n}^{2}  n}} \left(S^{n}_{\lfloor n t \rfloor} - x_{n} t\right)  ; 0 \le  t \le  1\right)   \cvloi B^{\br},\]
where $\widehat{\sigma}_{n}^{2}$ is the variance of $\widehat{X}^{n}$ in the above notation, which is given by~\eqref{eq:moments}; by continuity it converges to the same expression with $b$ instead of $b_{n}$, i.e.~$\sigma^{2}_{b}$ in the notation of Theorem~\ref{thm:gencrit_marches}.
Similarly, Theorems~\ref{thm:gencrit0_marches},~\ref{thm:nongenneg_marches}, and~\ref{thm:stable_marches_gaussien} immediately follow by combining Proposition~\ref{prop:LLT_implique_convergence_loi_varie} with Theorem~\ref{thm:LTT_small_endpoint}, Theorem~\ref{thm:LTT_large_endpoint}, and Theorem~\ref{thm:lls1} respectively.

In each case, the proof of the convergence of the sum of the increments squared is deferred to Appendix~\ref{sec:preuves_LLT}, see Proposition~\ref{prop:ctegeneric}, Proposition~\ref{prop:ctegencritfewvertices}, Proposition~\ref{prop:ctenongenericnegativeindex}, and Proposition~\ref{prop:cte_nongenstable_lambda_moins_infini} respectively. In each regime, the basic idea is to control the first four moments of $\widehat{X}^{n}$, which are expressed in~\eqref{eq:moments}.
\end{proof}

We next modify the preceding statement in order to control the case where $x_{n}/n$ is slightly larger than $\Psi(\rho)$, which will enable us to treat Theorem~\ref{thm:stable_marches_drift}.
For $c \in \R$ and $X$ the preceding L\'evy process, recall that we define $X^{c,\br}$ as the process $(X_{t} - ct ; 0 \le t \le 1)$ conditioned to end at $0$.

\begin{prop}[Critical conditioning]
\label{prop:LLT_implique_convergence_loi_fixe}
Assume that $\Psi(\rho) < \infty$ and that there exists a sequence $\scale_{n} \to \infty$ such that for every $t > 0$,
\[\sup_{k \ge 0} \left|\scale_{n} \cdot \frac{\rho^{k}}{G(\rho)^{\floor{nt}}} \cdot \Pr{S^{n}_{\floor{nt}} = k} - d_{t}\left(\frac{k-\Psi(\rho) n t}{\scale_{n}}\right)\right| \cv 0.
\]
If there exists $c \in \R$ such that $\scale_{n}^{-1} (x_{n} - \Psi(\rho) n) \to c$, then the convergence
\[\left( \frac{1}{\scale_{n}} \left(S^{n}_{\lfloor n t \rfloor} - x_{n} t\right)  ; 0 \le  t \le  1\right)   \cvloi X^{c, \br}\]
holds in distribution in $\D([0,1])$.
\end{prop}

\begin{proof}
Let us replace $b_{n}$ in the preceding proof by $\rho$, i.e.~define a random walk $\widehat{S}$ with step distribution $\widehat{X}$ given by
\[\P\big(\widehat{X} = k\big) = \frac{1}{G(\rho)} \rho^{k} \poids_{k}
\qquad\text{for every}\qquad k \ge 0.\]
Note that its mean equals $\E[\widehat{X}] = \Psi(\rho)$, then the argument that lead to~\eqref{eq:cv_S_non_conditionne} applies here and we deduce, using also $\scale_{n}^{-1} (x_{n} - \Psi(\rho) n) \to c$, that
\[\left(\frac{1}{\scale_{n}} \left(\widehat{S}_{\floor{nt}} - x_{n} t\right) ; 0 \le t \le 1\right) \cvloi (X_{t} - ct ; 0 \le t \le 1),\]
under the unconditional law.
As previously, this then transfers to the bridge conditioning by absolute continuity and time-reversal.
\end{proof}

The convergence of paths in Theorem~\ref{thm:stable_marches_drift} then immediately follows as the preceding theorems.

\begin{proof}[Proof of Theorem~\ref{thm:stable_marches_drift}]
As the preceding case, we deduce immediately from~\eqref{eq:LLT_stable_classique} and Proposition~\ref{prop:LLT_implique_convergence_loi_fixe} the convergence
\[\left(\frac{1}{r_{n}} \left(S^{n}_{\lfloor n t \rfloor} - x_{n} t\right)  ; 0 \le  t \le  1\right)   \cvloi X^{\alpha,\lambda,\br}.\]
Let us next prove the convergence of the sum of the squares of the increments in the case $\alpha \in (1,2)$. 
The preceding convergence can be equivalently written as
\[\left(\frac{1}{r_{n}} \left(S^{n}_{\lfloor n t \rfloor} - mn t\right)  ; 0 \le  t \le  1\right)  \cvloi \left(X^{\alpha,\lambda,\br}_{t} -\lambda t  ; 0 \le  t \le  1\right),\]
for the Skorokhod topology. Let us denote by $X^{\alpha,\lambda}$ the process on the right, which is informally the stable process $X^{\alpha}$ conditioned on $X^{\alpha}_{1} = \lambda$. Note that its jumps are the same as those of $X^{\alpha,\lambda,\br}$.

This convergence implies that for every $N$ fixed, the $N$ largest values amongst $\{r_{n}^{-1} (X^{n}_{k} - m), 1 \le k \le n\}$ jointly converge towards the values of the $N$ largest jumps of $X^{\alpha,\lambda}$. Consequently the sum of the squares of the former converge to that of the latter.
Recall that we assume $\alpha \in (1,2)$ so $\sum_{k=1}^{n} X^{n}_{k} = x_{n} \sim m n = o(r_{n}^{2})$ and thus for every $N \ge 1$ fixed, the sum of the largest $N$ largest values amongst $\{r_{n}^{-2} (X^{n}_{k})^{2}, 1 \le k \le n\}$ converges towards the sum of the  $N$ largest jumps squared of $X^{\alpha,\lambda}$.
We next argue that one can fix $N$ large such that both rests are small with high probability, which implies our claim.

Recall that $S^{n}$ has the same law as the random walk $\widehat{S}$ conditioned on $\widehat{S}_{n} = x_{n}$ so we may replace the $X^{n}_{k}$'s by the increments $\widehat{X}_{k}$'s of the latter, under the conditional law.
We first notice that this holds true for the unconditioned processes. Indeed it is well-known that setting $(S^{\alpha}_{t})_{t \in [0,1]} = (\sum_{s\le t} (\Delta X^{\alpha}_{t})^2)_{t \in [0,1]}$ defines a stable subordinator with index $\alpha/2$, which is a pure jump process. Moreover, from the tail behaviour of $\widehat{X}^{2}$, one can check  (using e.g.~\cite[Theorem XVII.5.]{Fel71} and~\cite[Theorem 16.14]{Kal02}) that the rescaled process $(r_{n}^{-2} \sum_{k=1}^{\floor{nt}} (\widehat{X}_{k})^{2})_{t \in [0,1]}$ converges in distribution towards $(S^{\alpha}_{t})_{t \in [0,1]}$. Note that no centering is needed here since $\alpha/2 < 1$.
Then almost surely, $S^{\alpha}_{1/2} = \sum_{s\le 1/2} \Delta S^{\alpha}_{s}$ and one can fix $N$ such that the sum on the right, when the $N$ largest terms are excluded, is arbitrarily small with high probability. Using the preceding convergence, one gets that the same holds for the sequence $(\widehat{X}_{k})_{k \le \floor{n/2}}$.

Next, these properties are transferred to the bridge conditioning by the absolute continuity with the unconditioned process: There exists $C>0$ such that, both for the discrete and the continuum process, the probability of any event which only depends on the first half of the bridge is bounded by $C$ times the same probability for the unconditioned process. For $X^{\alpha}$, this follows from the fact that the density $d^{\alpha}_{t}$ is a bounded function for any $t>0$, whereas for $\widehat{S}$ one also appeals to the local limit theorem~\eqref{eq:LLT_stable_classique}. Using the invariance of the bridges under space-time reversal, we deduce that for every $\varepsilon, \delta > 0$, one can fix $N$ such that for every $n$ large enough, for each half, the sum of the square of the increments excepted the $N$ largest ones are smaller than $\delta$ with probability at least $1-\varepsilon$ and the proof is complete.

The convergence of the sum of the squares of the increments  in the case $\alpha=2$ is deferred to Proposition~\ref{prop:ctealpha2}.
\end{proof}

\subsection{The condensation phase}
\label{sec:preuve_condensation}

Let us finally focus on Theorem~\ref{thm:stable_marches_condensation}: Here we assume that $x_{n}=mn+ \lambda_{n}$ with $\lambda_{n}/r_{n} \to \infty$ and that $\poids(k)\rho^{-k}$ is regularly varying at infinity with some index $-\beta < -2$. 
If $G^{(2)}(\rho)<\infty$ (which implies $\beta\ge3$), we assume furthermore that there exists $c>(\beta-3)/\sigma^{2}$ such that $\lambda_{n} \geq \sqrt{c \ln(n)}$. 

For every $1 \le i \le n$, let us set $Y_{i}^{n} = X_{i}^{n} - m$ and then $T^{n}_{i} = Y_{1}^{n} + \dots + Y_{i}^{n} = S^{n}_{i} - im$.
The first claim is equivalent to the convergence 
\begin{equation}\label{eq:cv_condensation_centre}
\left(\frac{1}{\lambda_{n}} \cdot T^{n}_{\floor{nt}} ; 0 \le  t \le  1\right)  \cvloi (\ind{U \leq t}; 0 \leq t \leq 1),
\end{equation}
where $U$ is uniformly distributed on $[0,1]$. 
This is a straightforward consequence of a much more precise result.
Precisely, let
\[V_n \coloneqq \inf \left\{1\leq j \leq n : Y^{n}_j=\max\{Y^{n}_i : 1\leq i \leq n\} \right\}\]
be the first index of the maximal increment of $(T^{n}_{1}, \ldots,T^{n}_{n})$. 
Let $(Y_{i})_{i \geq 1}$ be a sequence of i.i.d.~random variables with distribution given by
\[\Pr{Y_{1} = k-m} = \frac{1}{G(\rho)} \rho^{k} \poids(k)
\qquad\text{for every}\quad k \ge 0.\]
The following result is a direct application of~\cite[Theorem 1]{AL11}, combined with~\cite[Theorem 8.1]{DDS08} (in the case $G^{(2)}(\rho)<\infty)$ and~\cite[Theorem 9.1]{DDS08}  (in the case $G^{(2)}(\rho)=\infty$).

\begin{prop}\label{prop:dTV}
We have
\[d_{\mathrm{TV}}\left( \big( Y^{n}_{1}, \ldots, Y^{n}_{V_{n}-1},Y^{n}_{V_{n}+1}, \ldots, Y^{n}_{n} \big),  \big(Y_{1}, \dots, Y_{n-1}\big) \right)  \cv 0,\]
where $d_{\mathrm{TV}}$ denotes the total variation distance on $\R^{n-1}$ equipped with the product topology.
\end{prop}

Theorem~\ref{thm:stable_marches_condensation} is now a simple consequence of this result.

\begin{proof}[Proof of Theorem~\ref{thm:stable_marches_condensation}]
Let the $Y_{i}$'s be as above; note that $\max(Y_{1}, \dots,Y_{n-1})/r_{n}$ and $(Y_{1} + \dots + Y_{n-1})/r_{n}$ both converge in distribution. Since we impose $Y_{1}^{n}+ \dots + Y^{n}_{n}=\lambda_{n}$, with $\lambda_{n}/r_{n}\to\infty$, then Proposition~\ref{prop:dTV} implies that
\[\frac{Y^{n}_{V_{n}}}{\lambda_{n}}  \cvproba 1
\qquad\text{and}\qquad
\frac{\max\{Y^{n}_i : i \ne V_{n}\}}{\lambda_{n}}  \cvproba 0.\]
Also, by exchangeability of the $Y^{n}_{k}$'s, we have $V_n \to U$ in distribution.
This yields~\eqref{eq:cv_condensation_centre}, which is equivalent to the convergence of the processes in Theorem~\ref{thm:stable_marches_condensation}.

Next, as in the proof of Theorem~\ref{thm:stable_marches_drift}, the rescaled sum $r_{n}^{-2} \sum_{k=1}^{n-1} (X_{k})^{2}$ converges in distribution as $n \to \infty$. The limit is $\sum_{t \in [0,1]} (\Delta X^{\alpha}_{t})^{2}$ in the case $\alpha \in (1,2)$, and it is constant when $\alpha=2$, precisely simply $2$ in the case of infinite variance when $G^{(2)}(\rho) = \infty$, and to $2\E[X_{1}^{2}]/\Var(X_{1})$ otherwise.
Consequently,
\[\frac{1}{(\lambda_{n})^{2}} \sum_{k \ne V_{n}} (X^{n}_{k})^{2}  \cvproba 0,\]
and the second convergence of Theorem~\ref{thm:stable_marches_condensation} follows since $X^{n}_{V_{n}}/ \lambda_{n} \to 1$ in probability.
\end{proof}

\section{Biconditioned random trees}
\label{sec:arbres}

In this section, we focus on random trees with $n$ vertices and $K_{n}$ leaves sampled proportionally to a sequence of weights. In the case of the uniform distribution, Labarbe \& Marckert~\cite{LM07} proved the convergence of the so-called contour function. We shall here only consider the {\L}ukasiewicz path, which will be sufficient to our application to random maps in the next section. It also provides some  combinatorial information, such as the largest degree for example.

In Sec.~\ref{sec:arbres_Luka} we first recall the definition of the {\L}ukasiewicz path of a plane tree, then in Sec.~\ref{sec:arbres_aleatoires} we precisely define the distribution we consider. In Sec.~\ref{sec:vervaat} we relate the random {\L}ukasiewicz paths to the nondecreasing paths studied in the preceding section. Finally in Sec.~\ref{sec:Luka} we state and prove our main results on scaling limits of these {\L}ukasiewicz paths.

\subsection{Plane trees and coding path}
\label{sec:arbres_Luka}

Recall that a rooted plane tree is a connected graph with no cycle, with a distinguished root-corner, and such that the neighbours of every vertex are ordered. We shall interpret them as genealogical trees; the vertex at the root is the ancestor, the neighbour of a vertex closer to the root is its parent, whereas the other ones are its offspring (ordered from left to right); more generally, the vertices lying between the root and a given vertex are the ancestors of the latter. Finally an individual with no child is called a leaf, the other ones are called internal vertices.

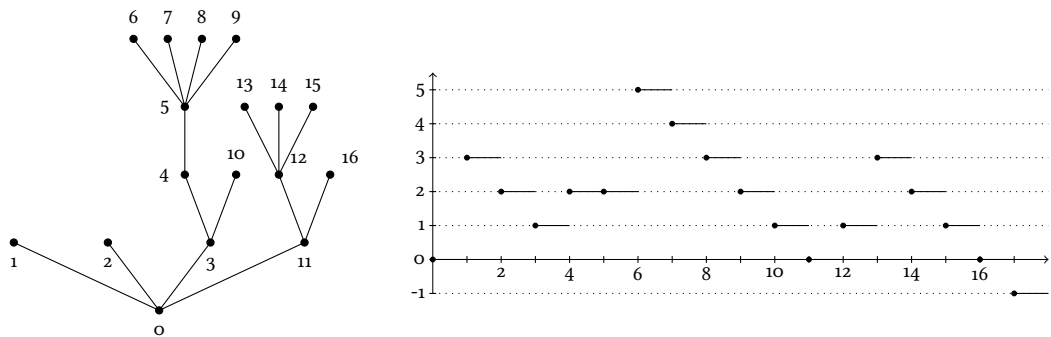
\begin{figure}[!ht]\centering
\def\longueur{2}
\def\R{.5}
\def\r{1}
\begin{tikzpicture}[scale=.45]
\coordinate (1) at (0,0*\longueur);
	\coordinate (2) at (-4.25,1*\longueur);
	\coordinate (3) at (-1.5,1*\longueur);
	\coordinate (4) at (1.5,1*\longueur);
		\coordinate (5) at (.75,2*\longueur);
			\coordinate (6) at (.75,3*\longueur);
				\coordinate (7) at (-.75,4*\longueur);
				\coordinate (8) at (.25,4*\longueur);
				\coordinate (9) at (1.25,4*\longueur);
				\coordinate (10) at (2.25,4*\longueur);
		\coordinate (11) at (2.25,2*\longueur);
	\coordinate (12) at (4.25,1*\longueur);
		\coordinate (13) at (3.5,2*\longueur);
			\coordinate (14) at (2.5,3*\longueur);
			\coordinate (15) at (3.5,3*\longueur);
			\coordinate (16) at (4.5,3*\longueur);
		\coordinate (17) at (5,2*\longueur);

\draw	(1) -- (2)	(1) -- (3)	(1) -- (4)	(1) -- (12);
\draw
	(4) -- (5)	(4) -- (11)
	(5) -- (6)
	(6) -- (7)	(6) -- (8)	(6) -- (9)	(6) -- (10)
	(12) -- (13)	(12) -- (17)
	(13) -- (14)	(13) -- (15)	(13) -- (16)
;

%
\foreach \x in {1, 2, ..., 17} \draw[fill=black] (\x) circle (3pt) ;

\begin{scriptsize}
\draw
	(1) node[below=1mm] {0}
	(2) node[below=1mm] {1}
	(3) node[below=1mm] {2}
	(4) node[below=1mm] {3}
	(5) node[left=1mm] {4}
	(6) node[left=1mm] {5}
	(7) node[above=1mm] {6}
	(8) node[above=1mm] {7}
	(9) node[above=1mm] {8}
	(10) node[above=1mm] {9}
	(11) node[above=1mm] {10}
	(12) node[below=1mm] {11}
	(13) node[above right=.5mm] {12}
	(14) node[above=1mm] {13}
	(15) node[above=1mm] {14}
	(16) node[above=1mm] {15}
	(17) node[above right=.5mm] {16}
;

\begin{scope}[shift={(8,1.5)}]
\draw[very thin, ->]	(0,0) -- (18,0);
\draw[very thin, ->]	(0,-1) -- (0,5.5);
\foreach \x in {-1, 1, 2, 3, 4, 5}
	\draw[dotted]	(0,\x) -- (18,\x);
\foreach \x in {-1, 0, 1, ..., 5}
	\draw (.1,\x)--(-.1,\x)
	(0,\x) node[left] {\x}
;
\foreach \x in {1, 2, ..., 17}
	\draw (\x,.1)--(\x,-.1);
\foreach \x in {2, 4, ..., 16}
	\draw (\x,0) node[below] {\x};

\draw[fill=black]
	(0, 0) circle (2pt) -- ++ (1,0)
	++(0,3) circle (2pt) -- ++ (1,0)
	++(0,-1) circle (2pt) -- ++ (1,0)
	++(0,-1) circle (2pt) -- ++ (1,0)
	++(0,1) circle (2pt) -- ++ (1,0)
	++(0,0) circle (2pt) -- ++ (1,0)
	++(0,3) circle (2pt) -- ++ (1,0)
	++(0,-1) circle (2pt) -- ++ (1,0)
	++(0,-1) circle (2pt) -- ++ (1,0)
	++(0,-1) circle (2pt) -- ++ (1,0)
	++(0,-1) circle (2pt) -- ++ (1,0)
	++(0,-1) circle (2pt) -- ++ (1,0)
	++(0,1) circle (2pt) -- ++ (1,0)
	++(0,2) circle (2pt) -- ++ (1,0)
	++(0,-1) circle (2pt) -- ++ (1,0)
	++(0,-1) circle (2pt) -- ++ (1,0)
	++(0,-1) circle (2pt) -- ++ (1,0)
	++(0,-1) circle (2pt) -- ++ (1,0)
;
\end{scope}
\end{scriptsize}
\end{tikzpicture}
\caption{A plane tree, with the depth-first search order indicated next to the nodes and its {\L}ukasiewicz path.}
\label{fig:foret_etiquetee}
\end{figure}

Let us recall the coding of a plane tree by a discrete path, see Figure~\ref{fig:foret_etiquetee} for an example.
Fix a tree $T$ with $n$ vertices; thanks to the planar ordering, we may list these vertices as $u_0 < u_1 < \dots < u_{n-1}$ in depth-first search order, which corresponds to the lexicographical order when trees are viewed as words. Then for each such vertex $u$, let us denote by $k_{u} \ge 0$ its offspring number. We define then the \emph{{\L}ukasiewicz path} $W(T) = (W_{i}(T) ; 0 \le i \le n)$ of the tree recursively by $W_{0}(T) = 0$ and
\[W_{i+1}(T) = W_{i}(T) + k_{u_i}-1
\qquad 0 \le i \le n-1.\]
One easily checks that $W_{n}(T) = -1$ whereas for every $0 \le i \le n-1$, we have $W_{i}(T) \ge 0$ as well as $W_{i+1}(T)-W_{i}(T) \ge -1$. 
There is a $1$-to-$1$ correspondence between the set of such paths and plane trees with $n$ vertices, see e.g.~\cite[Chapter~6]{Pit06} or~\cite{LG05}. Furthermore the leaves of the tree correspond to the negative increments of its {\L}ukasiewicz path. For every $k \in \{0, \dots, n-1\}$, we shall denote by
\[\Lambda_{k}(T) = \sum_{i = 0}^{k} \ind{W_{i+1}(T) - W_{i}(T) = -1}\]
the number of leaves amongst the first $k+1$ vertices of $T$.

\subsection{Random trees}
\label{sec:arbres_aleatoires}

We consider random trees sampled according to a sequence of weights.
For a tree $T$ and a vertex $u \in T$, recall that $k_{u} \ge 0$ denotes its number of children. Fix a sequence of nonnegative real numbers $\boldsymbol{\theta} = (\theta(i))_{i \ge 0}$ which, in order to avoid trivialities, satisfies $\theta(0) > 0$ and $\theta(i) > 0$ for at least one $i \ge 2$, and define  a measure $w^{\boldsymbol{\theta}}$ on the set of all finite trees by setting
\[w^{\boldsymbol{\theta}}(T) = \prod_{u \in T} \theta(k_{u}),\]
for every finite tree $T$. 
For $n \ge K_{n} \ge 1$ let $\Tree_{n,K_{n}}$ denote the finite set of trees with $n$ vertices amongst which $K_{n}$ are leaves. We shall always implicitly assume that $n$ and $K_{n}$ are compatible with the support of $\boldsymbol{\theta}$, in the sense that $\Tree_{n,K_{n}}$ has nonzero $w^{\boldsymbol{\theta}}$-weight. We then define a probability $\P^{\boldsymbol{\theta}}_{n,K_{n}} = w^{\boldsymbol{\theta}}(\cdot)/w^{\boldsymbol{\theta}}(\Tree_{n,K_{n}})$ on this set.

Trees with solely $n$ vertices sampled proportionally to their $w^{\boldsymbol{\theta}}$-weight are known in the literature as \emph{simply generated}. When the sequence $\boldsymbol{\theta}$ is a probability measure, they correspond to Bienaym\'e--Galton--Watson trees with offspring distribution $\boldsymbol{\theta}$ conditioned on having $n$ vertices; similarly, $\P^{\boldsymbol{\theta}}_{n,K_{n}}$ is the law of such a random tree conditioned to have $n$ vertices and $K_{n}$ leaves. A classical fact for trees conditioned only by their number of vertices is that an exponential tilting of $\boldsymbol{\theta}$ as in Lemma~\ref{lem:tilting} leaves the distribution of the random trees unaffected, see e.g. the survey of Janson~\cite[Sec.~4]{Jan12}. Although we shall not need it, let us mention that in the context of trees conditioned both by their number of vertices and leaves, one can gain a degree of freedom by only making the exponential change to the weights $(\theta(i))_{i \ge 1}$ and setting for $0$ any arbitrary value.

We next consider the random \L ukasiewicz path of such a tree.
For every $n \ge 1$, let $\mathbf{B}_{n,K_{n}}$ denote the set of paths $w = (w_{k} ; 0 \le k \le n)$ such that
\[w_{0} = 0, \qquad w_{n} = -1, \qquad w_{k} - w_{k-1} \in \Z_{\ge-1} \enskip\text{for every}\enskip 1 \le k \le n,\]
\emph{and}
\[\#\{k \in \{1, \dots, n\} : w_{k} - w_{k-1} = -1\} = K_{n}.\]
Further, let $\mathbf{W}_{n,K_{n}} \subset \mathbf{B}_{n,K_{n}}$ denote the set of such paths which in addition satisfy
\[w_{k} \ge 0 \qquad\text{for every}\qquad 1 \le k \le n-1.\]
Let us define probability measures on the sets $\mathbf{B}_{n,K_{n}}$ and $\mathbf{W}_{n,K_{n}}$, denoted respectively by $\P^{\boldsymbol{\theta}}_{n,K_{n}}$ and $\P^{\boldsymbol{\theta},+}_{n,K_{n}}$, as in Sec.~\ref{sec:def_ponts}, by normalising the measures which assign a weight
\begin{equation}\label{eq:mesure_SG_Luka}
\prod_{k=1}^{n} \theta(w_{k}-w_{k-1}+1),
\end{equation}
to every path $w$ in $\mathbf{B}_{n,K_{n}}$ and $\mathbf{W}_{n,K_{n}}$ respectively.
Then the coding of a plane tree $T$ by its \L ukasiewicz path $W(T)$ bijectively maps $\Tree_{n,K_{n}}$ onto $\mathbf{W}_{n,K_{n}}$ and the law $\P^{\boldsymbol{\theta},+}_{n,K_{n}}$ on $\mathbf{W}_{n,K_{n}}$ is the image measure of $\P^{\boldsymbol{\theta}}_{n,K_{n}}$ on $\Tree_{n,K_{n}}$.

\subsection{The Vervaat transform}
\label{sec:vervaat}

Our aim is to study the asymptotic behaviour of the preceding random \L ukasiewicz paths. In order to relax the positivity constraint, and therefore to study paths sampled from $\P^{\boldsymbol{\theta}}_{n,K_{n}}$ instead of $\P^{\boldsymbol{\theta},+}_{n,K_{n}}$, we make use of the notion of \emph{Vervaat transform}, also known as \emph{cyclic shift}, of a bivariate sequence (with respect to the first variable).
First, if $\boldsymbol{x}=(x_{k})_{1 \le  k \le  m}=(a_{k},b_{k})_{1 \le  k \le  m}$ is a sequence of integer-valued couples and $i \in \Z/m\Z$, we define the cyclic shift $\boldsymbol{x}^{(i)}$ by $x^{(i)}_k=(a_{k+i \bmod m},b_{k+i \bmod m})$ for $1 \le  k \le  m$ (where representatives modulo $m$ are chosen in $ \{1,2, \ldots,m\}$).  We then define the (discrete) Vervaat transform $ \Vcd(\boldsymbol{x})$ as follows (we introduce the superscript $\mathsf{d}$ to underline the fact that this transformation acts on discrete sequences). Let $i_{*}(\boldsymbol{x})$ be defined by
\[i_{*}(\boldsymbol{x}) = \min \left\{ j  \in \{1,2, \ldots,m\} ; a_{1}+a_{2}+\cdots+a_{j}= \min_{1 \le  i \le  m} (a_{1}+a_{2}+\cdots+a_{i}) \right\}.\]
Then $\Vcd(\boldsymbol{x}) \coloneqq \boldsymbol{x}^{(i_{*}(\boldsymbol{x}))}$.

\begin{lem}
\label{lem:vervaat}
Fix $n$ and $K_{n}$ and given $W^{n} \in \mathbf{B}_{n,K_{n}}$, for all $k \in \{0, \dots, n\}$, let
\begin{equation}\label{eq:def_Lambda_W}
\Lambda^{n}_{k}=\sum_{i=1}^{k}\ind{W^{n}_{i}-W^{n}_{i-1}=-1}
\end{equation}
denote the number of negative increments up to time $k$.
Then the following equality holds in distribution:
\[\Vcd\big((W^{n}_{ k}, \Lambda^{n}_{k})_{0 \le  k \le  n}\big) \enskip\text{under}\enskip \P^{\boldsymbol{\theta}}_{n,K_{n}}\enskip\eqloi\enskip
(W^{n}_{ k}, \Lambda^{n}_{k})_{0 \le  k \le  n} \enskip\text{under}\enskip \P^{\boldsymbol{\theta},+}_{n,K_{n}}.\]
\end{lem}

\begin{proof}This follows from a simple extension of the so-called cyclic lemma (see e.g.~\cite[Lemma~6.1]{Pit06}). Indeed, if $\boldsymbol{x}=(a_{i},b_{i})_{1 \le  i \le  m}$ are integers such that $a_{1}+a_{2}+\cdots+a_{m}=-1$, then there is a unique $j \in \Z/m\Z$ such that the cyclic shift $\boldsymbol{x}^{(j)}=(a^{(j)}_{i},b^{(j)}_{i})_{1 \le  i \le  m}$ fulfills $a^{(j)}_{1}+ \cdots +a^{(j)}_{i} \ge  0$ for every $1 \le  i<m$. It is then standard to obtain the desired results by an exchangeability argument; we leave details to the reader. 
\end{proof}

In the continuum setting, following Miermont~\cite[Definition~1]{Mie01}, we define the Vervaat transform $\mathcal V\, f$ of a function $f \in \D([0,1])$ with $f(0)=f(1-)=f(1)$ as follows:
let $t_{\min}$ be the location of the left-most minimum of $f$, that is, the smallest $t$ such
that $f(t-) \wedge f (t) = \inf\, f$, then, for $t \in [0,1)$ set 
\[\mathcal V\, f (t) =f(t + t_{\min} \bmod 1]) - \inf_{[0,1]} f,\]
and $\mathcal V\, f (1) = \lim_{t \uparrow 1} \mathcal V\, f (t)$. 
Note that  if $f \in \D([0,1])$ with $f(0)=f(1-)=f(1)$ reaches its infimum at a unique time with no jump at that time, then $ \mathcal{V}$ is continuous at $f$. This is the case of the spectrally positive stable processees with drift from Sec.~\ref{sec:def_ponts}.
\smallskip

Finally, the following simple observation is really the key idea that will allow us to transfer results from Sec.~\ref{sec:marches} to the present `biconditioned' setting. Recall from Sec.~\ref{sec:def_ponts} the law $\P^{\boldsymbol{\poids},n}_{x_{n}}$ on nondecreasing paths which end at $x_{n}$ at time $n$,  sampled from a weight sequence $\boldsymbol{\poids}$.

\begin{lem}\label{lem:separation_sauts}
Fix a weight sequence $\boldsymbol{\theta}$ and define two sequences $\boldsymbol{p}$ and $\boldsymbol{q}$ by $\poids(k) = \theta(k+1)$ for every $k\ge 0$, $q(k) = 1$ if $k \in \{0,1\}$, and $q(k) = 0$ for $k \ge 2$.
Fix $1 \leq K_{n} \leq n$. Independently, sample $(S^{n-K_{n}}_{k} ; 0 \le k \le n-K_{n})$ with distribution $\P^{\boldsymbol{p},n-K_{n}}_{K_{n}-1}$ and $(L^{n}_{k} ; 0 \le k \le n)$ with distribution $\P^{\boldsymbol{q},n}_{K_{n}}$. For every $k \in \{0, \dots, n\}$, define
\[W^{n}_{k} = S^{n-K_{n}}_{k - L^{n}_{k}} - L^{n}_{k}.\]
Then the path $W^{n}$ has  law $\P^{\boldsymbol{\theta}}_{n,K_{n}}$.
\end{lem}

The proof is straightforward and left to the reader. Intuitively speaking, this just amounts to saying that in a simply generated bridge of $\mathbf{B}_{n,K_{n}}$, the position of the $K_{n}$ negative steps are uniformly distributed amongst the $n$ different jumps and the path obtained by removing them is then a simply generated nondecreasing path which ends at $K_{n}-1$ at time $n-K_{n}$.

\subsection{Scaling limits of random \L ukasiewicz paths}
\label{sec:Luka}

We are now ready to state and prove several invariance principles for the \L ukasiewicz paths of $\boldsymbol{\theta}$-simply generated random trees with $n$ vertices and $K_{n}$ leaves. Recall from the above discussion that these paths $W^{n}$ have the law $\P^{\boldsymbol{\theta},+}_{n,K_{n}}$. In other words, they are paths starting from $0$, whose increments all belong to $\{-1, 0, 1, 2, \dots\}$, which reach at time $n$ the value $-1$ for the first time by making their $K_{n}$'th negative increments, and they are sampled at random proportionally to their weight~\eqref{eq:mesure_SG_Luka}. For every $k \in \{1, \dots, n\}$, we shall denote by $X^{n}_{k} = W^{n}_{k} - W^{n}_{k-1}$ the $k$'th increment of $W^{n}$.

We shall canonically view our discrete paths as c\`adl\`ag functions on $[0,1]$, equipped with the Skorokhod $J_{1}$ topology; recall that when the limit is a continuous path, convergence in this topology is equivalent to the uniform convergence. In the following statement, we denote by $B^{\exc}$ a standard Brownian excursion, constructed e.g. from the Brownian bridge presented in Sec.~\ref{sec:def_ponts} via the Vervaat transform as above.

We need to introduce some notation. Set
\begin{equation}\label{eq:serie_gen_poids_arbres}
F(z)=\sum_{k=0}^{\infty}  \theta(k)  z^{k}
\qquad\text{and}\qquad
A(z)=1- \frac{F(z)-F(0)}{zF'(z)}.
\end{equation}
Denote by $\rho$ the radius of convergence of   $F$ and set $k_{0}= \min \{k \ge  1:  q_{k} \neq 0\}$; then it is a simple matter to check that $A$ is increasing on $(0, \rho)$ and we extend it by continuity with $A(0) = 1-1/k_{0}$ and $A(\rho) < 1$ if and only if $F'(\rho) < \infty$.

The next result is analogous to (and indeed based on) Theorem~\ref{thm:CVmarches}.

\begin{thm}
\label{thm:CVluka}
Let $(W^{n}_{k} ; 0 \le k \le n)$ have the law $\P^{\boldsymbol{\theta},+}_{n,K_{n}}$ for every $n \ge 1$. 
Then the convergences
\[\left(\frac{1}{\sqrt{\scale_{n}}} \cdot W^{n}_{\floor{nt}} ; 0 \le  t \le  1\right) \cvloi B^{\exc}
\qquad\text{and}\qquad
\frac{1}{\scale_{n}} \sum_{k=1}^{n} (X^{n}_{k})^{2} \cvproba 1,\]
hold in each of the following cases:
\begin{thmlist}
\item\label{thm:CVluka_bulk} There exists $\tau \in  (1-1/k_{0},A(\rho))$ such that $K_{n}/n \to \tau$ and $\scale_{n} = b F^{(2)}(b) n/F'(b)$ with $b=A^{-1}(\tau)$.

\item\label{thm:CVluka_small} $K_{n}/n \to 0$ with $K_{n} \to \infty$, $\theta(1), \theta(2)>0$, and $\scale_{n} = 2 K_{n}$.

\item\label{thm:CVluka_large} $K_{n}/n \to 1$ and $n-K_{n} \to \infty$, $F$ is $\Delta$-analytic and there exist $\alpha, c>0$ such that $F(\rho-z) \sim c/z^{\alpha}$ as $z \to 0$ with $\mathrm{Re}(z)>0$,  and finally $\scale_{n} = (1+\alpha)n^{2}/(\alpha(n-K_{n}))$.
\end{thmlist}
\end{thm}

Our last main result on these random paths is analogous to (and again, based on) Theorem~\ref{thm:stable_marches} and studies specifically the case $K_{n}/n \to A(\rho)$ when the latter is finite. Here, we assume that the probability distribution with generating function $F(\rho \,\cdot)/F(\rho)$ belongs to the domain of attraction of a stable law with index $\alpha \in (1, 2]$, i.e.~that there exist $\alpha \in (1,2]$, $\rho>0$, $m>0$, and a slowly varying function $L$ such that for every $0 \le  s <  \rho$,
\[F(s)=F(\rho) \left(1-m+\frac{m s}{\rho}+\left(1-\frac{s}{\rho}\right)^{\alpha} L\left(\frac{1}{1-s/\rho}\right)\right).\]
In the case $F^{(2)}(\rho) < \infty$, i.e.~when $\alpha=2$ and $L$ admits a finite limit at infinity, say $\ell$, the law with generating function $F(\rho \,\cdot)/F(\rho)$ has a finite variance, equal to
\begin{equation}\label{eq:variance_F}
\sigma^{2} = 2\ell + m - m^{2}
= \frac{\rho^{2} F^{(2)}(\rho)}{F(\rho)} + \frac{\rho F'(\rho)}{F(\rho)} - \left(\frac{\rho F'(\rho)}{F(\rho)}\right)^{2}
.\end{equation}
In this case we set $r_{n} = \sqrt{n \sigma^{2}/2}$ for every $n\ge 1$, otherwise, we let $(r_{n})_{n \ge 1}$ be a sequence such that $r_{n}^{\alpha} \sim n L(r_{n})$ as $n \to \infty$.
We have already mentioned and used the fact that if $(S_{n})_{n\ge 1}$ is a random walk with i.i.d. increments with generating function $F(\rho \,\cdot)/F(\rho)$, then the convergence
\[\left(\frac{1}{r_{n}} \cdot \left(S_{\floor{nt}} - mnt\right) ; t \ge 0\right)  \cvloi X^{\alpha},\]
holds in distribution in $\D([0,\infty), \R)$, where $X^{\alpha}$ is the $\alpha$-stable L\'evy process whose law is characterised by $\E[\exp(-q X^{\alpha}_{t})] = \exp(t q^{\alpha})$ for every $q,t > 0$. 

Recall finally from Sec.~\ref{sec:def_ponts} that for $\lambda\in\R$, we denote by $X^{\alpha,\lambda,\br}$ the process $(X^{\alpha}_{t} - \lambda t)_{t}$ conditioned to end at $0$ at time $1$. Let us also denote by $\Delta X^{\alpha,\lambda,\br}_{t} = X^{\alpha,\lambda,\br}_{t}-X^{\alpha,\lambda,\br}_{t-}$ the value of the `jump' of $X^{\alpha,\lambda,\br}$ at time $t \in [0,1]$, which is null except for a countable set of times. We let further $X^{\alpha,\lambda,\exc}$ denote the associated excursion, e.g. obtained after applying the Vervaat transform to $X^{\alpha,\lambda,\br}$.

\begin{thm}\label{thm:CVluka_distorted}
Under the preceding assumptions, if $K_{n}$ takes the form
\[K_n = A(\rho) n + \lambda_{n},\]
where 
$\lambda_{n}/r_{n}$ converges to some $\lambda \in [-\infty,\infty]$, then three cases occur.
\begin{thmlist}
\item\label{thm:CVluka_distorted_drift} If $\lambda \in \R$, then with $\widehat{\lambda} = \lambda m / (1-A(\rho))$,
{\[\left( \frac{1}{r_{n}} \cdot W^{n}_{\floor{nt}} ; 0 \le  t \le  1\right)  \cvloi \sqrt{1 + \frac{m^{2}-m}{\sigma^{2}} \ind{F^{(2)}(\rho) < \infty}} \cdot \left(X^{\alpha,\widehat{\lambda},\exc}_{t/m} ; 0 \le  t \le  1\right).\]}
In addition, 
\[
\frac{1}{r_{n}^{2}} \sum_{k=1}^{n} (X^{n}_{k})^{2} \cvloi
\begin{dcases}
\frac{2 F(\rho)}{\rho F'(\rho)} \left(1 + \frac{m^{2}-m}{\sigma^{2}} \ind{F^{(2)}(\rho) < \infty}\right) &  \text{when } \alpha=2, \\
\sum_{t \in [0,1]} (\Delta X^{\alpha,\widehat{\lambda},\exc}_{t/m})^{2} &  \text{when } \alpha \in (1,2).
\end{dcases}
\]
\item\label{thm:CVluka_distorted_gaussien} If $\lambda = -\infty$ and if $ \lambda_{n}/n \to 0$, then with $\varepsilon_{n}$ such that $A(\rho(1-\varepsilon_{n})) = K_{n}/n$, we have
\[\left(\sqrt{\frac{\varepsilon_{n}}{|\lambda_{n}|}} \cdot W^{n}_{\floor{nt}} ; 0 \le  t \le  1\right)  \cvloi \sqrt{C} \cdot B^{\exc}
\qquad\text{and}\qquad
\frac{\varepsilon_{n}}{|\lambda_{n}|} \sum_{k=1}^{n} (X^{n}_{k})^{2} \cvproba C,\]
where
\[C = \frac{\alpha-1}{1-A(\rho)} \ind{F^{(2)}(\rho) = \infty} + \frac{2\ell}{2\ell(1-A(\rho)) - m A(\rho)} \ind{F^{(2)}(\rho) < \infty}.\]

\item\label{thm:CVluka_distorted_condensation} If $\lambda = \infty$, we further assume that $\theta(k) \rho^{-k}$ is regularly varying at infinity with some index $-\beta < -2$. 
If $F^{(2)}(\rho)<\infty$ (then $\beta\ge3$), we assume furthermore that there exists $c>(\beta-3)/\sigma^{2}$ such that $\lambda_{n} \geq \sqrt{c \ln(n)}$. 
Then
\[ \left( \frac{1}{\lambda_{n}}  W^{n}_{\max\{0, \floor{nt}\}}  ; -1 \le  t \le  1\right)  \cvloi ((1-t) \ind{t>0}; -1 \leq t \leq 1).\]
In addition,
\[ \frac{\max \{X^{n}_{1}, \ldots, X^{n}_{n}\}^{2} }{\sum_{k=1}^{n} (X^{n}_{k})^{2} }  \cvproba1.\]
\end{thmlist}
\end{thm}

Let us note that in the last regime, we extend both sides by $0$ on the interval $[-1, 0)$; this is only because the function $t\mapsto (1-t) \ind{t>0}$ on $[0,1]$ is not c\`adl\`ag. Also, by combining both results, we deduce that $\max \{X^{n}_{1}, \ldots, X^{n}_{n}\} / \lambda_{n} \to 1$ in probability.

\begin{rem}\label{rem:var_finie_Luka}
In the case $F^{(2)}(\rho) < \infty$, one can check that in the two cases $\lambda \in \R$, or $\lambda = -\infty$ and  $\lambda_{n}/n \to 0$, we have the following unified convergences:
\[\left( \frac{1}{\sqrt{n}}\, W^{n}_{\floor{nt}} ; 0 \le  t \le  1\right)  \cvloi \sqrt{\frac{\rho F^{(2)}(\rho)}{F'(\rho)}} \, B^{\exc}
\quad\text{and}\quad
\frac{1}{n} \sum_{k=1}^{n} (X^{n}_{k})^{2} \cvproba \frac{\rho F^{(2)}(\rho)}{F'(\rho)}.\]
In other words, Theorem~\ref{thm:CVluka_bulk} also holds for $\tau=A(\rho)$ in this finite variance regime.
\end{rem}

In order to lighten the notation, in the proof of these theorems, we assume that all discrete-time processes are extended to continuous time by assigning at any time $t$ the same value as at $\floor{t}$, so we write e.g. $W_{nt}$ instead of $W_{\floor{nt}}$.

\begin{proof}[Proof of Theorem~\ref{thm:CVluka}]
Recall that the Brownian excursion $B^{\exc}$ can be constructed as the Vervaat transform $ \mathcal{V} B^{\br}$ of the Brownian bridge. Since $B^{\br}$ attains its infimum almost surely at a unique time, by standard continuity properties of the Vervaat transform, it is enough to show that the scaling limits of paths sampled according to $\P^{\boldsymbol{\theta}}_{n,K_{n}}$ is $B^{\br}$.

\textsc{Generalities.}
We let henceforth $W^{n}$ denote a random path sampled from $\P^{\boldsymbol{\theta}}_{n,K_{n}}$ and we let $\Lambda^{n}$ be as in~\eqref{eq:def_Lambda_W}.
Finally, introduce $(S^{n-K_{n}}_{k} ; 0 \le k \le n-K_{n})$  and $(L^{n}_{k}; 0 \leq k \leq n)$ as in Lemma~\ref{lem:separation_sauts}, so that the path $(W^{n}_{nt} ; 0 \le t \le 1)$ has the same law as the process defined for every $t \in [0,1]$ by
\begin{align}
S^{n-K_{n}}_{nt - L^{n}_{nt}} - L^{n}_{nt}
&= \left(S^{n-K_{n}}_{nt - L^{n}_{nt}} - K_{n} \frac{nt - L^{n}_{nt}}{n-K_{n}}\right) - \left(L^{n}_{nt} - K_{n}t\right) - K_{n} \left(t - \frac{nt - L^{n}_{nt}}{n-K_{n}}\right)
\nonumber
\\
&= \left(S^{n-K_{n}}_{nt - L^{n}_{nt}} - K_{n} \frac{nt - L^{n}_{nt}}{n-K_{n}}\right) - \frac{n}{n-K_{n}} \left(L^{n}_{nt} - K_{n}t\right)
\label{eq:decomposition_W}
.\end{align}
The process $L^{n}$ fits in our framework, but it also has already been studied in the literature. Indeed, it is a simple example of an \emph{urn process} studied in~\cite[Sec.~20]{Aldous:Saint_Flour}. By Theorem~20.7 in this reference, we have the convergence of the normalised process
\begin{equation}\label{eq:cv_Lambda_urne_Brownien}
\left(\frac{L^{n}_{nt} - K_{n} t}{\sqrt{K_{n} (n-K_{n})/n}} ; 0 \le t \le 1\right) \cvloi B^{\br}.
\end{equation}
Note that $\sqrt{K_{n} (n-K_{n})/n} \le \sqrt{n-K_{n}} = o(n-K_{n})$ since we assume that $n-K_{n} \to \infty$. Consequently
\begin{equation}\label{eq:cv_Lambda_urne_identite}
\left(\frac{nt - L^{n}_{nt}}{n-K_{n}} ; 0 \le t \le 1\right) \cvproba (t ; 0 \le t \le 1).
\end{equation}
It remains to study $S^{n-K_{n}}$. To this end, recall the definition of $F$ and $A$ from~\eqref{eq:serie_gen_poids_arbres} and let $\boldsymbol{\poids} = (\theta(k+1))_{k\ge 0}$. If $G$ and $\Psi$ are defined by~\eqref{eq:serie_gen_poids_marches}, then for every $s \in (0,\rho)$, we have
\begin{equation}\label{eq:F_A_G_Gamma}
F(s) = F(0) + s G(s)
\qquad\text{and}\qquad
A(s) = \frac{\Psi(s)}{1+\Psi(s)}
.\end{equation}

\textsc{Regime i.}
If $\boldsymbol{\theta}$ and $K_{n}$ satisfy the assumptions of Theorem~\ref{thm:CVluka_bulk}, then, according to Theorem~\ref{thm:gencrit_marches}, it holds
\[\left(\frac{1}{\sqrt{(n-K_{n}) \sigma^{2}_{b}}} \left(S^{n-K_{n}}_{(n-K_{n}) t} - K_{n} t\right) ; 0 \le t \le 1\right)  \cvloi B^{\br},\]
with $\Psi(b) = bG'(b)/G(b) = \lim_{n} K_{n}/(n-K_{n}) = \tau/(1-\tau)$ and $\sigma^{2}_{b} = \frac{b^{2} G^{(2)}(b)}{G(b)} + \Psi(b) - \Psi(b)^{2}$, where we recall that $G$ is related to $F$ by~\eqref{eq:F_A_G_Gamma}.
Combined with~\eqref{eq:cv_Lambda_urne_identite} and standard properties of the Skorokhod topology (see e.g.~\cite[Chapter~VI. Theorem~1.14]{JS03}), we may replace $(n-K_{n})t$ by $nt-L^{n}_{nt}$, so this gives the asymptotic behaviour of the first term in~\eqref{eq:decomposition_W}. Jointly with~\eqref{eq:cv_Lambda_urne_Brownien}, and since $K_{n}/n \to\tau$, this implies that, with $B^{\br,1}$ and $B^{\br,2}$ two independent Brownian bridges, we have, as processes,
\begin{eqnarray*}
\frac{1}{\sqrt{n}} \left(S^{n-K_{n}}_{nt - L^{n}_{nt}} - L^{n}_{nt}\right)
&=& \frac{1}{\sqrt{n}} \left(S^{n-K_{n}}_{nt - L^{n}_{nt}} - K_{n} \frac{nt - L^{n}_{nt}}{n-K_{n}}\right) - \sqrt{\frac{n}{(n-K_{n})^{2}}} \left(L^{n}_{nt} - K_{n}t\right)
\\
&\displaystyle \mathop{\longrightarrow}^{(d)}_{n \to \infty}& \sqrt{(1-\tau) \sigma^{2}_{b}}\cdot B^{\br,1} - \sqrt{\frac{\tau}{1-\tau}}\cdot B^{\br,2}
\\
&\displaystyle \mathop{=}^{(d)}& \sqrt{(1-\tau) \sigma^{2}_{b} + \frac{\tau}{1-\tau}} \cdot B^{\br}
.\end{eqnarray*}
Notice that $\tau = \Psi(b)/(1+\Psi(b))$, it is then a simple matter to check that
\[(1-\tau) \sigma^{2}_{b} + \frac{\tau}{1-\tau} = \frac{b^{2} G^{(2)}(b) + 2bG'(b)}{G(b) + bG'(b)} = \frac{b F^{(2)}(b)}{F'(b)}.\]
The first claim in Theorem~\ref{thm:CVluka_bulk} finally follows from Lemma~\ref{lem:separation_sauts} and the Vervaat transform. Similarly, recall that the second claim is about the limit of $n^{-1} \sum_{k=1}^{n} (X^{n}_{k})^{2}$ under $\P^{\boldsymbol{\theta},+}_{n,K_{n}}$; since this quantity is invariant under cyclic shift, it has the same law under $\P^{\boldsymbol{\theta}}_{n,K_{n}}$. Notice that the total contribution of the negative increments to the sum is $K_{n}$, then we infer from Lemma~\ref{lem:separation_sauts} and Theorem~\ref{thm:gencrit_marches} that
\[n^{-1} \sum_{k=1}^{n} (X^{n}_{k})^{2} \cvproba (1-\tau) \frac{b^{2} G^{(2)}(b) + bG'(b)}{G(b)} + \tau
= \frac{b F^{(2)}(b)}{F'(b)}.\]

\textsc{Regime ii.}
Theorem~\ref{thm:CVluka_small} can be similarly obtained: we may now combine  Theorem~\ref{thm:gencrit0_marches} with~\eqref{eq:cv_Lambda_urne_identite} and~\cite[Chapter~VI. Theorem~1.14]{JS03} to  deduce the convergence in distribution
\[\left(\frac{1}{\sqrt{K_{n}}} \left(S^{n-K_{n}}_{nt - K_{n}t} - K_{n} t\right); 0 \le t \le 1\right)  \cvloi B^{\br},\]
which, together with~\eqref{eq:decomposition_W},~\eqref{eq:cv_Lambda_urne_Brownien},~\eqref{eq:cv_Lambda_urne_identite}, and the fact that $K_{n}/n \to 0$, yields
\[\left(\frac{1}{\sqrt{K_{n}}} \left(S^{n-K_{n}}_{nt - L_{nt}} - L^{n}_{nt}\right) ; 0 \le t \le 1\right)
\cvloi B^{\br,1} - B^{\br,2}
\eqloi \sqrt{2} \cdot B^{\br}
.\]
Furthermore, since the total contribution of the negative increments to $\sum_{k=1}^{n} (X^{n}_{k})^{2}$ equals $K_{n}$, then Theorem~\ref{thm:gencrit0_marches} yields
\[\frac{1}{K_{n}} \sum_{k=1}^{n} (X^{n}_{k})^{2} \cvproba 1+1=2.\]

\textsc{Regime iii.}
As for Theorem~\ref{thm:CVluka_large}, we combine Theorem~\ref{thm:nongenneg_marches} with~\eqref{eq:cv_Lambda_urne_identite} and~\cite[Chapter~VI. Theorem~1.14]{JS03} to obtain
\[\left(\sqrt{\frac{n-K_{n}}{K_{n}^{2}}} \left(S^{n-K_{n}}_{nt - K_{n}t} - K_{n} t\right) ; 0 \le  t \le  1\right)  \cvloi \frac{1}{\sqrt{\alpha}} \cdot B^{\br},\]
which, together with~\eqref{eq:decomposition_W},~\eqref{eq:cv_Lambda_urne_Brownien},~\eqref{eq:cv_Lambda_urne_identite}, and the fact that $K_{n}/n \to 1$, yields
\[\left(\sqrt{\frac{n-K_{n}}{K_{n}^{2}}} \left(S^{n-K_{n}}_{nt - L_{nt}} - L^{n}_{nt}\right) ; 0 \le t \le 1\right)
\cvloi \frac{1}{\sqrt{\alpha}} \cdot B^{\br,1} - B^{\br,2}
\eqloi \sqrt{\frac{\alpha+1}{\alpha}} \cdot B^{\br}
.\]
And finally, since the total contribution of the negative increments to $\sum_{k=1}^{n} (X^{n}_{k})^{2}$ equals $K_{n} = o(K_{n}^{2}/(n-K_{n}))$, then by Theorem~\ref{thm:nongenneg_marches},
\[\frac{n-K_{n}}{K_{n}^{2}} \sum_{k=1}^{n} (X^{n}_{k})^{2} \cvproba \frac{\alpha+1}{\alpha}.\]
In both cases, the conclusion follows from Lemma~\ref{lem:separation_sauts} and the Vervaat transform.
\end{proof}

The same argument applies also to prove Theorem~\ref{thm:CVluka_distorted}. We shall use the same notation.

\begin{proof}[{Proof of Theorem~\ref{thm:CVluka_distorted}}]
Recall from Sec.~\ref{sec:def_ponts} the construction of a bridge $X^{\br}$ of $X$, a L\'evy process with linear drift, we denote by $X^{\exc} = \mathcal{V} X^{\br}$ the associated excursion. The L\'evy bridges which appear here are those of stable processes with a drift, their infimum is almost surely attained at a unique time and with no jump at that time. 
This can be easily derived from the absolute continuity relation~\eqref{eq:lawbridgestable}. See also~\cite{Kni96} for more general results.
Therefore, by standard continuity properties of the Vervaat transform, the scaled distributional convergence of a random path under $\P^{\boldsymbol{\theta},+}_{n,K_{n}}$ towards $X^{\exc}$ will follow from the scaled distributional convergence of the path under $\P^{\boldsymbol{\theta}}_{n,K_{n}}$ towards $X^{\br}$.

\textsc{Generalities.}
By the assumption on $F$, the function $G$ defined by~\eqref{eq:F_A_G_Gamma} is such that for every $s \in (0,1)$:
\[\frac{G(\rho s)}{G(\rho)} = 1 - \frac{L(1)}{m-L(1)} (1-s) + (1-s)^{\alpha} \widehat{L}\left(\frac{1}{1-s}\right),\]
where
\[\widehat{L}(x) = \frac{L(x) - L(1) x^{-(2-\alpha)}}{\frac{x-1}{x} (m-L(1))}.\]
The function $\widehat{L}$ is slowly varying at infinity so $G$ satisfies the assumption of Theorem~\ref{thm:stable_marches}. Note that
\[A(\rho) = \frac{L(1)}{m}
\qquad\text{and}\qquad
\frac{L(1)}{m-L(1)} = \frac{\rho G'(\rho)}{G(\rho)} = \Psi(\rho) = \frac{A(\rho)}{1-A(\rho)}.\]
Also, if $\alpha = 2$ and $L$ tends to $\ell < \infty$, then $\lim_{\infty} \widehat{L} = (\ell - L(1))/(m-L(1)) \eqqcolon \widehat{\ell}$, otherwise $\widehat{L}(x) \sim L(x)/(m-L(1))$ as $x\to\infty$.
Recall that in the former case, we let $r_{n} = \sqrt{n\sigma^{2}/2}$, where $\sigma^{2} = 2\ell + m - m^{2}$, whereas in the second case, $r_{n}$ is such that $r_{n}^{\alpha} \sim n L(r_{n})$. Define similarly $\widehat{r}_{n} = \sqrt{n\widehat{\sigma}^{2}/2}$ with $\widehat{\sigma}^{2} = 2\widehat{\ell} + \Psi(\rho) - \Psi(\rho)^{2}$ in the first case, and let $\widehat{r}_{n}$ be such that $\widehat{r}_{n}^{\alpha} \sim n \widehat{L}(\widehat{r}_{n})$ in the second case.
Since $K_{n}/n \to A(\rho)$ and $\widehat{\ell} + \Psi(\rho) = \ell/(m-L(1))$,
 if $\alpha = 2$ and $L$ tends to $\ell < \infty$, then as $n\to\infty$,
\[\frac{\widehat{r}_{n-K_{n}}}{r_{n}}
\cv \sqrt{\frac{(1-A(\rho)) \widehat{\sigma}^{2}}{\sigma^{2}}}
= \sqrt{\frac{1}{\sigma^{2}} \left(\frac{2\widehat{\ell} + 2\Psi(\rho)}{1+\Psi(\rho)} - \Psi(\rho)\right)}
= \sqrt{\frac{1}{\sigma^{2}} \left(\frac{2\ell}{m} - \Psi(\rho)\right)}
,\]
and otherwise
\[\widehat{r}_{n-K_{n}} 
\equi \frac{1}{m^{1/\alpha}} r_{n}.\]
Finally, let
\[\widehat{\lambda}_{n-K_{n}} = {K_{n} - \Psi(\rho) (n-K_{n})}
= \frac{K_{n} - A(\rho) n}{1-A(\rho) }
= \frac{\lambda_{n}}{1-A(\rho) }
.\]

\textsc{Regime i.}
Assume that $K_{n} = A(\rho)n + \lambda_{n}$ with $\lambda_{n}/r_{n} \to \lambda \in \R$ and define $\widehat{\lambda} = \lim_{n} \widehat{\lambda}_{n-K_{n}}/\widehat{r}_{n-K_{n}}$,
then Theorem~\ref{thm:stable_marches_drift} implies
\[\left(\frac{1}{\widehat{r}_{n-K_{n}}} \left(S^{n-K_{n}}_{nt-K_{n}t} - K_{n} t\right) ; 0 \le  t \le  1\right)  \cvloi X^{\alpha,\widehat{\lambda},\br}.\]
Suppose first that $\alpha=2$ and that the function $L$ has a finite limit; in this case, the scaling satisfies $r_{n} \sim \sqrt{n \sigma^{2}/2}$ and the preceding convergence, combined with~\eqref{eq:decomposition_W},~\eqref{eq:cv_Lambda_urne_Brownien},~\eqref{eq:cv_Lambda_urne_identite}, and the fact that $K_{n}/n \to A(\rho)$, yields
\[\left(\frac{1}{r_{n}} \left(S^{n-K_{n}}_{nt - L_{nt}} - L^{n}_{nt}\right) ; 0 \le t \le 1\right)
\cvloi \sqrt{\frac{1}{\sigma^{2}} \left(\frac{2\ell}{m} - \Psi(\rho)\right)} \cdot X^{2,\widehat{\lambda},\br} - \sqrt{\frac{2 \Psi(\rho)}{\sigma^{2}}} \cdot B^{\br}
,\]
where $X^{2,\widehat{\lambda},\br}$ and $B^{\br}$ are independent.
Recall that $X^{2,\widehat{\lambda},\br}$ has the same law as $\sqrt{2} \cdot B^{\br}$, so the right-hand side above as the same law as
\[
\sqrt{\frac{2\ell}{\sigma^{2}}} \cdot X^{2,\widehat{\lambda},\br}
= \sqrt{\frac{\sigma^{2}-m+m^{2}}{\sigma^{2}}} \cdot X^{2,\widehat{\lambda},\br}
.\]
Let us next turn to the limit of $r_{n}^{-2} \sum_{k=1}^{n} (X^{n}_{k})^{2}$, still in the regime $\alpha=2$ and when $L$ tends to $\ell$. Note that
\begin{align*}
\widehat{\sigma}^{2} 
&= 2 \widehat{\ell} + \Psi(\rho) - \Psi(\rho)^{2}
\\
&= 2 \left(\frac{\ell}{m} (\Psi(\rho)+1) - \Psi(\rho)\right) + \Psi(\rho) - \Psi(\rho)^{2}
\\
&= \left(\frac{2\ell}{m} - \Psi(\rho)\right) \left(\Psi(\rho)+1\right)
.\end{align*}
Since the total contribution of the negative increments to the sum is $K_{n} \sim A(\rho) n$, then we infer from Lemma~\ref{lem:separation_sauts} and Theorem~\ref{thm:stable_marches_drift} that
\[\frac{1}{r_{n}^{2}} \sum_{k=1}^{n} (X^{n}_{k})^{2} 
\cvproba \frac{1}{\sigma^{2}} \left(\frac{2\ell}{m} - \Psi(\rho)\right) \left(2+\frac{2\Psi(\rho)^{2}}{\widehat{\sigma}^{2}}\right) + \frac{2 A(\rho)}{\sigma^{2}}
= \frac{4\ell}{m \sigma^{2}}
= \frac{2 \rho F^{(2)}(\rho)}{F'(\rho) \sigma^{2}}
.\]

The cases $\alpha=2$ and $\lim_{\infty} L = \infty$ as well as $\alpha \in(1,2)$ are similar, except that now $\sqrt{n} = o(r_{n})$, thus the contribution of $L^{n}_{nt} - K_{n} t$ in~\eqref{eq:decomposition_W} is negligible and we simply obtain from Theorem~\ref{thm:stable_marches_drift}
\[\left(\frac{1}{r_{n}} \left(S^{n-K_{n}}_{nt - L_{nt}} - L^{n}_{nt}\right) ; 0 \le t \le 1\right) 
\cvloi m^{-1/\alpha} X^{\alpha,\widehat{\lambda},\br}
.\]
Note that $m^{-1/\alpha} \widehat{\lambda} = \lambda / (1-A(\rho))$; we infer from the scaling property of $X^{\alpha}$ that the right-hand side has the same law as
the process $(X^{\alpha}_{t/m} - \frac{\lambda m}{1-A(\rho)} \frac{t}{m} ; 0 \le t \le 1)$ conditioned to be at $0$ at time $1$.
The convergence of $r_{n}^{-2} \sum_{k=1}^{n} (X^{n}_{k})^{2}$ again follows from Theorem~\ref{thm:stable_marches_drift}: in the case $\alpha=2$, we directly read
\[\frac{1}{r_{n}^{2}} \sum_{k=1}^{n} (X^{n}_{k})^{2} 
\cvproba \frac{2}{m}
= \frac{2 F(\rho)}{\rho F'(\rho)},\]
whereas in the case $\alpha \in (1,2)$, 
\[\frac{1}{r_{n}^{2}} \sum_{k=1}^{n} (X^{n}_{k})^{2} 
\cvloi m^{-2/\alpha} \sum_{t \in [0,1]} (\Delta X^{\alpha,\widehat{\lambda}}_{t})^{2}
\eqloi \sum_{t \in [0,1]} (\Delta X^{\alpha,\lambda,m}_{t/m})^{2}.\]
In any case we conclude the proof of Theorem~\ref{thm:CVluka_distorted_drift} by applying Lemma~\ref{lem:separation_sauts} and the Vervaat transform.

\smallskip
\textsc{Regime ii.}
Next, in Theorem~\ref{thm:CVluka_distorted_gaussien}, recall that we assume $\lambda_{n}/r_{n} \to -\infty$ with $\lambda_{n}/n \to 0$, and we take $\varepsilon_{n}$ such that $A(\rho(1-\varepsilon_{n})) = K_{n}/n$.
Observe that $\varepsilon_{n}$ satisfies $\Psi(\rho(1-\varepsilon_{n})) = K_{n}/(n-K_{n})$; we infer from Theorem~\ref{thm:stable_marches_gaussien} that
\[\left(\sqrt{\frac{\varepsilon_{n}}{|\widehat{\lambda}_{n-K_{n}}| }} \cdot \left(S^{n-K_{n}}_{nt} - K_{n} t\right) ; 0 \le  t \le  1\right)  \cvloi \sqrt{\alpha-1} \cdot B^{\br}.\]
Observe that $\widehat{\lambda}_{n-K_{n}} \sim \lambda_{n} / (1-A(\rho))$ and recall from~\eqref{eq:varfini} that, when both $\alpha=2$ and $L$ has a finite limit, then so does $\widehat{L}$ and so
\[\frac{|\widehat{\lambda}_{n-K_{n}}|}{(n-K_{n}) \varepsilon_{n}} \cv \widehat{\sigma}^{2}
.\]
Combined with~\eqref{eq:decomposition_W},~\eqref{eq:cv_Lambda_urne_Brownien},~\eqref{eq:cv_Lambda_urne_identite}, we infer in this case that, with $B^{\br,1}$ and $B^{\br,2}$ two independent Brownian bridges,
\begin{eqnarray*}
\left(\sqrt{\frac{(1-A(\rho)) \varepsilon_{n}}{|\lambda_{n}|}} \left(S^{n-K_{n}}_{nt - L_{nt}} - L^{n}_{nt}\right) ; 0 \le t \le 1\right)
&\displaystyle \mathop{\longrightarrow}^{(d)}_{n \to \infty}& B^{\br,1} - \sqrt{\frac{\Psi(\rho)}{(1-A(\rho)) \widehat{\sigma}^{2}}} \cdot B^{\br,2}
\\
&\displaystyle \mathop{=}^{(d)}& \sqrt{1+\frac{\Psi(\rho)}{\frac{2\ell}{m} - \Psi(\rho)}} \cdot B^{\br}
\\
&=&\sqrt{\frac{2\ell}{2\ell - m \Psi(\rho)}} \cdot B^{\br}
.\end{eqnarray*}
As for $r_{n}^{-2} \sum_{k=1}^{n} (X^{n}_{k})^{2}$, still in the regime $\alpha=2$ and when $L$ tends to $\ell$, since the total contribution of the negative increments to the sum is $K_{n} \sim A(\rho) n$, then we infer from Lemma~\ref{lem:separation_sauts} and Theorem~\ref{thm:stable_marches_gaussien} and~\eqref{eq:varfini} that
\[\frac{(1-A(\rho)) \varepsilon_{n}}{|\lambda_{n}|} \sum_{k=1}^{n} (X^{n}_{k})^{2} 
\cvproba 1 + \frac{\Psi(\rho)^{2}}{\widehat{\sigma}^{2}} + \frac{\Psi(\rho)}{\widehat{\sigma}^{2}}
= \frac{2(\widehat{\ell} + \Psi(\rho))}{\widehat{\sigma}^{2}}
= \frac{2\ell}{2\ell - m \Psi(\rho)}
.\]

Now when either $\alpha \in (1,2)$ or $L\to\infty$, (so $\widehat{L} \to \infty$), recall from~\eqref{eq:equiv_scaling_distorted-new} that
\[\frac{|\lambda_{n}|}{(1-A(\rho)^{2} n \varepsilon_{n}} 
\equi \frac{|\widehat{\lambda}_{n-K_{n}}|}{(n-K_{n}) \varepsilon_{n}} 
\equi \alpha \varepsilon_{n}^{\alpha-2} \widehat{L}(1/\varepsilon_{n}),\]
which tends to infinity. Therefore the contribution of $L^{n}_{nt} - K_{n} t$, which is of order $\sqrt{n}$, is again negligible and we simply obtain
\[\left(\sqrt{\frac{\varepsilon_{n}}{|\lambda_{n}|}} \left(S^{n-K_{n}}_{nt - L_{nt}} - L^{n}_{nt}\right) ; 0 \le t \le 1\right)
\cvloi \sqrt{\frac{\alpha-1}{1-A(\rho)}} \cdot B^{\br}
.\]
Further, we read from Theorem~\ref{thm:stable_marches_gaussien} that
\[\frac{\varepsilon_{n}}{|\widehat{\lambda}_{n-K_{n}}|} \sum_{k=1}^{n} (X^{n}_{k})^{2} \cvproba \alpha-1,\]
and since now $K_{n}$ is small compared to $\varepsilon_{n} / (|\lambda_{n}|)$, then
\[\frac{\varepsilon_{n}}{|\lambda_{n}|} \sum_{k=1}^{n} (X^{n}_{k})^{2} 
\cvproba \frac{\alpha-1}{1-A(\rho)}
.\]
In any case we conclude the proof of Theorem~\ref{thm:CVluka_distorted_gaussien} by applying Lemma~\ref{lem:separation_sauts} and the Vervaat transform.

\smallskip
\textsc{Regime iii.}
For the last regime, recall that $\widehat{\lambda}_{n-K_{n}} \sim \lambda_{n} / (1-A(\rho))$ is large compared to $\sqrt{n}$, so the contribution of $L^{n}$ in the decomposition of $W^{n}$ is negligible, and we infer from Theorem~\ref{thm:stable_marches_condensation} and Lemma~\ref{lem:separation_sauts} that, under $\P^{\boldsymbol{\theta}}_{n,K_{n}}$,
\[\left( \frac{1-A(\rho)}{\lambda_{n}} \cdot \left(W^{n}_{nt} - K_{n} t\right)  ; 0 \le  t \le  1\right)  \cvloi (\ind{U \leq t} - t; 0 \leq t \leq 1),\]
Then the first claim in Theorem~\ref{thm:CVluka_distorted_condensation} follows from the Vervaat transform.
As for the second claim, we also deduce from the proof of Theorem~\ref{thm:stable_marches_condensation} that if we let $U_n \coloneqq \inf \{1\leq j \leq n : X^{n}_j=\max\{|X^{n}_i| : 1\leq i \leq n\}\}$, then
\[\frac{X^{n}_{U_{n}}}{\widehat{\lambda}_{n-K_{n}} }  \cvproba 1,
\qquad\text{and}\qquad
\frac{1}{(\widehat{\lambda}_{n-K_{n}} )^{2}} \sum_{k \ne U_{n}} (X^{n}_{k})^{2}  \cvproba 0,\]
and the second claim of Theorem~\ref{thm:CVluka_distorted_condensation} follows.
\end{proof}

\section{Random planar maps}
\label{sec:cartes}

This last section is devoted to the study of random biconditioned Boltzmann planar maps. Let us start in Sec.~\ref{sec:definition} by precisely defining the model. 
We then present in Sec.~\ref{sec:bijection} the bijection between maps and labelled trees and describe the law of the random trees which code Boltzmann maps.
Finally our main results are stated and proved in Sec.~\ref{sec:results} by relying on the results from the preceding section.
Figure~\ref{tab:bigpicture} summarises our results in the different regimes.

\begin{figure}[!ht] \centering
\includegraphics[width=\textwidth]{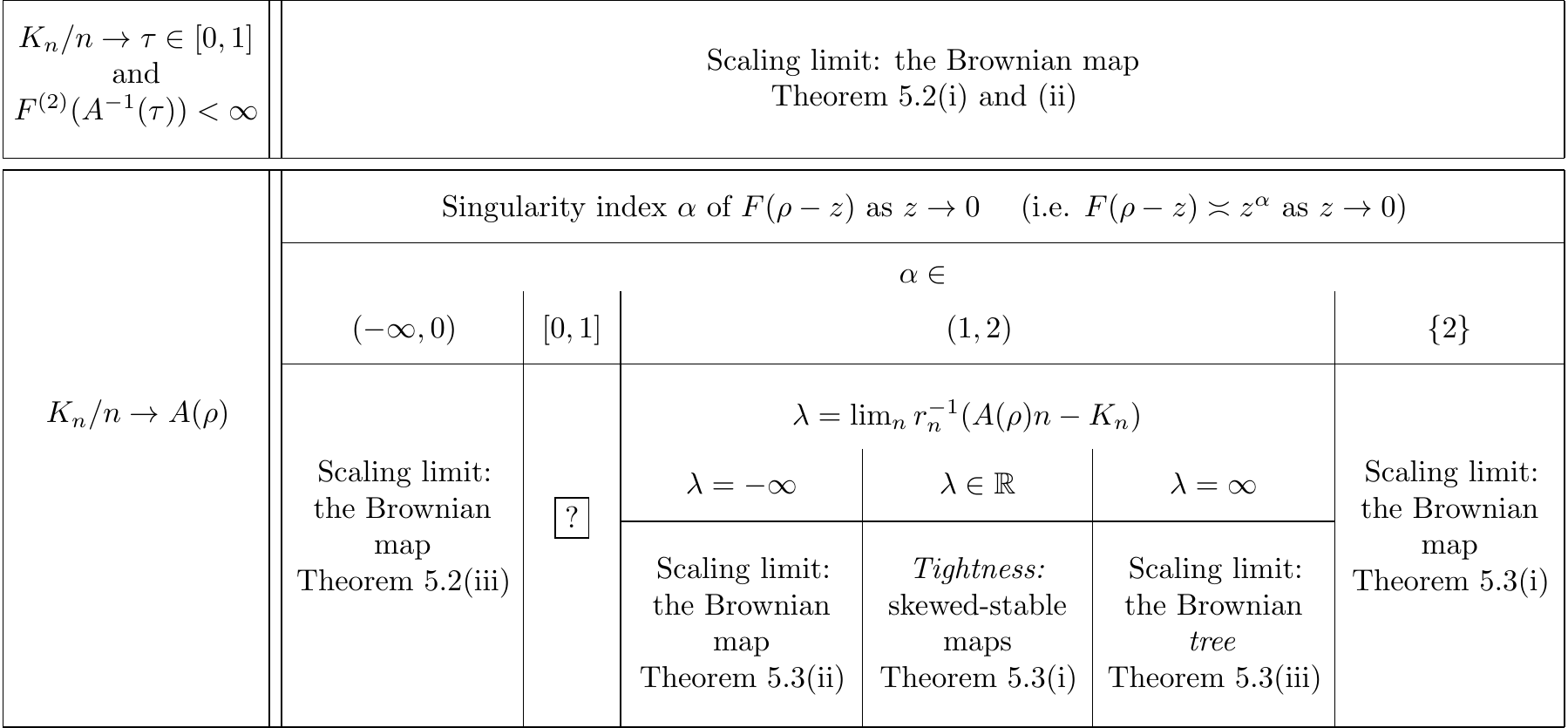}
\caption{The big picture for scaling limits of biconditioned bipartite Boltzmann planar maps with weight sequence $\q$, conditioned to have $n-1$ edges and $K_{n}+1$ vertices (and $n-K_{n}$ faces). 
The generating function $F$, with radius of convergence $\rho$, as well as the function $A$ are defined in~\eqref{eq:serie_gen_poids}. For $\alpha \in (1,2)$, the sequence $r_{n}$ is of order $n^{1/\alpha}$.}
\label{tab:bigpicture}
\end{figure}

\subsection{Biconditioned Boltzmann maps}
\label{sec:definition}

A natural generalisation of the model of random quadrangulations is that of Boltzmann random maps which allows the face degrees to be random, introduced in~\cite{MM07}; it also naturally appears when considering random maps coupled with statistical physics models, see e.g.~\cite[Sec.~8]{LGM11}. Specifically, denote by $\Map$ the set of all finite rooted bipartite planar maps, in which all faces have even degree. Let us fix a sequence of nonnegative real numbers $\q = (q_i ; i \ge 1)$ which, in order to avoid trivialities, satisfies $q_i > 0$ for at least one $i \ge 2$. We define a measure $w^{\q}$ on $\Map$ by setting
\[w^{\q}(M) = \prod_{f \text{ face}} q_{\mathrm{deg}(f)/2},
\qquad M \in \Map,\]
where $\mathrm{deg}(f)$ is the degree of the face $f$. When $w^{\q}(\Map) = \sum_{M \in \Map} w^{\q}(M)<\infty$ (one says that $\q$ is \emph{admissible}), one can sample a planar map $\map^{\q}$ at random proportionally to its $w^{\q}$-weight. The size of  $\map^{\q}$  is also random but one can condition $\map^{\q}$  to have a given number $n$ either of vertices, or of edges, or of faces and study its asymptotic behaviour as $n \to \infty$. 

Le~Gall~\cite{LG13}, based on~\cite{MM07}, proved that if, loosely speaking, the weight sequence $\q$ is such that the degree of a typical face of $\map^{\q}$  has finite exponential moments, then $\map^{\q}$, conditioned to have $n$ vertices, and rescaled by a factor of order $n^{-1/4}$ converges in distribution to the Brownian map. 
On the other hand, when the weight sequence $\q$ is such that, roughly speaking, a typical face of $\map^{\q}$ has degree $2k$ with probability of order $c k^{-1-\alpha}$ when $k \to \infty$, with $c>0$ and $\alpha \in (1,2)$, then Le~Gall \& Miermont~\cite{LGM11} proved a \emph{tightness} result: From every sequence of integers, one can extract a subsequence along which the map $\map^{\q}$ conditioned to have $n$ vertices, rescaled by a factor $n^{-1/(2\alpha)}$, converges in distribution towards a limit space which has almost surely Hausdorff dimension $2\alpha$. 
Both results were extended in~\cite{Mar18a} to the full domain of attraction of a stable law with index in $(1,2]$, and by conditioning with respect to either  the number vertices, or of edges, or of faces.

Recall from the introduction that we are interested in maps conditioned on having a large fixed number of vertices, edges, and faces at the same time.
More precisely, we consider a sequence $(K_{n})_{n\ge 1}$ of integers and, in order to discard degenerate cases, we shall always assume that both $K_{n}$ and $n-K_{n}$ tend to infinity. Let us denote by $\Map_{n,K_{n}}$ the set of all rooted bipartite planar maps with $n-1$ edges and $K_{n}+1$ vertices; by Euler's formula, all maps in $\Map_{n,K_{n}}$ have $n-K_{n}$ faces (hence the name \emph{biconditioned}). As before, a weight sequence $\q = (q_i ; i \ge 1)$ being given, we shall always implicitly assume that $n$ and $K_{n}$ are compatible with the support of $\q$, in the sense that $w^{\q}(\Map_{n,K_{n}}) > 0$; we also assume without further notice that $ \{i \geq 1: q_{i}>0\}$ is not included in a sublattice of $\Z$, i.e.~the largest $h>0$ such that there exists $a \in \R$ such that $\{i \geq 1: q_{i}>0\} \subset a+h\Z$ is $h=1$. The results carry through without such an aperiodicity condition after mild adaptations.
We then define a probability
\[\P^{\q}_{n,K_{n}}(M)=  \frac{w^{\q}(M)}{w^{\q}(\Map_{n,K_{n}})}, \qquad M \in \Map_{n,K_{n}},\]
so that $\P^{\q}_{n,k}$ is the law of a $\q$-Boltzmann planar map, biconditioned to have $n-1$ edges and $K_{n}+1$ vertices (and $n-K_{n}$ faces). Observe that $\Map_{n,K_{n}}$ is finite, so $w^{\q}(\Map_{n,K_{n}}) < \infty$ and we do not need to restrict to admissible weight sequences.
Note that if $q_{i} = 1$ for every $i \ge 1$, then $\P^{\q}_{n,K_{n}}$ is simply the uniform distribution on $\Map_{n,K_{n}}$; more generally, if $A$ is a subset of integers and  $q_{i}=\ind{i \in A}$, then $\P^{\q}_{n,K_{n}}$ is the uniform distribution on the subset of maps of $\Map_{n,K_{n}}$ having all face degrees in  $2A$.

In the same way stable processes are the only possible scaling limits of appropriately rescaled random walks, it is natural to expect that the Brownian map and the stable maps are the only possible scaling limits of appropriately rescaled size-conditioned Boltzmann planar maps. 
However, as we saw in the preceding sections, the biconditioning in a sense amounts to let the face degree distribution vary with the size, which allows more general scaling limits, in the same way L\'evy processes are scaling limits of appropriately rescaled random walks with varying jump distribution. In our case, the L\'evy processes that appear are just stable processes with a linear drift.

\subsection{Bijection with labelled trees}
\label{sec:bijection}

A powerful tool to prove scaling limits of random maps is a coding with labelled trees.
Recall the formal definition of (rooted plane) tree from Sec.~\ref{sec:arbres}. A \emph{labelling} of such tree $T$ is a function $\ell : V(T) \to \Z$ which equips each vertex with an integer, which we can view as a spatial position of the corresponding individual. We say that a labelled tree is a \emph{well-labelled tree} if it satisfies the following two constraints:
\begin{enumerate}
\item The root of the tree has label $0$;
\item For every internal vertex, say $u_0$, with $k \ge 1$ offspring, say $u_1$, \dots, $u_k$ from left to right, it holds that $\ell(u_i) \ge \ell(u_{i-1})-1$ for every $i = 1, \dots, k$, \emph{and} $\ell(u_k) = \ell(u_0)$.
\end{enumerate}
Observe that since the root has label $0$, the label of a vertex is the sum of the label increments between every two successive ancestors so the quantities $\Xi_{u_{0}} = (0, \ell(u_{1})-\ell(u_{0}), \dots, \ell(u_{k})-\ell(u_{0}))$ for each internal vertex $u_{0}$ entirely describe the labelling $\ell$; finally, the second property states that for every such vertex $u_{0}$, the vector $\Xi_{u_{0}}$ belongs to the following set of bridges:
\begin{equation}\label{eq:def_ponts_labels}
\mathscr{B}^{\ge-1}_{k} = \left\{(b_{i})_{0 \le i \le k} : b_{0} = 0 = b_{k} \text{ and } b_{i}-b_{i-1} \in \Z_{\ge-1} \text{ for } 1 \le i \le k\right\}.
\end{equation}

By combining the works~\cite{BDG04} and~\cite{JS15}, there is a (constructive) $2$-to-$1$ correspondence between well-labelled trees and bipartite planar maps with a distinguished vertex (in addition to the root-edge); the $2$-to-$1$ comes from the fact that it looses the orientation of the root-edge.
See Figure~\ref{fig:bijection_arbre_carte} for an example and~\cite[Sec.~2.3]{Mar18a} for details.
This correspondence enjoys the following properties:
\begin{enumerate}
\item The edges of the tree correspond to those of the map.
\item The internal vertices of the tree correspond to the faces of the map and the offspring number of an internal vertex is half the degree of the associated face.
\item The leaves of the tree correspond to the non-distinguished vertices of the map.
\item If one shifts all the labels so that the minimum is $1$ (one then recovers the original labels by shifting them so the root has label $0$), then the label of a leaf equals the distance in the map of the corresponding vertex to the distinguished one.
\end{enumerate}

Observe that by the properties (i) and (iii) above, the labelled tree associated with  a bipartite planar map having $n-1$ edges and $k+1$ vertices with one distinguished vertex has $n$ vertices and $k$ leaves.
The last property shows that controlling the labels on the tree is equivalent to controlling distances to the distinguished vertex in the map. It then remains to control all the pairwise distances, which is particularly challenging and was done for quadrangulations, by Le~Gall~\cite{LG13} and Miermont~\cite{Mie13}. Thanks to their work, going from the control of labels on the tree to that of all distances in the map, in the case of the Brownian map, is now rather simple. 
We will in fact not consider the labels here, but only the tree itself, and simply rely on general results from~\cite{Mar19}.

\begin{figure}[!ht] \centering
\includegraphics[width=.32\linewidth, page = 1]{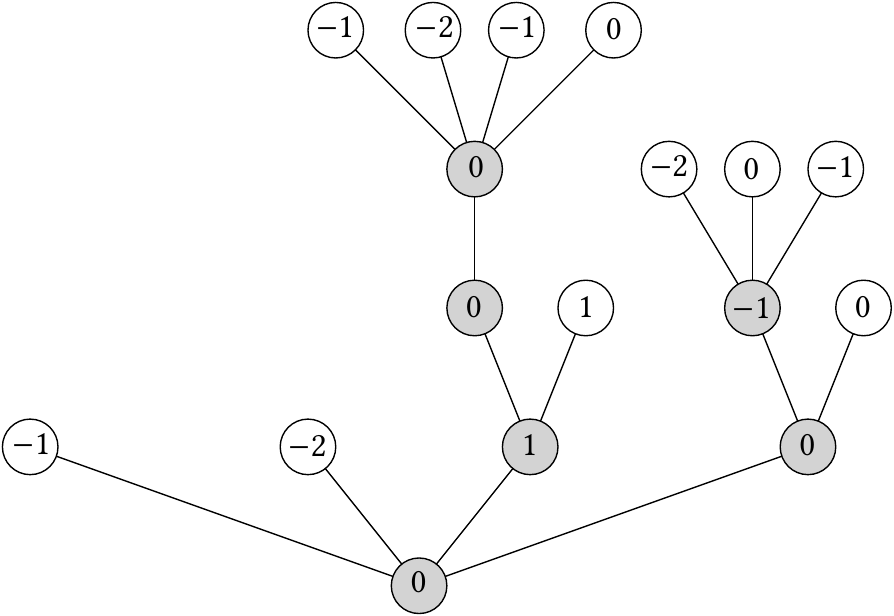}
\includegraphics[width=.33\linewidth, page = 2]{Bijection_carte_arbre}
\includegraphics[width=.33\linewidth, page = 3]{Bijection_carte_arbre}
\caption{A pointed map (right) associated with a labelled tree (left): Labels on the map indicate, up a to a shift, the graph distance to the distinguished vertex, which is the one carrying the smallest label (here $-3$). The figure in the middle indicates the construction of the map from the tree, first by applying a Schaeffer-like rule, except visiting vertices in depth-first search order, and then merging each internal vertex of the tree with its right-most offspring (represented by dashed lines).}
\label{fig:bijection_arbre_carte}
\end{figure}

Recall the law $\P^{\q}_{n,K_{n}}$ on maps with $n-1$ edges and $K_{n}+1$ vertices, induced by a sequence of weights $\q = (q_i ; i \ge 1)$. Property (ii) above allows us to determine the law of the associated random labelled tree.
In the context of the bijection from~\cite{BDG04} only, this was first discussed by Marckert \& Miermont~\cite[Proposition~7]{MM07}, and it was then adapted using bijection from~\cite{JS15} in~\cite{Mar18b}, see the proof of Proposition~11 in the latter reference.
Recall the set of bridges $\mathscr{B}_{k}$ as defined in~\eqref{eq:def_ponts_labels}; its cardinality is the binomial factor $\binom{2k-1}{k-1}$. Then define from $\q$ a weight sequence $\boldsymbol{\theta} = (\theta(i) ; i \ge 0)$ by
\begin{equation}\label{eq:poids_cartes_arbres}
\theta(0)=1 \qquad\text{and}\qquad \theta(i) = \binom{2i-1}{i-1} q_{i} \qquad\text{for every }i \ge 1.
\end{equation}
The preceding correspondence implies that the law of the labelled tree $(T, \ell)$ associated with a random map $M$ sampled from $\P^{\q}_{n,K_{n}}$, in which we distinguish a vertex chosen uniformly at random amongst the $K_{n}+1$ possibilities and independently of $M$, is as follows: First $T$ is a simply generated tree with $n$ vertices and $K_{n}$ leaves sampled from the weight sequence $\boldsymbol{\theta}$ as defined in Section~\ref{sec:arbres_aleatoires}, and then, conditionally given $T$, the labels are obtained by sampling independently for every internal vertex $u$ a uniform random bridge in $\mathscr{B}_{k_{u}}$ as defined in~\eqref{eq:def_ponts_labels}.

\subsection{Scaling limits}
\label{sec:results}

Our aim is to determine the asymptotic behaviour of a $\q$-Boltzmann random planar map $\map_{n,K_n}$ conditioned to have $n-1$ edges and $K_{n}+1$ vertices, sampled from $\P^{\q}_{n,K_n}$, and more precisely of the metric measured space obtained by endowing the set of vertices $V(\map_{n,K_n})$ with the graph distance $\dgr$, suitably rescaled, and the uniform probability measure $\pgr$, in the Gromov--Hausdorff--Prokhorov topology (see e.g.~\cite[Sec.~6]{Mie09} for details on this topology).

We need to introduce some notation. Set
\begin{equation}\label{eq:serie_gen_poids}
F(z)=1+\sum_{k=1}^{\infty}  \binom{2k-1}{k-1} q_{k}  z^{k}
\qquad\text{and}\qquad
A(z)=1- \frac{F(z)-F(0)}{zF'(z)}.
\end{equation}
Denote by $\rho$ the radius of convergence of   $F$ and set $k_{0}= \min \{k \ge  1:  q_{k} \neq 0\}$; then it is a simple matter to check that $A$ is increasing on $(0, \rho)$ and we extend it by continuity with $A(0) = 1-1/k_{0}$ and $A(\rho) < 1$ if and only if $F'(\rho) < \infty$.
These are the functions defined in~\eqref{eq:serie_gen_poids_arbres} with the sequence $\boldsymbol{\theta}$ from~\eqref{eq:poids_cartes_arbres}.

\begin{ex}
\label{ex:cartes}
It is instructive to keep in mind the example of uniform bipartite maps, which corresponds to the weight sequence $q_{i}=1$ for every $i \geq 1$. In this case, $k_{0}=1$, $\rho=1/4$, and
\[F(z)= \frac{1}{2} + \frac{1}{2\sqrt{1-4z}},
\qquad\text{and}\qquad
A(z)= \frac{6z-1+(1-4z)^{3/2}}{2z}.\]
One can easily apply the next theorem to deduce Theorem~\ref{thm:cartes_unif} appearing the introduction; for example, the scaling function $S$ in the latter is $S(x) = \frac{F'(A^{-1}(x))}{A^{-1}(x) F^{(2)}(A^{-1}(x))}$ in Theorem~\ref{thm:gencrit} below.
The assumptions of Theorem~\ref{thm:nongenneg} are also satisfied, with $\alpha = 1/2$.
\end{ex}

Extracting a subsequence if necessary, we may assume that $K_{n}/n$ converges as $n \to \infty$ to some limit $\tau \in [0,1]$.
Theorem~\ref{thm:CVluka} combined with~\cite{Mar19}  will readily entail  the following result.

\begin{thm}
\label{thm:CV_cartes}
Let $\map_{n,K_n}$ denote a $\q$-Boltzmann random planar map conditioned to have $n-1$ edges and $K_{n}+1$ vertices.
If $(\Bmap, \dBmap, \pBmap)$ denotes the Brownian map, then the convergence in distribution
\[\left(V(\map_{n,K_n}),  v_{n}^{-1/4}  \dgr, \pgr\right)  \cvloi(\Bmap, \dBmap, \pBmap)\]
holds for the Gromov--Hausdorff--Prokhorov topology
in each of the following cases:
\begin{thmlist}
\item\label{thm:gencrit} There exists $\tau \in  (1-1/k_{0},A(\rho))$ such that $K_{n}/n \to \tau$ and $\scale_{n} = \frac{4 b F^{(2)}(b) n}{9 F'(b)}$ with $b=A^{-1}(\tau)$.

\item\label{thm:gencrit0} $K_{n}/n \to 0$ with $K_{n} \to \infty$, $q_{1}, q_{2}>0$, and $\scale_{n} = 8 K_{n} / 9$.

\item\label{thm:nongenneg} $K_{n}/n \to 1$ and $n-K_{n} \to \infty$, $F$ is $\Delta$-analytic and there exist $\alpha, c>0$ such that $F(\rho-z) \sim c/z^{\alpha}$ as $z \to 0$ with $\mathrm{Re}(z)>0$,  and finally $\scale_{n} = \frac{4(1+\alpha)n^{2}}{9\alpha(n-K_{n})}$.
\end{thmlist}
\end{thm}

As for random paths, in the first regime, since $F^{(2)}(b) < \infty$, then essentially the degree of a typical face  in  $\map_{n,K_n}$ has uniformly bounded variance and the Brownian map is naturally expected in the limit.
In the second regime, we force the number of vertices to be small compared to the number of edges. Clearly, this requires $k_{0}=1$, i.e.~$q_{1} > 0$: in order to have much less vertices than edges, the map must have (many!) double edges, i.e.~faces with degree $2$. On the other hand we also need the extra assumption $q_{2} > 0$ to avoid periodicity issues which would require a more technical analysis. However, there is no additional regularity assumption on $F$.
Finally, in the last regime, when the number of vertices is asymptotically equivalent to the number of edges, we recall that $F$ is $ \Delta$-analytic if there exist $R>\rho$ and $\phi \in (0,\pi/2)$ such that $F$ is analytic on $ \{z : |z|<R, z \neq \rho, |\arg(z-1)|>\phi\}$, see~\cite[Definition~VI]{FS09}. This condition on $F$ is crucial and Theorem~\ref{thm:nongenstable_condensation} below shows a very different behaviour.

\smallskip
Our final result concerning random maps focuses on the case $K_{n}/n \to A(\rho)$ when the latter is finite. Here, we assume that the probability distribution with generating function $F(\rho \,\cdot)/F(\rho)$ is aperiodic and belongs to the domain of attraction of a stable law with index $\alpha \in (1, 2]$, i.e.~that there exist $\alpha \in (1,2]$, $\rho>0$, $m>0$, and a slowly varying function $L$ such that for every $0 \le  s <  \rho$,
\[F(s)=F(\rho) \left(1-m+\frac{m s}{\rho}+\left(1-\frac{s}{\rho}\right)^{\alpha} L\left(\frac{1}{1-s/\rho}\right)\right).\]
In the case $F^{(2)}(\rho) < \infty$, i.e.~when $\alpha=2$ and $L$ admits a finite limit at infinity, say $\ell$, the law with generating function $F(\rho \,\cdot)/F(\rho)$ has a finite variance, equal to
\[\sigma^{2} = 2\ell + m - m^{2}
= \frac{\rho^{2} F^{(2)}(\rho)}{F(\rho)} + \frac{\rho F'(\rho)}{F(\rho)} - \left(\frac{\rho F'(\rho)}{F(\rho)}\right)^{2}
.\]
In this case we set $r_{n} = \sqrt{n \sigma^{2}/2}$ for every $n\ge 1$, otherwise, we let $(r_{n})_{n \ge 1}$ be a sequence such that $r_{n}^{\alpha} \sim n L(r_{n})$ as $n \to \infty$.

\begin{ex}
Let us give an explicit example taken from~\cite{AMB16}. The weights
\[q_k = \frac{-\sqrt{\pi}}{2\Gamma(1-\alpha)} \left(\frac{1}{4\alpha}\right)^{k-1} \frac{\Gamma(-\alpha+k)}{\Gamma(\frac{1}{2}+k)}
\qquad\text{for}\enskip k \ge 2.
\]
give rise to the generation function $F$ satisfying $F(\alpha z)/F(\alpha) = z + \alpha^{-1} (1-z)^\alpha$. In this case, $\rho=\alpha$, $m=1$, $F(\alpha)=\alpha$, $F'(\rho)=1$, $F^{(2)}(\rho)=1/2$ when $\alpha=2$, $L$ is constant equal to $1$, and finally $A(\alpha) = 1/\alpha$. In particular, one can take $r_{n}=n^{1/\alpha}$.
\end{ex}

The next statements are our last applications to random maps.

\begin{thm}
\label{thm:nongenstable}
Under the previous assumptions, assume that $K_{n}$ takes the form
\[K_n = A(\rho) n + \lambda_{n},
\qquad\text{where}\qquad
\frac{\lambda_{n}}{r_{n}} \cv \lambda \in [-\infty,\infty].\]
The following convergences hold for the Gromov--Hausdorff--Prokhorov topology:
\begin{thmlist}
\item\label{thm:nongenstable_drift} 
If $\lambda \in \R$ and $\alpha = 2$, and
\[v_{n} = \frac{\rho F^{(2)}(\rho)}{F'(\rho)} \frac{4n}{9} \ind{F^{(2)}(\rho) < \infty}
+ \frac{F(\rho)}{\rho F'(\rho)}\frac{8 r_{n}^{2}}{9} \ind{F^{(2)}(\rho) = \infty}.\]
then
\[\left(V(\map_{n,K_n}), v_{n}^{-1/4} \dgr, \pgr\right)  \cvloi (\Bmap, \dBmap, \pBmap),\]
where $(\Bmap, \dBmap, \pBmap)$ is the Brownian map.

\item\label{thm:nongenstable_gaussien} If $\lambda = -\infty$ and $\lambda_{n}/n \to 0$, let $\varepsilon_{n} \rightarrow 0$ be defined by $A(\rho(1-\varepsilon_{n})) = K_{n}/n$, then for
\[v_{n} = \frac{\rho F^{(2)}(\rho)}{F'(\rho)} \frac{4n}{9} \ind{F^{(2)}(\rho) < \infty}
+ \frac{4}{9} \frac{\alpha-1}{1-A(\rho)} \frac{|\lambda_{n}|}{\varepsilon_{n}} \ind{F^{(2)}(\rho) = \infty},\]
we have
\[\left(V(\map_{n,K_n}), v_{n}^{-1/4} \dgr, \pgr\right)  \cvloi (\Bmap, \dBmap, \pBmap),\]
where $(\Bmap, \dBmap, \pBmap)$ is the Brownian map.

\item\label{thm:nongenstable_condensation} If $\lambda = \infty$, we further assume that $\binom{2k-1}{k-1} q_{k} \rho^{-k}$ is regularly varying at infinity with some index $-\beta < -2$. 
If $F^{(2)}(\rho)<\infty$ (which implies $\beta\ge3$), we assume furthermore that there exists $c>(\beta-3)/\sigma^{2}$ such that $\lambda_{n} \geq \sqrt{c \ln(n)}$. 
Then
\[\left(V(\map_{n,K_n}),  (2\lambda_{n})^{-1/2} \dgr, \pgr\right)  \cvloi (\CRT, \dCRT, \pCRT)\]
where $(\CRT, \dCRT, \pCRT)$ is Aldous' Brownian \emph{tree} coded by the standard Brownian excursion.
\end{thmlist}
\end{thm}

We shall establish these convergences by combining  Theorem~\ref{thm:CVluka_distorted} with~\cite{Mar19}. Our last result concerns the first regime, $\lambda \in \R$, when $\alpha \in (1,2)$. In this case, we expect new limit spaces to appear, and  in this direction here we establish a tightness result as a first step.

\begin{thm}\label{thm:skewed_stable}
In the same framework as the previous theorem, still assuming that $K_{n}$ takes the form
\[K_n = A(\rho) n + \lambda_{n},
\qquad\text{where}\qquad
\frac{\lambda_{n}}{r_{n}} \cv \lambda \in \R.\]
When $\alpha \in (1,2)$, the sequence
\[\left(V(\map_{n,K_n}),  \frac{1}{\sqrt{r_{n}}} \dgr, \pgr\right)_{n \ge 1}\]
is tight for the Gromov--Hausdorff--Prokhorov topology and nondegenerate (in the sense that it does not converge towards a deterministic limit).
\end{thm}

In the case of maps solely conditioned by their number of vertices, Le~Gall \& Miermont~\cite{LGM11} proved such a tightness result by showing that the \emph{label process} 
(which records the label of the vertices of the tree in depth-first search order) converges in distribution. This enabled them to show that all subsequential limits have the same Hausdorff dimension $2\alpha$ almost surely; furthermore some limit theorems then do not require the extraction of a subsequence, such as the convergence of the profile of distances to an independent uniform random vertex.

Here, Theorem~\ref{thm:skewed_stable} follows by a direct application of~\cite{Mar19} which only shows tightness of the label process, which does not provide any extra information on the limit spaces other than the scale of the typical distances.
In the companion paper~\cite{dimensions}, we extend Le~Gall \& Miermont's results which involved the so-called `continuous distance process' based on the excursion of stable L\'evy process to more general L\'evy processes with no negative jump.
Applied to the present setting of biconditioned maps, we obtain that all subsequential limits in Theorem~\ref{thm:skewed_stable} have the same Hausdorff, packing, and Minkowski dimensions, all equal to $2\alpha$ almost surely, independently of $\lambda \in \R$.
Finally, in the case $\lambda=0$, the limit spaces are the same as in~\cite{LGM11}.

\begin{rem}
It is interesting to note that in the case $m<1$, under an additional regularity assumption, the scaling limit of $\q$-Boltzmann planar maps conditioned on having a large number of edges is the Brownian CRT (see~\cite{JS15} and~\cite[Theorem~10(i)]{Mar19}), but an appropriate biconditioning also allows to escape this universality class.
\end{rem}

\begin{proof}[Proof of Theorems~\ref{thm:CV_cartes},~\ref{thm:nongenstable}, and~\ref{thm:skewed_stable}]
The paper~\cite{Mar19} specifically studies  a model of uniform random maps `with prescribed degrees' in the sense that given, say, $N$ positive integers $d_{N,1} \ge \cdots \ge d_{N,N}$,  a (rooted planar) map is sampled uniformly at random amongst all those with $N$ faces, with degrees $2d_{N,1}, \dots, 2d_{N,N}$. Then one wants to understand to behaviour of such a map $M_{N}$ as $N \to \infty$ in terms of the triangular array $(d_{N,i})_{N \ge i \ge 1}$. It is shown in~\cite[Theorem~1]{Mar19} that 
the sequence
\[\left(V(M_{N}),  \sigma_{N}^{-1/2} \dgr, \pgr\right)_{n \ge 1}
\qquad\text{where}\qquad
\sigma_{N}^{2} = \sum_{i=1}^{N} d_{N,i} (d_{N,i}-1),\]
is always tight for the Gromov--Hausdorff--Prokhorov topology. 
Moreover, by Theorem~2 and Theorem~3 respectively in~\cite{Mar19}, as $N \to \infty$,
\begin{itemize}
\item if $\sigma_{N}^{-1} d_{N,1} \to 0$, then it converges in distribution towards $\sqrt{2/3}$ times the Brownian map $(\Bmap, \dBmap, \pBmap)$;
\item if $\sigma_{N}^{-1} d_{N,1} \to 1$,
then it converges in distribution towards $\sqrt{2}$ times the Brownian tree $(\CRT, \dCRT, \pCRT)$.
\end{itemize}

Since two maps with same face degrees have the same $w^{\q}$-weight,  a size-conditioned Boltzmann map $\map_{n,K_n}$ sampled from $\P^{\q}_{n,K_{n}}$ can be seen as a mixture of this model with $N=n-K_{n}$: One first samples the degrees $d_{n-K_{n},1} \ge \cdots \ge d_{n-K_{n},n-K_{n}}$ randomly with an appropriate distribution, and then conditionally given this sequence, the map $\map_{n,K_n}$ has the uniform distribution on maps with $n-K_{n}$ faces, with degrees respectively $2d_{n-K_{n},1}, \dots, 2d_{n-K_{n},n-K_{n}}$.  
We then deduce that, once rescaled by a, now \emph{random}, factor, the sequence of maps $(\sigma_{n-K_{n}}^{-1/2} \map_{n,K_n})_{n \ge 1}$ is tight, and it suffices to prove the convergence in probability $d_{n-K_{n},1} / \sigma_{n-K_{n}} \to 0$ as $n \to \infty$ to deduce that they converge to $\sqrt{2/3}$ times the Brownian map, or that $d_{n-K_{n},1} / \sigma_{n-K_{n}} \to 1$ in probability to deduce that they converge to $\sqrt{2}$ times the Brownian tree. 

According to Sec.~\ref{sec:bijection}, the half face-degrees of $\map_{n,K_{n}}$ are given by the nonnegative increments plus one of the \L ukasiewicz path of the associated tree. Further, from Sec.~\ref{sec:arbres_aleatoires}, the latter has the law $\P^{\boldsymbol{\theta},+}_{n,K_{n}}$, where $\boldsymbol{\theta}$ is defined by~\eqref{eq:poids_cartes_arbres}. If we denote by $X^{n}_{k}$ the increments of this path, then the quantity $\sigma_{n-K_{n}}^{2}$ has therefore the same law as
\[\sum_{k=1}^{n} X^{n}_{k} (X^{n}_{k}+1) = -1 + \sum_{k=1}^{n} (X^{n}_{k})^{2}.\]
We infer from Theorems~\ref{thm:CVluka} and~\ref{thm:CVluka_distorted}, with also Remark~\ref{rem:var_finie_Luka} in the finite variance regimes that:
\begin{enumerate}
\item The random factor $\sigma_{n-K_{n}}$ can be replaced by the deterministic factor used in Theorems~\ref{thm:CV_cartes}, Theorem~\ref{thm:nongenstable}, and~\ref{thm:skewed_stable} respectively;

\item The ratio $d_{n-K_{n},1} / \sigma_{n-K_{n}}$ converges in probability to $0$ as $n \to \infty$ in Theorem~\ref{thm:CV_cartes} as well as in Theorem~\ref{thm:nongenstable_drift} and in Theorem~\ref{thm:nongenstable_gaussien}.

\item The ratio $d_{n-K_{n},1} / \sigma_{n-K_{n}}$ converges in probability to $1$ as $n \to \infty$ in Theorem~\ref{thm:nongenstable_condensation}.
\end{enumerate}
Our claim then follows from~\cite{Mar19}.
\end{proof}

\begin{appendix}
\section{On the sum of the increments squared}
\label{sec:preuves_LLT}

In this section, we provide the proof of the convergence of the sum of the squares of the increments of the simply generated random path $S^{n}$ conditioned to end at $S^{n}_{n} = x_{n}$, as stated in 
Theorems~\ref{thm:CVmarches},~\ref{thm:stable_marches_gaussien}, 
and finally~\ref{thm:stable_marches_drift} in the case $\alpha=2$. These results were needed when dealing with random maps just above.
Throughout this section, we shall denote by $X^{n}_{k} = S^{n}_{k}-S^{n}_{k-1}$ for every $1 \le k \le n$ the increments of the random bridge $S^{n}$.

Each regime requires a different argument and is studied in a separate subsection. 
Recall the notation $G$ and $\Psi$ from~\eqref{eq:serie_gen_poids_marches}.
In each regime studied from Sec.~\ref{ssec:gencrit} to Sec.~\ref{sec:regime_stable_moins_infini}, we have $x_{n}/n < \Psi(\rho)$ for every $n$ large enough so we can define as in the proof of Proposition~\ref{prop:LLT_implique_convergence_loi_varie} a random variable $\widehat{X}^{n}$ with generating function $\widehat{G}^{n} = G(b_{n} \,\cdot) / G(b_{n})$, where $b_{n}$ is such that $\E[\widehat{X}^{n}] = \Psi(b_{n}) = x_{n}/n$. We will use the explicit expressions of the first moments of $\widehat{X}^{n}$, which are given by~\eqref{eq:moments}.
We let further $(\widehat{S}^{n}_{i} ; i \ge 0)$ denote a random walk with such i.i.d. steps. 
On the other hand, in Sec.~\ref{sec:regime_stable_drift}, we have $x_{n}/n \ge \Psi(\rho)$, so the latter is finite and we now define $\widehat{X}$ similarly but with $b = \rho$ instead of $b_{n}$; then $(\widehat{S}_{i} ; i \ge 0)$ shall denote a random walk with such i.i.d. steps.
In both cases, Lemma~\ref{lem:tilting} shows that the paths $(\widehat{S}^{n}_{i} ; 0 \le i \le n)$ under $\P(\,\cdot\mid \widehat{S}^{n}_{n} = x_{n})$ and $(\widehat{S}_{i} ; 0 \le i \le n)$ under $\P(\,\cdot\mid \widehat{S}_{n} = x_{n})$ have the same law as our original path $S^{n}$.

\subsection{The bulk regime}
\label{ssec:gencrit}

In this first subsection, we focus on Theorem~\ref{thm:gencrit_marches}, i.e.~we assume that $x_{n}/n$ converges to some limit $\gamma \in (i_{\boldsymbol{\poids}}, \Psi(\rho))$. 
Let $b_{n}$ and $\widehat{X}^{n}$ as above, then $b_{n} \to b$ where $b \in (0, \rho)$ is such that $\Psi(b) = \gamma$; we let similarly $\widehat{X}^{b}$ have the generating function $\widehat{G}^{b}(s) = G(b s) / G(b)$. 
The next result completes Theorem~\ref{thm:gencrit_marches} and was used in the proof of Theorem~\ref{thm:gencrit}.

\begin{prop}
\label{prop:ctegeneric}
The following convergence holds in probability:
\[\frac{1}{n} \sum_{k=1}^{n} (X^{n}_{k})^{2} \cvproba \E\big[(\widehat{X}^{b})^{2}\big] = \frac{b^{2} G^{(2)}(b) + b G'(b)}{G(b)}
.\]
\end{prop}

\begin{proof}
According to Lemma~\ref{lem:tilting}, the law of the path $S^{n}$ and the random walk $\widehat{S}^{n}$ under $\P(\,\cdot\mid \widehat{S}^{n}_{n}=x_{n})$ are the same so we may replace the increments of the former by those of the latter. Let us therefore consider a sequence $(\widehat{X}^{n}_{k})_{k \ge 1}$ of i.i.d. random variables with generating function $\widehat{G}^{n} = G(b_{n} \,\cdot) / G(b_{n})$, and recall their second moment from~\eqref{eq:moments}, namely
\[s^{2}_{n} = \E\big[(\widehat{X}^{n}_{1})^{2}\big] 
= \frac{b_{n}^{2} G^{(2)}(b_{n}) + b_{n} G'(b_{n})}{G(b_{n})}.\]
Then $s^{2}_{n} \to s^{2} = \E[(\widehat{X}^{b})^{2}]$.

Fix $\varepsilon>0$; the proof by truncation of the weak law of large numbers provided e.g. in~\cite[p. 271]{Gut13}, shows the following: let $\Sigma_{n} = \sum_{k=1}^{n} ((\widehat{X}^{n}_{k})^{2} - s_{n}^{2})$ and $\Sigma_{n}' = \sum_{k=1}^{n} ((\widehat{X}^{n}_{k})^{2} - s_{n}^{2}) \ind{|(\widehat{X}^{n}_{k})^{2} - s_{n}^{2}| \le  n \varepsilon^{3}}$, then
\begin{align*}
\Pr{\left|\Sigma_{n}-\Es{\Sigma_{n}'}\right| > n \varepsilon} 
&\le \varepsilon\, \Es{\left|(\widehat{X}^{n}_{1})^{2} - s_{n}^{2}\right|} + n\, \Pr{\left|(\widehat{X}^{n}_{1})^{2} - s_{n}^{2}\right| > n \varepsilon^{3}}
\\
&\le 2\varepsilon s_{n}^{2} + \varepsilon^{-3}\, \Es{\left|(\widehat{X}^{n}_{1})^{2} - s_{n}^{2}\right| \ind{|(\widehat{X}^{n}_{1})^{2} - s_{n}^{2}| > n \varepsilon^{3}}},
\end{align*}
and
\[n^{-1} \left|\Es{\Sigma_{n}'}\right| \le \Es{\left|(\widehat{X}^{n}_{1})^{2} - s_{n}^{2}\right| \ind{|(\widehat{X}^{n}_{1})^{2} - s_{n}^{2}| > n \varepsilon^{3}}}.\]
Recall the fourth moment of $\widehat{X}^{n}_{1}$ from~\eqref{eq:moments}; since $b_{n} \to b \in (0,\rho)$, then $\E[(\widehat{X}^{n}_{1})^{4}]$ converges to the similar quantity with $b$ in place of $b_{n}$ (which equals $\E[(\widehat{X}^{b})^{4}]$). Since $s_{n}^{2} \to s^{2}$, we deduce that $((\widehat{X}^{n}_{1})^{2} - s_{n}^{2})_{n}$ is uniformly bounded in $L^{2}$ and therefore the right-hand side in the preceding display converges to $0$. Hence,
\[\lim_{n\to\infty} n^{-1} \Es{\Sigma_{n}'} = 0
\qquad\text{and}\qquad
\limsup_{n\to\infty} \Pr{\left|\Sigma_{n}-\Es{\Sigma_{n}'}\right| > n \varepsilon} 
\le 2\varepsilon s^{2}
.\]
Since this holds for every $\varepsilon>0$, we conclude that $n^{-1} \Sigma_{n}$ converges to $0$ in probability, so
\[\frac{1}{n} \sum_{k=1}^{n} (\widehat{X}_{k}^{n})^{2}  \cvproba  s^{2},\]
under the unconditional law. Then as in the proof of Proposition~\ref{prop:LLT_implique_convergence_loi_varie}, by absolute continuity, we can transfer this convergence to the first half of the bridge, i.e.~under $\P( \, \cdot \mid \widehat{S}^{n}_{n}=x_{n})$; by time-reversal the same holds also for the second half of the bridge. Since this bridge of $\widehat{S}^{n}$ has the same law as our path $S^{n}$, this shows our assertion.
\end{proof}

\subsection{The small endpoint regime}
\label{sec:LLT_gen_peu_feuilles}

We focus in this subsection on Theorem~\ref{thm:gencrit0_marches}: we only assume that $\poids(0), \poids(1) > 0$ and we suppose that $x_{n}/n \to 0$.
The next result completes Theorem~\ref{thm:gencrit0_marches} and  was used in the proof of Theorem~\ref{thm:gencrit0}.

\begin{prop}
\label{prop:ctegencritfewvertices}
The following convergence in probability holds:
\[\frac{1}{x_{n}} \sum_{k=1}^{n} (X^{n}_{k})^{2} \cvproba 1
.\]
\end{prop}

\begin{proof}
Exactly as in the proof of Proposition~\ref{prop:ctegeneric},
it is sufficient to work under the unconditional law of the random walk $\widehat{S}^{n}$, and then conclude by absolute continuity and time-reversal. We therefore replace the $X^{n}_{k}$'s by a sequence $(\widehat{X}^{n}_{k})_{k \ge 1}$ of i.i.d. random variables with generating function $\widehat{G}^{n} = G(b_{n} \,\cdot) / G(b_{n})$, where $b_{n} \in (0, \rho)$ is such that $\Psi(b_{n}) = x_{n} / n$, so $b_{n} \to 0$.
We then derive from~\eqref{eq:moments} the mean and variance of $(\widehat{X}^{n}_{1})^{2}$, and further, by~\eqref{eq:derivees_superieures_G},
\[\E\big[(\widehat{X}^{n}_{1})^{2}\big] = \frac{b_{n}^{2} G^{(2)}(b_{n}) + b_{n} G'(b_{n})}{G(b_{n})}
= \frac{x_{n}}{n} \left(\frac{b_{n}^{2} G^{(2)}(b_{n})}{b_{n} G'(b_{n})} + 1\right)
\equi \frac{x_{n}}{n} 
,\]
and similarly
\[\E\big[(\widehat{X}^{n}_{1})^{4}\big] 
= \frac{x_{n}}{n} \left(\frac{b_{n}^{4} G^{(4)}(b_{n}) + 6 b_{n}^{3} G^{(3)}(b_{n}) - 11 b_{n}^{2} G^{(2)}(b_{n})}{b_{n} G'(b_{n})} + 6\right),\]
so
\[\Var\big((\widehat{X}^{n}_{1})^{2}\big) \equi 6 \frac{x_{n}}{n}.\]
Since $x_{n} \to \infty$, we infer that $\frac{1}{x_{n}} \sum_{k=1}^{n} (X^{n}_{k})^{2}$ converges to $1$ in $L^{2}$ under the unconditioned probability, which suffices to conclude.
\end{proof}

\subsection{The large endpoint regime}
\label{sec:LLT_non_gen_beaucoup_feuilles}

In this subsection, we focus on Theorem~\ref{thm:nongenneg_marches}: We now assume that $x_{n}/n \to \infty$, that $G$ is $\Delta$-analytic, and that there exist $c, \rho, \alpha>0$ such that $G(\rho-z) \sim c z^{-\alpha}$ as $z \to 0$ with $\mathrm{Re}(z)>0$.
The next result completes Theorem~\ref{thm:nongenneg_marches} and was used in the proof of Theorem~\ref{thm:nongenneg}.

\begin{prop}
\label{prop:ctenongenericnegativeindex}
The following convergence in probability holds:
\[\frac{n}{x_{n}^{2}} \sum_{k=1}^{n} (X^{n}_{k})^{2} \cvproba \frac{\alpha+1}{\alpha}
.\]
\end{prop}

\begin{proof}
The proof is similar to that of Proposition~\ref{prop:ctegencritfewvertices}. First, 
it is sufficient to work under the unconditional law of the random walk $\widehat{S}^{n}$, and then conclude by absolute continuity and time-reversal that the same holds for the bridge, which has the same law as $S^{n}$. We therefore replace the $X^{n}_{k}$'s by a sequence $(\widehat{X}^{n}_{k})_{k \ge 1}$ of i.i.d. random variables with generating function $\widehat{G}^{n} = G(b_{n} \,\cdot) / G(b_{n})$.
Second, we rely on the moments of $(\widehat{X}^{n}_{1})^{2}$ calculated in~\eqref{eq:moments}.

Indeed, we already noticed in~\eqref{eq:moments2-new} that, as $n\to\infty$,
\[\E\big[(\widehat{X}^{n}_{1})^{2}\big] \sim \frac{\alpha+1}{\alpha} \frac{x_{n}^{2}}{n^{2}}.\]
Similarly, combining the expression of the moments from~\eqref{eq:moments} and the derivatives of $G$ from~\eqref{eq:Gderiv} we deduce that as $n\to\infty$,
\[\Var\big((\widehat{X}^{n}_{1})^{2}\big) 
= O\left(\frac{x_{n}^{4}}{n^{4}}\right)
.\]
Therefore $\frac{n}{x_{n}^{2}} \sum_{k=1}^{n} (X^{n}_{k})^{2}$ converges to $(\alpha+1)/\alpha$ in $L^{2}$ under the unconditioned probability, which suffices to conclude.
\end{proof}

\subsection{Stable regime with a large negative deviation}
\label{sec:regime_stable_moins_infini}

We treat in this subsection the regime described in Theorem~\ref{thm:stable_marches_gaussien}, where the limit process is again a Brownian bridge.
Recall the assumptions of this theorem: we assume throughout that $G'(\rho) < \infty$ and that there exist $\alpha \in (1,2]$ and a slowly varying function $L$ such that for every $0 \le  s <  \rho$,
\[G(s)=G(\rho) \left(1-m+\frac{m s}{\rho}+\left(1-\frac{s}{\rho}\right)^{\alpha} L\left(\frac{1}{1-s/\rho}\right)\right).\]
Recall that in the case $G^{(2)}(\rho) < \infty$, i.e.~when $\alpha=2$ and $L$ admits a finite limit at infinity, say $\ell$, the law with generating function $G(\rho \,\cdot)/G(\rho)$ has a finite variance, equal to
\[\sigma^{2} = 2\ell + m - m^{2}
.\]
In this case we set $r_{n} = \sqrt{n \sigma^{2}/2}$ for every $n\ge 1$, otherwise, we let $(r_{n})_{n \ge 1}$ be a sequence such that $r_{n}^{\alpha} \sim n L(r_{n})$ as $n \to \infty$.
Finally, we let
\[x_n = m n + \lambda_{n},
\qquad\text{where}\qquad
\lambda_{n}/r_{n} \to -\infty
\qquad\text{and}\qquad
 \lambda_{n}/n \to 0,\]
and we write $b_{n} = \rho(1-\varepsilon_{n})$, which satisfies $\Psi(b_{n}) = x_{n}/n$, so $\varepsilon_{n} \to 0$.

Again, we prove a law of large number for the square of the increments; this completes Theorem~\ref{thm:stable_marches_gaussien} and was used in the proof of Theorem~\ref{thm:nongenstable} in the case $\lambda = -\infty$.

\begin{prop}
\label{prop:cte_nongenstable_lambda_moins_infini}
The following convergence in probability holds:
\[\frac{\varepsilon_{n}}{|\lambda_{n}|} \sum_{k=1}^{n} (X^{n}_{k})^{2} \cvproba \alpha-1 + \frac{m^{2}}{\sigma^{2}} \ind{G^{(2)}(\rho)<\infty}.\]
\end{prop}

\begin{proof}
Again, we proceed exactly as in the proof of Proposition~\ref{prop:ctegencritfewvertices}: First, 
it is sufficient to work under the unconditional law of the random walk $\widehat{S}^{n}$, and then conclude by absolute continuity and time-reversal that the same holds for the bridge, which has the same law as $S^{n}$. We therefore replace the $X^{n}_{k}$'s by a sequence $(\widehat{X}^{n}_{k})_{k \ge 1}$ of i.i.d. copies of a random variable $\widehat{X}^{n}$ with generating function $\widehat{G}^{n} = G(b_{n} \,\cdot) / G(b_{n})$.
Second, we rely on the moments of $(\widehat{X}^{n})^{2}$ calculated in~\eqref{eq:moments}.
Note that the generating function $\widehat{G}^{n}$ satisfies the same assumption that $G$ in Sec.~\ref{sec:LLT_stable} and $\widehat{X}^{n}$ has the law of $\xi^{(b_{n})}$ there.

Let us first treat the case where either $\alpha < 2$ or $L$ converges to $\infty$. In this case, combining~\eqref{eq:equiv_scaling_distorted-new} and~\eqref{eq:moment_3_scaling_distorted-new} we obtain
\[\frac{n \varepsilon_{n}}{|\lambda_{n}|} \E\big[(\widehat{X}^{n}_{1})^{2}\big]
\cv \alpha-1.\]
By the same argument, relying on the expression of the moments from~\eqref{eq:moments} and the derivatives of $G$ from~\eqref{eq:derivees_G_distorted} we infer that as $n\to\infty$,
\[\frac{n \varepsilon_{n}}{|\lambda_{n}|} \E\big[(\widehat{X}^{n}_{1})^{4}\big] 
\sim \frac{n \varepsilon_{n}}{|\lambda_{n}|} \frac{b_{n}^{4} G^{(4)}(b_{n})}{G(b_{n})}
\sim \frac{(\alpha-1)(\alpha-2)(\alpha-3)}{\varepsilon_{n}^{2}}
.\]
Recall also that we deduced just below~\eqref{eq:equiv_scaling_distorted-new} that $r_{n} \varepsilon_{n} \to \infty$, in addition to $|\lambda_{n}| \to \infty$, hence
\[n \left(\frac{\varepsilon_{n}}{|\lambda_{n}|}\right)^{2} \E\big[(\widehat{X}^{n}_{1})^{4}\big] \cv 0.\]
We conclude in this case that $\frac{\varepsilon_{n}}{|\lambda_{n}|} \sum_{k=1}^{n} (\widehat{X}^{n}_{k})^{2}$ converges to $\alpha-1$ in $L^{2}$ under the unconditioned probability, which suffices for our claim.

It remains to treat the case where $\alpha=2$ and $L$ admits a finite and positive limit, say $\ell$. Here we use~\eqref{eq:varfini} instead.
Moreover in this case $G^{(2)}(\rho) < \infty$; let $\widehat{X}$ have the law with generating function $G(\rho\,\cdot)/G(\rho)$, so $\E[\widehat{X}] = m$ and let $s^{2} = \E[(\widehat{X})^{2}]$, which is given by~\eqref{eq:moments}, and notice that since $2\ell = \E[\widehat{X}(\widehat{X}-1)]$, then $s^{2} = 2\ell + m$.
Similarly let $s^{2}_{n} = \E[(\widehat{X}^{n}_{1})^{2}]$, then by continuity $s^{2}_{n} \to s^{2}$. 
The claim can be proved very much as in the proof of Proposition~\ref{prop:ctegeneric}. Recall from there that all we need to prove is that for any $\varepsilon > 0$,
\[\E\Big[\big|(\widehat{X}^{n}_{1})^{2} - s_{n}^{2}\big| \ind{|(\widehat{X}^{n}_{1})^{2} - s_{n}^{2}| > n \varepsilon^{3}}\Big] \cv 0.\]
However here we cannot rely on a uniform bound for the fourth moment, and we argue instead by absolute continuity.
Indeed, for every $k \ge0$, we have
\[\P(\widehat{X} = k) = \frac{\rho^{k} \poids(k)}{G(\rho)}
\qquad\text{and}\qquad
\P(\widehat{X}^{n}_{1} = k) = \frac{b_{n}^{k} \poids(k)}{G(b_{n})}.\]
Consequently, since $b_{n} \le \rho$ for every $n$ and $G(b_{n}) \to G(\rho)$, then
\begin{align*}
\E\Big[\big|(\widehat{X}^{n}_{1})^{2} - s_{n}^{2}\big| \ind{|(\widehat{X}^{n}_{1})^{2} - s_{n}^{2}| > n \varepsilon^{3}}\Big]
&= \frac{G(\rho)}{G(b_{n})} \E\bigg[\big|(\widehat{X})^{2} - s_{n}^{2}\big| \ind{|(\widehat{X})^{2} - s_{n}^{2}| > n \varepsilon^{3}} \left(\frac{b_{n}}{\rho}\right)^{k}\bigg]
\\
&\le \frac{G(\rho)}{G(b_{n})} \E\Big[\big|(\widehat{X})^{2} - s_{n}^{2}\big| \ind{|(\widehat{X})^{2} - s_{n}^{2}| > n \varepsilon^{3}}\Big],
\end{align*}
which tends to $0$. We conclude as in the proof of Proposition~\ref{prop:ctegeneric} that $\frac{1}{n} \sum_{k=1}^{n} (\widehat{X}^{n}_{k})^{2}$ converges to $s^{2}$ in probability. 
Our claim then follows from~\eqref{eq:varfini} together with the fact that the constant there equals $s^{2} - m^{2}$.
\end{proof}

Notice that when $L$ has a finite limit at infinity, say $\ell$, then $|\lambda_{n}| / \varepsilon_{n} \sim (2\ell + m - m^{2}) n = (s^{2} - m^{2}) n$ by~\eqref{eq:varfini} 
so $n^{-1} \sum_{k=1}^{n} (X^{n}_{k})^{2}$ converges in probability to $s^{2} = \E[\widehat{X}^{2}]$, which is consistent with Proposition~\ref{prop:ctegeneric}.

\subsection{Gaussian regime with a small deviation}
\label{sec:regime_stable_drift}

We finally treat in this subsection the regime described in Theorem~\ref{thm:stable_marches_drift}, in the case $\alpha=2$, and one last time, we prove a law of large number for the square of the increments; this was used in the proof of Theorem~\ref{thm:nongenstable_drift} in the case $\alpha=2$.
We work under the same assumptions on $G$ as in the previous subsection with $\alpha=2$, that is
\[G(s)=G(\rho) \bigg(1-m+\frac{m s}{\rho}+\left(1-\frac{s}{\rho}\right)^{2} L\left(\frac{1}{1-s/\rho}\right)\bigg),\]
with $r_{n}$ defined as in Sec.~\ref{sec:regime_stable_moins_infini}.
However here we assume that there exists $\lambda \in \R$ such that
\[x_n = m n + \lambda_{n},
\qquad\text{where}\qquad
\lambda_{n}/r_{n} \to \lambda.\]
As opposed to the preceding regimes, we shall use a fixed distribution for the equivalent random walk. Precisely, we let $\widehat{X}$ have the law with generating function $\widehat{G} = G(\rho \,\cdot)/G(\rho)$, so its moments are given by~\eqref{eq:moments} with $b = \rho$. We let $\widehat{S}$ denote a random walk with such i.i.d. steps, so Lemma~\ref{lem:tilting} shows that the path $(\widehat{S}_{i} ; 0 \le i \ge n)$ under $\P(\,\cdot\mid \widehat{S}_{n} = x_{n})$ has the same law as our original path $S^{n}$.

\begin{prop}
\label{prop:ctealpha2}
When $\alpha=2$, the following convergence in probability holds:
\[r_{n}^{-2} \sum_{k=1}^{n} (X^{n}_{k})^{2} \cvproba 
2 + \frac{2m^{2}}{\sigma^{2}} \ind{G^{(2)}(\rho) < \infty}
.\]
\end{prop}

\begin{proof}
As now usual, 
it is sufficient to work under the unconditional law of the random walk $\widehat{S}$, and then conclude by absolute continuity and time-reversal that the same holds for the bridge, which has the same law as $S^{n}$. We therefore replace the $X^{n}_{k}$'s by a sequence $(\widehat{X}_{k})_{k \ge 1}$ of i.i.d. random variables with generating function $\widehat{G}$.
The latter belong to the domain of attraction of a Gaussian distribution;
in the case where $L$ has a finite limit, we have $\widehat{G}^{(2)}(\rho) < \infty$ so the $(\widehat{X}_{k})^{2}$'s have finite mean and the claim simply follows from the law of large numbers combined with the fact that $r_{n} = \sqrt{n \sigma^{2}/2}$ :
\[n^{-1} \sum_{k=1}^{n} (X^{n}_{k})^{2}\cvproba  \E[\widehat{X}^{2}]=\sigma^{2}+m^{2}\]
In particular, the convergence of Proposition~\ref{prop:ctegeneric} also holds for $b=\rho$ when $G^{(2)}(\rho) < \infty$.

This can be extended to the case  $\widehat{G}^{(2)}(\rho) =\infty$ by looking at the first two truncated moments, as in e.g.~\cite[Theorem~8(i)]{Mar19}. We reproduce the argument for the reader's convenience. Let $L_{1}$ be the slowly varying function defined by $L_{1}(x) = \E[\widehat{X}_{1}^{2} \ind{|\widehat{X}_{1}| \le x}]$ for every $x > 0$.
As in e.g.~\cite[Lemma~4.7]{BS15} we have $(L_{1}(x)-m) \sim 2 L(x)$ as $x \to \infty$. Further, by~\cite[Chapter~XVII, Equation~5.16]{Fel71}, the ratio $x^2 \P(\widehat{X}_{1} \ge x) / L_{1}(x)$ converges to $(2-\alpha)/2 = 0$ as $x \to \infty$.
Then for every $\varepsilon > 0$, since $L_1$ is slowly varying, then $L_1(\varepsilon r_n) \sim L_1(r_n)$ and so
\[\Pr{r_n^{-1} \max_{1 \le i \le n} \widehat{X}_i \ge \varepsilon}
\le n \Pr{\overline{X}_{1} \ge \varepsilon r_n}
= o\left(n (\varepsilon r_n)^{-2} L_{1}(\varepsilon r_n)\right)
= o\left(n r_n^{-2} L_{1}(r_n)\right)
= o(1),\]
where we used that $n r_n^{-2} \sim L(r_n)^{-1}$ and that the ratio $L_{1}(r_n)/L(r_n)$ is bounded.
Further,
\[\E\bigg[r_{n}^{-2} \sum_{1 \le i \le n} \widehat{X}_{i}^2 \ind{|\widehat{X}_{i}| \le \varepsilon r_{n}}\bigg]
= n r_{n}^{-2} \Es{\widehat{X}_{1}^2 \ind{|\widehat{X}_{1}| \le \varepsilon r_{n}}}
\sim n r_{n}^{-2} L_{1}(r_{n})
\sim \frac{L_{1}(r_{n})}{L(r_{n})},\]
and, similarly, the variance of this quantity equals
\[n r_{n}^{-4} \Var\left(\widehat{X}_{1}^2 \ind{|\widehat{X}_{1}| \le \varepsilon r_{n}}\right)
\le \varepsilon^2 n r_{n}^{-2} \Es{\widehat{X}_{1}^2 \ind{|\widehat{X}_{1}| \le \varepsilon r_{n}}}
\sim \varepsilon^2 \frac{L_{1}(r_{n})}{L(r_{n})}.\]
Since we have previously shown that with high probability, $|\widehat{X}_{i}| \le \varepsilon r_n$ for every $i \le n$, then we conclude that $r_n^{-2} \sum_{1 \le i \le n} \widehat{X}_{i}^2$ converges in probability to the limit of the ratio $L_{1}/L$, which is $2$. The completes the proof.
\end{proof}
\end{appendix}

\end{document}